\documentclass[a4paper,12pt,reqno,oneside]{amsart} % amsart loads amsmath, amsthm and amsfonts

\usepackage{setspace} 
\usepackage{typearea} % European margins
\usepackage{amscd,amsmath,amssymb,verbatim} % amssymb loads amsfonts

\usepackage[dvipsnames]{xcolor}
\usepackage{xr-hyper}
\usepackage{cite}
\usepackage[colorlinks,backref,final,bookmarksnumbered,bookmarks]{hyperref}
\usepackage[backrefs]{amsrefs}
\hypersetup{linkcolor=bluex,citecolor=Green,filecolor=cyan,urlcolor=magenta}
\usepackage{graphicx}
\definecolor{bluex}{Hsb}{220, 1, 1}

\newcommand{\borderref}[2]{{\hypersetup{linkcolor=black}\hyperref[#1]{#1 #2}}}
\usepackage{ifthen}
\newcommand{\arxiv}[2][]{\ifthenelse{\equal{#1}{}}
{\href{http://arxiv.org/abs/#2}{\tt arXiv:#2}}
{\href{http://arxiv.org/abs/math/#2}{\tt arXiv:math.#1/#2}}}

\theoremstyle{plain}
\newtheorem{theorem}{Theorem}[section]
\newtheorem{mainthm}{Theorem}

\newtheorem{lemma}[theorem]{Lemma}
\newtheorem{proposition}[theorem]{Proposition}
\newtheorem{corollary}[theorem]{Corollary}
\newtheorem*{corollary*}{Corollary}
\newtheorem{problem}[theorem]{Problem}

\newtheorem{conjecture}[theorem]{Conjecture}
\newtheoremstyle{addendum}% name of the style to be used
  {}% measure of space to leave above the theorem. E.g.: 3pt
  {}% measure of space to leave below the theorem. E.g.: 3pt
  {\itshape}% name of font to use in the body of the theorem
  {}% measure of space to indent
  {\bfseries}% name of head font
  {.}% punctuation between head and body
  {.5em}% space after theorem head; " " = normal interword space
  {\thmname{Addendum #2 to #1}{\thmnote{ #3}}}
\theoremstyle{addendum}
\newtheorem{addendumtotheorem}{Theorem}

\theoremstyle{definition}
\newtheorem{remark}[theorem]{Remark}
\newtheorem*{remark*}{Remark}
\newtheorem{example}[theorem]{Example}

\def\x{\times}
\def\but{\setminus}

\def\eps{\varepsilon}
\def\phi{\varphi}
\def\emptyset{\varnothing}
\renewcommand{\:}{\colon}
\def\A{\mathcal{A}}
\def\B{\mathcal{B}}

\def\R{\mathbb{R}}
\def\Q{\mathbb{Q}}
\def\Z{\mathbb{Z}}
\def\K{\mathcal{K}}

\def\xr#1{\xrightarrow{#1}} 

\def\<{\hspace{-2.5pt}<\hspace{-2.5pt}}
\def\fr{{\text{\rm fr}}}

\DeclareMathOperator{\lk}{lk}

\DeclareMathOperator{\Flux}{Flux}
\DeclareMathOperator{\sgn}{sgn}
\DeclareMathOperator{\Supp}{Supp}

\DeclareMathOperator{\curl}{curl}
\DeclareMathOperator{\Vol}{vol}

\begin{document}
\title{Satellites and invariants of links}
\author{Sergey A. Melikhov}
\address{Steklov Mathematical Institute of Russian Academy of Sciences, Moscow, Russia}
\email{melikhov@mi-ras.ru}
%\subjclass{Primary: 57Q35, secondary: 57N35}
%\date{\today}

\begin{abstract} An invariant $v$ of $m$-component links is called {\it cableable} if there exists a $k$ such that 
whenever a link $L'$ is obtained from a link $L=(K_1,\dots,K_m)$ by replacing each knot $K_i$ with its 
$(p_i,q_i)$-cable for some $p_i$ and $q_i$, we have $v(L')=(p_1\cdots p_m)^kv(L)$.

The following problem is implicit in a number of papers by P. M. Akhmetiev and originates from 
the Arnold--Moffatt program for finding topological lower bounds for the energy of a magnetic field: 
Does there exist a cableable finite type invariant of links in $S^3$ which is not a function 
of the pairwise linking numbers?
According to Akhmetiev, he solved this problem himself, with the desired invariant being defined 
for $3$-component links as an analytic expression involving a magnetic field modeled on the given link.
We offer a different solution, with the desired invariant for $m$-component links, $m\ge 3$, being 
a function of the coefficients of the Conway polynomial of the link and its sublinks; the function is 
given explicitly when $m=3$.
Moreover, we show that the cables can be replaced by arbitrary satellites.
Much of the proof is a study of low degree coefficients of the Conway potential function 
$\Omega_L(x_1,\dots,x_n)$ expanded as a formal power series in Conway's variables $z_i=x_i-x_i^{-1}$.

We also discuss type $n$ invariants which are ``cableable up to an invariant of type $n-1$'', some cableable 
invariants which are not of finite type (particularly a certain modification of Milnor's $\bar\mu$-invariants), 
and applications to links of solenoids.
\end{abstract}
\maketitle

\section{Introduction}\label{intro}

In the present paper by a ``link''/``knot'' we mean a smooth (or PL) link/knot in $S^3$.

\subsection{Cableable, braidable and satellitable invariants}
Given an $m$-component link $L=(K_1,\dots,K_m)$, a link $L'=(K_1',\dots,K_m')$ will be called a {\it satellite}, or more specifically
a {\it $(p_1,\dots,p_m)$-satellite} of $L$ if there exists a tubular neighborhood $N=(T_1,\dots,T_m)$ of $L$ such that each $K_i'$ lies 
in the solid torus $T_i$ and $[K_i']=p_i[K_i]\in H_1(T_i)$.
If moreover each projection $T_i\cong K_i\x D^2\to K_i$ restricts to a covering map $K_i'\to K_i$, then we call $L'$ 
a {\it braiding}, or more specifically a {\it $(p_1,\dots,p_m)$-braiding} of $L$.%
\footnote{Let us note that each $p_i$ must be nonzero here.
When each $p_i>0$, R. F. Williams \cite{Wi} calls $L'$ a ``generalized cabling'' of $L$.}
If in addition to that each $K'_i$ lies in $\partial T_i$, then we call $L'$ a {\it cabling}, or more specifically 
a {\it $(p_1,\dots,p_m)$-cabling} of $L$.
In this case each $K_i'$ must be the {\it $(p_i,q_i)$-cable} of $K_i$ for some $q_i$ coprime to $p_i$, where $q_i=\lk(K_i',K_i)$.

Let $k$ be a non-negative integer. 
An abelian group-valued invariant $v$ of $m$-component links will be called {\it $k$-cableable} (resp.\ {\it $k$-braidable}/{\it $k$-solenoidal}), 
if for all $p_1,\dots,p_m\in\Z\but\{0\}$, for every $m$-component link $L$ and every $(p_1,\dots,p_m)$-cabling (resp.\ 
$(p_1,\dots,p_m)$-braiding/$(p_1,\dots,p_m)$-satellite) $L'$ of $L$ 
\[v(L')=(p_1\cdots p_m)^k v(L).\]
When $k\ne 0$, we call $v$ {\it $k$-satellitable} if it satisfies the definition of a $k$-solenoidal invariant for arbitrary 
$p_1,\dots,p_m\in\Z$ (including $p_i=0$).
On the other hand, the definition of $0$-cableable, $0$-braidable and $0$-solenoidal invariants also applies to set-valued invariants.
We call $v$ {\it cableable}%
\footnote{P. M. Akhmetiev's term for cableable invariants is ``asymptotic'' \cite{A20}*{Assertion B.1}, \cite{A21}*{Theorem 14.2}
but he previously used this term in a different sense  \cite{A11}*{Definition 7}, \cite{A14}*{Definition 4.1}, \cite{A16}*{\S2.1}.}
(resp.\ {\it braidable}/{\it solenoidal}/{\it satellitable}) if it is $k$-cableable (resp.\ $k$-braidable/$k$-solenoidal/$k$-satellitable) 
for some $k$.
Clearly,
\[\text{ \large satellitable }\Rightarrow\text{ \large solenoidal }\Rightarrow\text{ \large braidable }\Rightarrow\text{ \large cableable}.\]

Let us note that every solenoidal invariant is an invariant of {\it F-isotopy}, that is, the equivalence relation on links 
generated by ambient isotopy and the operation of replacing a given link with any of its $(1,\dots,1)$-satellites.
F-isotopy was introduced by R.~ Fox and is named so after him (see \cite{HS}, \cite{Sm1}, \cite{KY}).
The literature on F-isotopy includes \cite{HS}, \cite{Sm1}, \cite{Sm2}, \cite{Lau}*{\S3}, \cite{Gu}*{\S2}, \cite{Ro2}*{\S4}, \cite{Gi2}, 
\cite{KY}, \cite{Tr2}*{\S3}, \cite{Hi}*{\S1.5, \S8.6}, \cite{An}, \cite{M24-3}*{Appendix}.
 
\begin{example}
Since every knot is F-isotopic to an unknot, there are no non-constant solenoidal invariants of knots.
\end{example}

\begin{example} Clearly, the linking number of two-component links is $1$-satellitable.
\end{example}

\begin{example}\label{lk} Let $l_{ij}$ denote the pairwise linking numbers $\lk(K_i,K_j)$ for an $m$-component link $L=(K_1,\dots,K_m)$, $m\ge 2$.
Then 

(a) $\lambda(L):=\prod_{i<j}l_{ij}$ is $(m-1)$-satellitable.

(b) $\lambda^\circ(L):=l_{12}l_{23}\cdots l_{m-1,m}l_{m1}$ is $2$-satellitable;

(c) if $m$ is even, then $l_{12}l_{34}\cdots l_{m-1,m}$ is $1$-satellitable.
\end{example}

\begin{example} \label{mumu}
(a) Milnor's $\bar\mu$-invariants \cite{Mi2} other than the linking number have no chances of being even cableable 
as they do not take values in a fixed abelian group --- specifically, each $\bar\mu_{i_1\dots i_n}(L)$ is a residue class modulo 
a certain integer $\delta_{i_1\dots i_n}(L)$ (we recall the definition in \S\ref{milnor}).
However, let $\bar{\bar\mu}_{i_1\dots i_n}(L)$ be the invariant of links defined by
\[\bar{\bar\mu}_{i_1\dots i_n}(L)=\begin{cases}
\bar\mu_{i_1\dots i_n}(L)&\text{if }\delta_{i_1\dots i_n}(L)=0,\\
0 &\text{otherwise.}
\end{cases}\]
Then each $\bar{\bar\mu}_{i_1\dots i_n}$ takes values in $\Z$, and it is not hard to show (see Corollary \ref{mu-satellite'}) that 
each $\bar{\bar\mu}_{i_1\dots i_n}$ with precisely $k$ occurrences of each index is $k$-solenoidal. 
(For example, $\bar{\bar\mu}_{1122}$ is $2$-solenoidal.
An alternative proof of this special case is given in Corollary \ref{beta-cable}(b).)
Moreover, it follows that each $\bar{\bar\mu}_{i_1\dots i_n}$ with pairwise distinct indices is $1$-satellitable (using that it is 
an invariant of link homotopy, cf.\ Remark \ref{mu-modified}(b)).
However, $\bar{\bar\mu}_{1122}$ is not $2$-satellitable (see Example \ref{wh-hopf}).
\smallskip

(b) Let $\bar{\bar{\bar\mu}}_{i_1\dots i_n}(L)$ be the invariant of links defined by
\[\bar{\bar{\bar\mu}}_{i_1\dots i_n}(L)=\begin{cases}
1&\text{if }\bar{\bar\mu}_{i_1\dots i_n}(L)\ne 0,\\
0 &\text{otherwise.}
\end{cases}\]
Then it similarly follows (see Corollary \ref{mu-satellite'}) that each $\bar{\bar{\bar\mu}}_{i_1\dots i_n}$ is $0$-solenoidal.
Let us note that each $\bar{\bar{\bar\mu}}_{i_1\dots i_n}$ with pairwise distinct indices is also an invariant of link homotopy.
\end{example}
 
\begin{example} The braid index of links \cite{Wi} and the bridge number of knots \cite{Sch}*{p.~283, Satz 9} (see also \cite{Su}) and links \cite{Wi} 
{\it almost} satisfy the definition of a $1$-braidable invariant --- namely, they do so for links $L$ with no unknotted components.
\end{example}
 
\begin{example} \label{kojima-yamasaki} 
Given a $2$-component link $L=(K,Q)$ with $\lk(L)=0$, Kojima and Yamasaki observed that a certain Laurent polynomial
$\lambda(L)$, which is well-defined up to multiplication by the units of $\Z[t^{\pm1}]$,%
\footnote{Namely, let $\tilde K$ be a lift of $K$ in the infinite cyclic cover $\tilde X$ of $X:=S^3\but Q$.
It is shown in \cite{KY} that the annihilator ideal of $[\tilde K]\in H_1(\tilde X)$ in $\Z[t^{\pm1}]$ is always principal 
(using that $\Q[t^{\pm1}]$ is a PID and that $H_1(\tilde X)$ is $\Z$-torsion free \cite{Cr1}, \cite{KY}*{proof of Proposition 6}).
Now $\lambda(L)$ is any generator of this ideal.}
is an invariant of F-isotopy \cite{KY}*{\S6, \S7}.
In fact their proof works to show that the equivalence class of $\lambda(L)$ up to substituting $t$ with $t^q$, $q\ne 0$, and multiplying 
by the units of $\Z[t^{\pm1}]$ is a $0$-solenoidal invariant.%
\footnote{In more detail, arguing as in \cite{KY}*{Proof of Theorem 4(1)} and using again that $H_1(\tilde X)$ is $\Z$-torsion free,
one gets that whenever $L$ is replaced by any of its $(p,q)$-satellites, $\lambda(L)$ changes by substituting $t$ with $t^q$, for any 
$p,q\ne 0$ (compare \cite{KY}*{Remark 5}).}
This construction does not easily extend to the case of three or more components.%
\footnote{$H_1(\tilde X)$ need not be $\Z$-torsion free when $Q$ is a $2$-component link and $\tilde X$ is the universal abelian cover 
of $X:=S^3\but Q$ \cite{Hi1}*{pp.\ 47--48}.}
\end{example}

\begin{remark}\label{solenoids-remark}
Let us discuss how solenoidal invariants relate to links of solenoids.
By a {\it solenoid} we mean the limit of an inverse sequence of the form $\dots\xr{\pi_1}S^1\xr{\pi_0}S^1$, where each $\pi_i$ is a covering.
(Thus we formally regard $S^1$ as a solenoid.)

A {\it knotted solenoid} is an embedding of a solenoid in $S^3$.
The image of every knotted solenoid is the intersection of a nested sequence of handlebodies (since it is tame in the sense of Shtan'ko 
\cite{Sh2}*{Theorem 4}, \cite{McM}*{Theorem 3}).
We call a knotted solenoid {\it tubular} if its image is the intersection of a nested sequence $\dots\subset T_1\subset T_0$ of tame solid tori; 
and {\it tame} if in addition the core knot of each $T_{i+1}$ is a braiding of the core knot of $T_i$.
It is well-known that not every topological knot is tubular (see \cite{M24-3}*{remarks (1)--(4) on pp.~3--4}). %*{Remark \ref{rolf:nested}}).
It is easy to see that a topological knot is tame as a knotted solenoid if and only if it is tame in the usual sense (that is, ambient isotopic to 
a smooth knot).

The literature on knotted solenoids includes \cite{MiSc}, \cite{JWZZ}, \cite{CMR}, \cite{BaS}, \cite{Hui}.
All these papers deal only with tubular knotted solenoids, and with the exception of \cite{BaS} and one example in \cite{JWZZ}, they actually 
deal only with tame knotted solenoids.

By a {\it (tubular/tame) link of solenoids} we mean an embedding of a finite disjoint union of solenoids whose components are
(tubular/tame) knotted solenoids.
It turns out that:

\begin{enumerate}
\item Every invariant of F-isotopy extends to an invariant of tubular topological links.

\item Every $0$-solenoidal invariant extends to an invariant of tubular links of solenoids.

\item Every $0$-braidable invariant extends to an invariant of tame links of solenoids.

\item Every $0$-cableable invariant of link homotopy extends to an invariant of link homotopy of arbitrary links of solenoids.
\end{enumerate}

\noindent
In fact, we obtain somewhat stronger results (see Theorem \ref{extension} and Proposition \ref{stabilization}).
\end{remark}

\subsection{A problem originating from physics}

The simplest finite type invariants apart from the linking number show a rather complex behavior under 
cabling (see Examples \ref{casson-ex}, \ref{beta-ex}).
So the following question is certainly non-trivial:

\begin{problem}[Akhmetiev%
\footnote{Problem \ref{aa-problem} is implicit in the abstract of P. M. Akhmetiev's talk \cite{A-3}; and its variation 
with ``cableable'' (or ``asymptotic'' in Akhmetiev's terminology) replaced by its more complicated precursor version
(also called ``asymptotic'') is stated more explicitly in Akhmetiev's book \cite{A16}*{\S2.1}.}%
]\label{aa-problem}
Does there exist a cableable finite type invariant of links which is not a function of the pairwise linking numbers?
\end{problem}

Problem \ref{aa-problem} is motivated by the quest for invariants of magnetic fields that are modeled on link invariants and 
provide lower bounds for the energy of the field.
This needs to be explained, but the explanation is not so short --- so we postpone it to Appendix~\ref{magnetic} and include
only a brief summary here.
A magnetic field generated by a moving medium (such as magma within a planet or plasma within a star) is 
mathematically modeled as a divergence-free vector field on a 3-dimensional Riemannian manifold that is tangent to its boundary; 
the motions of the medium correspond to volume-preserving diffeomorphisms. 
These diffeomorphisms do not necessarily conserve the energy of the magnetic field (as it may dissipate into heat). 
Moreover, as shown by A. D. Sakharov and Ya. B. Zeldovich (the physicists) and M. Freedman (the mathematician), for one of 
the simplest examples of magnetic fields one can make the energy arbitrarily small by volume-preserving diffeomorphisms. 
On the other hand, back in the 1960s and 70s, V. I. Arnold and H. K. Moffatt constructed a real-valued invariant of magnetic fields 
(with respect to volume-preserving diffeomorphisms) based on the linking number of pairs of closed (or nearly closed) 
integral curves; the absolute value of this invariant provides a lower bound for the energy of the magnetic field. 
This leads to a natural question: is it possible to construct similar invariants for magnetic fields --- providing a lower bound 
for their energy --- by utilizing more sophisticated link invariants? This somewhat informal problem was raised by Arnold and Moffatt 
themselves in the 1980s and 90s; depending on the precise formulation, it can be regarded as solved in the affirmative 
(by results of M. Freedman et al.), solved in the negative (by results of E. A. Kudryavtseva et al.) or still open.
More details on the Arnold--Moffatt problem and an outline of how it leads to Akhmetiev's Problem \ref{aa-problem} can be found 
in Appendix \ref{magnetic}.

According to P. M. Akhmetiev, he constructed a $4$-cableable finite type invariant $M$ of $3$-component links which is not a function of 
the pairwise linking numbers and which is defined as a limit of integrals depending on a certain magnetic field with support in 
a tubular neighborhood of the given link \cite{A05}, \cite{A13}, \cite{A21} (see also \cite{A11}, \cite{A12}, \cite{A15}, \cite{A14}, 
\cite{A16}, \cite{A20}, \cite{A21'}).
Beware that the definition of $M$ is different in different papers as some errors got corrected over the years.%
\footnote{According to \cite{A21}*{\S5}, ``the extra terms (25)--(28) are not well-presented and an extra term (20)
is missing'' from the definition of $M$ in \cite{A13}. (This refers to \cite{A13}*{(25)--(28)} and \cite{A21}*{(20)} respectively.)
Some details on how to correct the aforementioned formulas (25)--(28) of \cite{A13} are given in \cite{A21'}*{\S5} where these 
are wrongly referred to as formulas of \cite{A14} (in reality, there are no formulas (25)--(28) in \cite{A14}). 
Some additional errors in \cite{A13}, \cite{A14} and \cite{A15} are also pointed out in \cite{A21}*{\S5} and \cite{A21'}*{\S5}.
%According to a private email from P. M. Akhmetiev from 14.02.2021, $M$ ``is a finite type invariant in the sense of Vassiliev.
%This is proved in [the paper \cite{A13}] with misprints, because the integral formula [for $M$ in \cite{A13}] turned out 
%to contain wrong summands. The invariant itself is correctly constructed in [the paper \cite{A05}], but it is not proved there that it is 
%a finite type invariant in the sense of Vassiliev.''
}
The minimal type of $M$ as a finite type invariant was first claimed to be 7 \cite{A21'}*{Theorem~1}, \cite{A-1}, \cite{A-3} and 
later to be 9 \cite{A24}*{Theorem 1}, \cite{A-6}; then the latter claim was also retracted (see \cite{A-8}*{0:09:30--0:11:00 in the video}),
and eventually, under influence of the present work, $M$ is again claimed to be of type 7 (see Remark \ref{contradiction}(e)).
It should also be noted that 5 basic properties persistently claimed of $M$ in Akhmetiev's papers and talks of 2017--2025 (namely: 
$4$-cableability; finite type; a crossing change formula under link homotopies; anti-chirality; and the value on the $3$-component 
Hopf link) turned out to be mutually contradictory when the present work appeared, and then under its influence one of them was 
abandoned (see Remark \ref{contradiction}).

Regardless of how reliable Akhmetiev's analytic solution of Problem \ref{aa-problem} is, from the viewpoint of pure mathematics 
it seems desirable not just to solve Problem \ref{aa-problem} via an analytic formula, 
but to find a solution which is understandable in terms of algebraic and/or combinatorial topology.
To this end let us recall that the Conway polynomial $\nabla_L(z)$ of an $m$-component link $L$ is always of the form 
$z^{m-1}\big(c_0+c_1z^2+\dots+c_rz^{2r}\big)$ for some $r$ (see e.g.\ \cite{M24-1}*{Lemma \ref{fti:lickorish}}).

The original construction of Akhmetiev's analytic invariant $M$ was presumably motivated in part by a certain 
combinatorial invariant $\gamma$ of $3$-component links \cite{M03}*{p.~11} which is defined as an explicit polynomial 
in $c_0$ and $c_1$ of the link and of its $2$-component sublinks (see Remark \ref{gamma} for further details).
In any case, according to \cite{A14}*{Main Theorem 3.3} and \cite{A21}*{Theorem 17} Akhmetiev's analytic invariant $M$ is itself a polynomial in $c_0$ and $c_1$ of the link and of its $2$-component sublinks.
Moreover it is of the form $\gamma+\delta+\epsilon$, where $\delta$ is an explicit polynomial in $c_0$ and $c_1$ of the $2$-component sublinks
and $\epsilon$ is an unknown polynomial of degree $\le 14$ in the pairwise linking numbers.
In particular, according to \cite{A21} Akhmetiev obtained an analytic proof of the following

\begin{conjecture}[Akhmetiev] \label{akh-conj} There exists a $4$-cableable invariant of $3$-component links which is a polynomial in the coefficients $c_0$, $c_1$ of the Conway polynomial of the link and of its proper sublinks, but is not a function of the pairwise linking numbers.
\end{conjecture}

Two natural questions arise: exactly what is this polynomial, and how to obtain a combinatorial proof of Conjecture \ref{akh-conj},
without relying on magnetic fields?

A number of papers and talks by P. M. Akhmetiev were devoted to this subject. 
Between 2019 and 2025 several attempts at a combinatorial proof of Conjecture \ref{akh-conj} containing explicit 
polynomials claimed to satisfy Conjecture \ref{akh-conj} appeared in Akhmetiev's papers 
\cite{A21'}*{Theorem 1}, \cite{A24}*{Theorem 1} and talks \cite{A-2}, \cite{A-3}, \cite{A-5}, \cite{A-6}.
Later he communicated that in his own view all these attempts were not successful
(see \cite{A-8}*{0:23:00--0:24:00 and/or 0:09:30--0:11:00 in the video}, see also \cite{A-7}).
See also \cite{A11}*{Theorem 8}, \cite{A16}*{Theorem 9} for attempted proofs of an earlier (2011) version 
of Conjecture \ref{akh-conj} and \cite{A05}*{\S1.1} for a still earlier (2005) form of Conjecture \ref{akh-conj} itself.

Akhmetiev's approaches to Conjecture \ref{akh-conj} in the aforementioned papers and talks were based 
on direct geometric constructions, including a frequent use of $\Delta$-moves.
On the other hand, one may observe that by using the Seifert--Cimasoni formula for the behavior of the Conway potential 
function under cabling (see Theorem \ref{cimasoni}) it must be conceptually easy (even if technically laborious) 
to prove or disprove Conjecture \ref{akh-conj}.
This was clear to me by February 2021, and I told Akhmetiev about this approach then, and repeated this suggestion 
after (and sometimes during) each of his subsequent talks at the Moscow Geometric Topology Seminar devoted to his attempts 
to prove Conjecture \ref{akh-conj} by his geometric techniques \cite{A-5}, \cite{A-6}, \cite{A-7}.
As he never made use of these suggestions by the last of these talks, in April 2025, I got convinced that I have
to do this job myself.

\subsection{A cableable finite type invariant}\label{cfti}

\begin{mainthm} \label{main-1} For each $m\ge 3$ there exists a $\Q$-valued cableable finite type invariant $\bar{\bar\omega}$ of 
$m$-component links which is not a function of the pairwise linking numbers.
Moreover,

(a) $\bar{\bar\omega}$ is $(m+1)$-satellitable;

(b) $\bar{\bar\omega}$ is of type $1+\frac{m(m+1)}2$ and of colored\,%
\footnote{An invariant $v$ of $m$-component links is said to be of colored type $n$ \cite{KL}, \cite{M24-1} 
if its standard extension to singular links, defined by 
$v\big(\raisebox{-2pt}{\includegraphics[width=0.35cm]{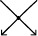}}\big)
=v\big(\raisebox{-2pt}{\includegraphics[width=0.35cm]{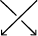}}\big)-v\big(\raisebox{-2pt}{\includegraphics[width=0.35cm]{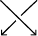}}\big)$, 
vanishes on all singular links with $>n$ self-intersections of components.
We recall that $v$ is said to be of type $n$, if the same extension vanishes on all singular links with $>n$
double points (regardless of whether they are self-intersections of components or intersections of
distinct components).} type $1$.

(c) $\bar{\bar\omega}$ is a polynomial in the coefficients $c_0$, $c_1$ of the Conway polynomials 
of the link and of its proper sublinks;

(d) $\bar{\bar\omega}(L)$ is not a function of invariants of proper sublinks of $L$, at least for $m=3,4$;

(e) $\bar{\bar\omega}(L)$ is not a sum of a link homotopy invariant and an invariant of type $\frac{m(m+1)}2$;

(f) $\bar{\bar\omega}$ is $\Z$-valued at least for $m=3$, and certainly $2^{m-1}\cdot 3\,\bar{\bar\omega}$ is $\Z$-valued.
\end{mainthm}

The polynomial in (c) for $m=3$ is written out explicitly in Addendum \ref{main-3comp}.
 
The definition of $\bar{\bar\omega}$ for arbitrary $m$ is given in Addendum \ref{mainmain}.
In terms of this definition, the assertions of Theorem \ref{main-1} are obtained as follows.

\begin{itemize}
\item Part (a) is proved in Corollary \ref{satellites'} and the weaker assertion that $\bar{\bar\omega}$ is cableable is proved
already in Theorem \ref{cables} (see also Theorem \ref{3-comp} for the case $m=3$).
\item The first assertion of (b) follows from Theorem \ref{mho}(b) and the Leibniz formula for finite type invariants (see \cite{M24-1}*{Corollary \ref{fti:product}}).
\item The second assertion of (b) follows from Proposition \ref{jump}.
\item Part (c) follows from Addendum \ref{main-mcomp}.
\item Part (d) for $m=3$ follows from Proposition \ref{brunnianity}(b) or alternatively from \cite{M24-1}*{Corollary \ref{fti:brunnianity}} (see also Remark \ref{gamma}, Corollary \ref{3-comp-bar} and
Lemma \ref{omega} below).
\item Part (d) for $m=4$ follows from Proposition \ref{brunn4}.
\item The assertion that $\bar{\bar\omega}$ is not a function of the pairwise linking numbers is a special case of part (e), which in turn
follows from Proposition \ref{lhlt}.
\item Part (f) follows from Corollaries \ref{integrality} and \ref{denominator}.
\end{itemize}

\begin{remark} It seems that apart from $\bar{\bar\omega}$ and its constant multiples and their sums with $4$-cableable
polynomials in the pairwise linking numbers there exist no other solutions to Conjecture \ref{akh-conj}.
Moreover, it appears that it should be possible to prove this claim quite easily by analyzing the proof of Theorem \ref{main-1}
(see also Proposition \ref{axioms}(c)).
We leave the details to the interested reader.
\end{remark}

Theorem \ref{main-1} implies

\begin{corollary}\label{bbbomega}
Let $\lambda(L)$ and $\lambda^\circ(L)$ be as in Example \ref{lk}, and let
\[\bar{\bar{\bar\omega}}(L)=\begin{cases}
\dfrac{\bar{\bar\omega}(L)}{\lambda(L)\lambda^\circ(L)}&\text{if }\lambda(L)\ne 0;\\
0&\text{otherwise.}
\end{cases}\]
Then $\bar{\bar{\bar\omega}}$ is $0$-solenoidal.
\end{corollary}

\subsection{Conway polynomial}
For a link $L=(K_1,\dots,K_m)$ let $\nabla_L$ be its Conway polynomial, and let us consider the formal power series 
\[\bar\nabla_L(z)=\frac{\nabla_L(z)}{\nabla_{K_1}(z)\cdots\nabla_{K_m}(z)},\]
which is easily seen to be invariant under PL isotopy (that is, under addition and deletion of local knots), cf.\ e.g.\ 
\cite{M24-1}*{\S\ref{fti:conway}} or \cite{Ro4}.
Let us note that every solenoidal invariant is an invariant of PL isotopy (since PL isotopy implies F-isotopy).

Writing $\nabla_L(z)$ in the form $z^{m-1}\big(c_0(L)+c_1(L)z^2+\dots+c_r(L)z^{2r}\big)$, we get
\[\bar\nabla_L(z)=z^{m-1}\Big(\alpha(L)+\beta(L)z^2+\dots\Big),\]
where $\alpha(L)=c_0(L)$ and $\beta(L)=c_1(L)-c_0(L)\big(c_1(K_1)+\dots+c_1(K_m)\big)$.

When $L$ is a $2$-component link, $\alpha(L)$ is the linking number and $\beta(L)$ is the so-called generalized Sato--Levine invariant,
which for any fixed value of $\lk(L)$ generates the group of colored type $1$ invariants modulo colored type $0$ invariants
\cite{KL} (see also \cite{M24-1}*{\S\ref{fti:fti}, \S\ref{fti:conway}}).
The generalized Sato--Levine invariant emerged independently in the work of Traldi \cite{Tr2}*{\S10}, Polyak--Viro (see \cite{AMR}),
Kirk--Livingston \cite{KL} (see also \cite{Liv}), Akhmetiev (see \cite{AR}) and Nakanishi--Ohyama \cite{NO}.
It is proved in \cite{NO} (see also \cite{M18} for an alternative proof) that $\alpha(L)$ and $\beta(L)$ constitute a complete set of 
invariants of self $C_2$-equivalence (also known as $\Delta$-link homotopy) of $2$-component links; and it is shown in 
\cite{M18}*{Corollary 5.2} that self $C_2$-equivalence is the same thing as $\frac12$-quasi-isotopy.

When $L$ is an $m$-component link, $\alpha(L)$ is a polynomial in the pairwise linking numbers $\lk(K_i,K_j)$ (see Corollary \ref{hhh2}), 
and $\beta(L)$, being a colored type $1$ invariant (see \cite{M24-1}*{\S\ref{fti:conway}}), is an invariant of self $C_2$-equivalence 
(see \cite{M24-1}*{Remark \ref{fti:self-colored}}, \cite{NO}*{Proposition 2}).

\begin{addendumtotheorem}[\ref{main-1}] \label{main-mcomp}
$\bar{\bar\omega}(L)$ is a polynomial in the following: $\beta(L)$, the invariants $\beta(\Lambda)$ for proper sublinks $\Lambda\subset L$ 
and the pairwise linking numbers.
\end{addendumtotheorem}

Addendum \ref{main-mcomp} follows from Corollary \ref{forests''} and Lemma \ref{omega}.

The said polynomial has been written out only partially, but the interested reader should be able to obtain an explicit formula
(which is not going to be short) by using the same methods as in the proof of Addendum \ref{main-mcomp}.
A fully explicit formula is written out in the case $m=3$:

\begin{addendumtotheorem}[\ref{main-1}] \label{main-3comp}
For a $3$-component link $L=(K_1,K_2,K_3)$ let $l_{ij}=\lk(K_i,K_j)$, $\lambda=l_{12}l_{23}l_{31}$ and $\beta_{ij}=\beta(K_i,K_j)$.
Then
\[\bar{\bar\omega}(L)=\beta(L)\lambda-\alpha(L)\sum\limits_{(i,j,k)\in\langle 3\rangle!}l_{ij}l_{jk}\beta_{ik}-
\lambda\sum\limits_{(i,j,k)\in\langle 3\rangle!}l_{ij}l_{jk}\tfrac{2l_{ik}^2+l_{ij}l_{jk}+1}{12},\]
where $\langle 3\rangle!$ denotes the set $\{(1,2,3),\,(2,3,1),\,(3,1,2)\}$ of all circular shifts of $(1,2,3)$.%
\footnote{Here is some logic behind this notation.
$[n]!$ is a rather standard notation (also used below) for the set of all permutations of the set $[n]=\{1,\dots,n\}$
(understood in the static sense, i.e.\ not as operations but as arrangements of the elements of $[n]$ in some linear order).
We denote by $\langle n\rangle$ the same set $\{1,\dots,n\}$ endowed with the structure of a (total) cyclic order, i.e.\ a ternary
relation $[a,b,c]$ satisfying four axioms: $[a,b,c]\Rightarrow [b,c,a]$; $[a,b,c]\Rightarrow\neg[c,b,a]$;
$[a,b,c]\land[a,c,d]\Rightarrow[a,b,d]$; $a\ne b\ne c\ne a\Rightarrow [a,b,c]\lor[c,b,a]$.
For a cyclically ordered set $Q$ we denote by $Q!$ the set of all its (circular) permutations, that is, all arrangements 
of the elements of $Q$ in some linear order, compatible with the cyclic structure of $Q$.
(Compatibility means that $[a,b,c]\Leftrightarrow a\<b\<c\lor b\<c\<a\lor c\<a\<b$.)}
\end{addendumtotheorem}

Addendum \ref{main-3comp} follows from Corollary \ref{3-comp-barbar} and Lemma \ref{omega}.

\begin{remark}\label{predicted}
It must be noted that P. M. Akhmetiev correctly predicted the terms 
\[\beta(L)\lambda-\alpha(L)\sum\limits_{(i,j,k)\in\langle 3\rangle!}l_{ij}l_{jk}\beta_{ik}\]
(in a different notation, explained in Remark \ref{gamma} below), firstly with a sign error \cite{A05}*{\S1.1}, \cite{A12}*{\S5.5} 
and then without it \cite{A14}*{Theorem 3.3}, \cite{A20}*{Assertion B.1}, \cite{A-1}; and later also the terms
\[\lambda\sum\limits_{(i,j,k)\in\langle 3\rangle!}l_{ij}l_{jk}\tfrac{2l_{ik}^2+l_{ij}l_{jk}}{12},\]
(see \cite{A21'}*{Theorem 1}, \cite{A-3}, \cite{A21}*{Conjecture 18}).
However, the eventually remaining term was not
guessed correctly: it was first claimed to be zero \cite{A21'}*{Theorem 1}, \cite{A-3}, \cite{A21}*{Conjecture 18} and later to be a degree 9 polynomial in the pairwise linking numbers \cite{A24}*{Theorem 1}, \cite{A-6}.
Judging by Akhmetiev's talks on the subject \cite{A-1}, \cite{A-2}, \cite{A-3}, \cite{A-4}, \cite{A-5}, \cite{A-6}, \cite{A-7}, at least 
some of these conjectural formulas were primarily based on his ``numerical experiments''.
\end{remark}

\begin{remark} Upon learning the formula of Addendum \ref{main-3comp} from my talk on the subject of the present paper (cf.\ \cite{A-8}),
P. M. Akhmetiev tried to reprove the $4$-cableability of the right hand side of this formula by his geometric techniques based on 
$\Delta$-moves, using only the Conway polynomial and not the multi-variable Conway potential function.
According to \cite{AS}*{Theorem 1} he succeeded in doing this in the case where the pairwise linking numbers $l_{12}$ and $l_{13}$
are equal to each other, and the condition of $4$-cableability is relaxed to the case where $p_2=p_3$.
\end{remark}

Addendum \ref{main-3comp} implies

\begin{corollary} \label{integrality} $\bar{\bar\omega}(L)\in\Z$ for $3$-component links $L$.
\end{corollary}

\begin{proof} It suffices to show that $n:=abc\big(ab(2c^2+ab+1)+bc(2a^2+bc+1)+ca(2b^2+ca+1)\big)$ is divisible by 12
for any $a,b,c\in\Z$.
Let us write $n=abck$, where \[k=a^2b^2+b^2c^2+c^2a^2+ab(2c^2+1)+bc(2a^2+1)+ca(2b^2+1).\]
If $3\nmid abc$, then $a^2\equiv b^2\equiv c^2\equiv 1\pmod3$, and it follows that $3\mid k$.
It remains to show that $4\mid n$.
If more than one of $a,b,c$ is even, then $4\mid abc$.
If precisely one of $a,b,c$ is even, then $2\mid abc$ and $2\mid k$.
Suppose that all of $a,b,c$ are odd.
Then $a^2\equiv b^2\equiv c^2\equiv 1\pmod4$.
Hence $k\equiv3+3ab+3bc+3ac\pmod4$.
If all of $a,b,c$ have the same residues modulo $4$, then $ab\equiv bc\equiv ca\equiv 1\pmod4$
and hence $k\equiv 12\pmod4$.
If not all of $a,b,c$ have the same residues modulo $4$, then two of $ab$, $bc$ and $ca$
have residue $3$ and one has residue $1$ modulo $4$ and hence $k\equiv 24\pmod4$.
\end{proof}

\subsection{Conway potential function}
For a link $L=(K_1,\dots,K_m)$ let $\Omega_L(x_1,\dots,x_m)$ be its Conway potential function;
it has a number of different constructions and some axiomatic characterizations (see references on the first page of \cite{M24-2}).
Let $\mho_L(z_1,\dots,z_m)$ be the expansion of $\Omega_L(x_1,\dots,x_m)$ in the Conway variables $z_i=x_i-x_i^{-1}$, understood
as a formal power series in $z_1,\dots,z_m$ (see \S\ref{mho-section} for further details on this expansion).
Then the formal power series
\[\bar\mho_L(z_1,\dots,z_m):=\frac{\mho_L(z_1,\dots,z_m)}{\nabla_{K_1}(z_1)\cdots\nabla_{K_m}(z_m)}\]
is invariant under PL isotopy.
Let $\omega(L)$ be its coefficient at $z_1\cdots z_m$.
In fact $\omega(L)$ is also the coefficient of $\mho_L(z_1,\dots,z_m)$ at $z_1\cdots z_m$, since each $\nabla_{K_i}(z_i)$ 
is of the form $1+\text{\rm (terms of degrees $\ge 2$)}$.

Next let $\ell=\sqrt{\prod_{j<k}\lk(K_j,K_k)}\in[0,\infty)\cup i[0,\infty)$ and let
\[\bar{\bar\mho}_L(z_1,\dots,z_m)=\frac{\ell^{4-m}\,\bar\mho_L(\ell z_1,\dots,\ell z_m)}
{\prod_{i<j}\bar\mho_{(K_i,K_j)}(\ell z_i,\ell z_j)},\]
which does make sense also in the case $\ell=0$ if understood appropriately, with cancellations preceding evaluations 
(see \S\ref{3-comp-section} for the details; the point of the substitution $z_i\mapsto\ell z_i$ is exactly to ensure 
that the fraction makes sense when $\ell=0$).
Finally let \[\bar{\bar{\bar\mho}}_L(z_1,\dots,z_m)=
\frac{\bar{\bar\mho}_L(z_1,\dots,z_m)}{\prod\limits_{(i,j,k)\in\langle m\rangle^{\underline{\!\langle 3\rangle\!}}}
\big(1+\frac{1}{12}l_{ij}l_{ik}\lambda z_jz_k\big)},\]
where $l_{ij}=\lk(K_i,K_j)$, $\lambda=\prod_{i<j}l_{ij}$ and $\langle m\rangle^{\underline{\!\langle 3\rangle\!}}$ denotes the set of all injections 
$\langle 3\rangle\to\langle m\rangle$ that respect the cyclic order.%
\footnote{In other words, triples $(i,j,k)$ of elements of $\{1,\dots,m\}$ such that either $i\<j\<k$ or $j\<k\<i$ or $k\<i\<j$.}

\begin{remark} Admittedly, the definition of $\bar{\bar{\bar\mho}}_L$ looks somewhat ad hoc.
In view of how $\bar\mho_L$ and $\bar{\bar\mho}_L$ are defined, it would be logical to somehow make sense out of the fraction 
\[\frac{\bar{\bar\mho}_L(z_1,\dots,z_m)}{\prod_{i<j<k}\bar{\bar\mho}_{(K_i,K_j,K_k)}(z_i,z_j,z_k)}.\] 
But it is not clear to the author how to do it.
Of course, this fraction does make sense as an element of the field of fractions of $\Q[[z_1,\dots,z_m]]$ (unless the denominator vanishes identically)
--- but a priori this element does not have any ``coefficients'' (compare \cite{AMK}).
\end{remark}

\begin{addendumtotheorem}[\ref{main-1}] \label{mainmain}
For a link $L$ of $m\ge 3$ components $\bar{\bar\omega}(L)$ is the coefficient of
$\bar{\bar{\bar\mho}}_L(z_1,\dots,z_m)$ at $z_1\cdots z_m$. 
\end{addendumtotheorem}

This definition of $\bar{\bar\omega}(L)$ is unwrapped in Lemma \ref{omega} into an explicit formula in terms of
$\omega(L)$, the invariants $\omega(\Lambda)$ for two-component sublinks $\Lambda$ of $L$ and the pairwise linking numbers.

\subsection{Some properties}

\begin{proposition} \label{axioms}
In the case of $3$-component links $\bar{\bar\omega}$ satisfies the following.

(a) The crossing change formula
\[\bar{\bar\omega}(K_{i+},K_j,K_k)-\bar{\bar\omega}(K_{i-},K_j,K_k)=(l_{i'j}l_{jk}l_{ki''}-l_{i''j}l_{jk}l_{ki'})^2\]
where the knots $K_{i+}$ and $K_{i-}$ differ by a positive self-intersection, $K_{i'}$ and $K_{i''}$ are the smoothed lobes 
of the singular knot in between and $l_{ij}=\lk(K_i,K_j)$.

(b) $\bar{\bar\omega}$ is anti-chiral (that is, reverses its sign when the orientation of $S^3$ is reversed).

(c) Every $4$-cableable finite type invariant of $3$-component links satisfying (a) and (b) equals $\bar{\bar\omega}$.

(d) $\bar{\bar\omega}$ vanishes on the $3$-component Hopf link  (consisting of 
three fibers of the Hopf fibration $S^3\to S^2$, with pairwise linking numbers $+1$).
\end{proposition}

Of course, the crossing change formula (a) is the reason why $\bar{\bar\omega}$ is of colored type $1$.

\begin{proof}[Proof. (a)] This is the $3$-component case of the general crossing change formula for 
$\bar{\bar\omega}$ (see Proposition \ref{jump}), rewritten using that $l_{ij}=l_{i'j}+l_{i''j}$.
\end{proof}

\begin{proof}[(b)] In general, when the link is replaced by its mirror image, $\bar{\bar\omega}$ is multiplied by 
$(-1)^{\frac{(m-1)(m-2)}2}$ by Lemma \ref{omega} and \cite{Ha}*{5.6}.
\end{proof}

\begin{proof}[(c)] Let $\bar{\bar\omega}'$ be the given invariant.
Then by (a) $\bar{\bar\omega}-\bar{\bar\omega}'$ is a link homotopy invariant
and by Theorem \ref{main-1}(b) it is a finite type invariant.
Hence by \cite{MT} $\bar{\bar\omega}-\bar{\bar\omega}'$ is a polynomial in the pairwise linking numbers. 
But by Theorem \ref{main-1}(a) $\bar{\bar\omega}-\bar{\bar\omega}'$ is also $4$-cableable, and then it is easily seen to be 
a constant multiple of $\lambda^2$ (see Example \ref{lk} concerning $\lambda$).
Now by (b) $\bar{\bar\omega}-\bar{\bar\omega}'$ is anti-chiral, whereas $\lambda^2$ is obviously not anti-chiral
(because it is always positive).
Thus $\bar{\bar\omega}-\bar{\bar\omega}'=0$.
\end{proof}

\begin{proof}[(d)] By either Addendum \ref{main-3comp} or Addendum \ref{mainmain}.
\end{proof}

\begin{remark} \label{contradiction} Let us now discuss Akhmetiev's analytic invariant $M$.

(a) According to \cite{A14}*{Lemma 3.4} and \cite{A21}*{Lemma 13} (see also \cite{A-3}) $M$
satisfies the crossing change formula of Proposition \ref{axioms}(a).%
\footnote{More precisely, $\tilde M$ of \cite{A14}, $-M$ of \cite{A21} and $-M_3$ of \cite{A-3}
satisfy a crossing change formula which is easily seen to be the same as that of Proposition \ref{axioms}(a) 
using that $l_{ij}=l_{i'j}+l_{i''j}$.}

(b) According to \cite{A11}*{Remark to Theorem 1} (reiterated in \cite{A12}*{Remark to Theorem 1}, \cite{A15}*{Remark 2.2}, 
\cite{A24}*{Theorem 1}, \cite{A-3}, \cite{A-5}) $M$ is anti-chiral.

(c) According to \cite{A20}*{Lemma 4.2}, \cite{A-1} (reiterated in \cite{A21'}*{Theorem 1}, \cite{A24}*{Theorem 1}, \cite{A21}*{\S5}, 
\cite{A-6}) $M$ does not vanish on the $3$-component Hopf link. 

(d) Akhmetiev's assertions mentioned in (a), (b), (c) as well as his assertions that $M$ is $4$-cableable and of finite type
cannot be all true.
Indeed, if all of them are true, then by Proposition \ref{axioms}(c) $M$ must equal $\bar{\bar\omega}$.
But then (c) contradicts Proposition \ref{axioms}(d).

(e) After learning of Theorem \ref{main-1}(a)--(c) and Addendum \ref{main-3comp} from my talk on the subject of the present paper 
(cf.\ \cite{A-8}), P. M. Akhmetiev co-authored a new paper, where he now says \cite{AS}*{p.\ 95} that his assertion (c) was in error due to 
a mistake in the proof of \cite{A20}*{Lemma 4.2} and that $M$ equals the invariant $\bar{\bar\omega}$ of the present paper
\cite{AS}*{Theorem 1}.
%In a private correspondence on 10.10.2026 he also clarified the error in the proof of \cite{A20}*{Lemma 4.2}: ``I have not surveyed
%all the summands, I only found a nonzero one and overlook that it gets canceled.''
\end{remark}

\begin{remark} 
P. M. Akhmetiev also discussed a number of further assertions related to Problem \ref{aa-problem} and
Conjecture \ref{akh-conj} \cite{A22}, \cite{A-1}, \cite{A-4}, \cite{A-4a}, \cite{A-4b} (see also \cite{AD}).
Some of them involve invariants of $4$- and $5$-component links, which are quite different from $\bar{\bar\omega}$.
\end{remark}

\subsection{Proper sublinks}

This subsection is devoted to a proof for $m=3,4$ and a discussion of the remaining cases of the following

\begin{conjecture} \label{irred-conj}
$\bar{\bar\omega}(L)$ is not a function of invariants of proper sublinks of $L$.
\end{conjecture}

Here $\bar{\bar\omega}(L)$ can be replaced by $\lambda\omega(L)$, where $\lambda=\lambda(L)$ is the product of the pairwise linking numbers of $L$,
because $\bar{\bar\omega}(L)-\lambda\omega(L)$ is a function of invariants of two-component sublinks of $L$ (more precisely, it is a polynomial 
in their generalized Sato--Levine invariants and linking numbers --- see Lemma \ref{omega} and Corollary \ref{2-comp}).

\begin{proposition}\label{brunnianity}

(a) $\omega(L)$ is not a function of invariants of proper sublinks of $L$.

(b) For $m=3$, $\lambda(L)\omega(L)$ is not a function of invariants of proper sublinks of $L$.

(c) For $m>3$ there is for each $i=1,\dots,m$ a decomposition $\lambda\omega(L)=\kappa_i(L)+\kappa'_i(L)$, where $\kappa_i(L)$ 
is a function of invariants of $3$-component sublinks of $L$ and $\kappa'_i(L)$ is a finite type invariant which is invariant 
under link homotopy with support in the $\mathrm{i}$th component.
Moreover, if the pairwise linking numbers of $L$ are all equal to each other, then $\kappa_i(L)$ does not depend on $i$
(and so $\kappa'_i(L)$ is a link homotopy invariant).
\end{proposition}

Let us note that for $m=4,5$ all finite type invariants of $m$-component links that are link homotopy invariants are known 
to be polynomials in the pairwise linking numbers \cite{MT} (but this is not so for $m=6$ \cite{Lin}).

\begin{figure}[h]
\includegraphics[width=0.95\linewidth]{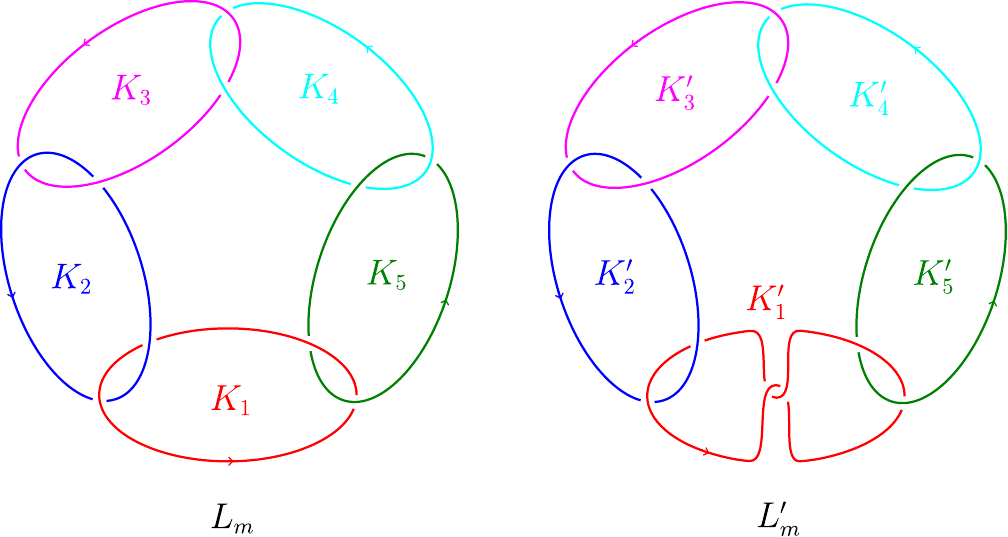}
\caption{Links $L_m$ and $L_m'$ for $m=5$.}
\label{circle}
\end{figure}

\begin{proof}[Proof. (a)] $\omega(L)$ has a remarkably simple crossing change formula for a positive self-intersection of the $i$th component of $L$:
\[\omega(L_+)-\omega(L_-)
=\sum_{(j_1,\dots,j_{m-1})\in([m]\but\{i\})!}l_{i'j_1}l_{j_1j_2}\cdots l_{j_{m-2}j_{m-1}}l_{j_{m-1}i''},\tag{$\x$}\label{jump-formula}\]
where $L_\pm=(K_1,\dots,K_{i_\pm},\dots,K_m)$, the singular knot between $K_{i_+}$ and $K_{i_-}$ is smoothed to a two-component link $(K_{i'},K_{i''})$
and $l_{jk}=\lk(K_j,K_k)$ (see Proposition \ref{jump}).

Figure \ref{circle} shows a pair of links $L_m$, $L_m'$ which differ by a single self-intersection of the component $K_1$.
When (\ref{jump-formula}) is applied to this self-intersection, only one summand in (\ref{jump-formula}) is nonzero, 
and each factor in this summand equals $1$.
Thus we get $\omega(L_m')-\omega(L_m)=1$.
On the other hand, it is easy to see that every proper sublink of $L_m$ is equivalent to the corresponding proper sublink of $L_m'$.
\end{proof}

\begin{proof}[(b)] In the case $m=3$ the pairwise linking numbers of $L_m$ and $L_m'$ all equal $1$, so in this case the proof of (a) 
also works for $\lambda\omega$.
\end{proof}

\begin{proof}[(c)] 
The $3$-component case of the crossing change formula (\ref{jump-formula}) takes the form
\[\omega(K_{i_+},K_j,K_k)-\omega(K_{i_-},K_j,K_k)=l_{i'j}l_{jk}l_{ki''}+l_{i'k}l_{kj}l_{ji''}.\label{jump-formula'}\]
Let 
\[\kappa_i(L)=\sum_{p<q<r}\omega(K_p,K_q,K_r)
\sum_{(j_2,\dots,j_{m-2})\in([m]\but\{p,q,r\})!}l_{j_2j_3}\cdots l_{j_{m-3}j_{m-2}}\,P,\]
where 
\[P=\begin{cases}
\dfrac{\lambda}{l_{jk}}l_{jj_2}l_{j_{m-2}k}&\text{if } \{i,j,k\}=\{p,q,r\},\, j<k\\
\frac13\sum\limits_{(j,k)\in\{(p,q),(p,r),(q,r)\}}\dfrac{\lambda}{l_{jk}}l_{jj_2}l_{j_{m-2}k}&\text{if } i\notin\{p,q,r\}.
\end{cases}\]
It is easy to see that $\kappa_i$ and $\lambda\omega$ have the same crossing change formula for self-intersections
of the $i$th component.
Hence $\kappa'_i:=\lambda\omega-\kappa_i$ is invariant under self-intersections of the $i$th component.
Also $\kappa_i$ is clearly a finite type invariant (see \cite{M24-1}*{Corollary \ref{fti:product}}), and hence so is $\kappa'_i$.
When all the pairwise linking numbers equal $l$, we have
$\kappa_i(L)=(m-3)!\ l^{m-3}\lambda\sum_{p<q<r}\omega(K_p,K_q,K_r)$.
\end{proof}

\begin{remark} \label{symmetry}
It may appear that the link $L_m'$ in Figure \ref{circle} is asymmetric, but in fact it is symmetric:
its variation obtained by cyclically permuting the components, $(K_1',\dots,K_m')\mapsto(K_2',\dots,K_m',K_1')$, is ambient isotopic to $L_m'$ 
by means of a $2\pi/m$ rotation followed by one full twist of the clasp between the components originally called $K_1'$ and $K_2'$.
\end{remark}

\begin{remark}
If one could represent $\omega(K_i,K_j,K_k)$ as a sum of three finite type invariants $\omega^i+\omega^j+\omega^k$
such that $\omega^i$ has the same crossing change formula as $\omega$ under a self-intersection of $K_i$
(that is, $\omega^i(K_{i_+},K_j,K_k)-\omega^i(K_{i_-},K_j,K_k)=l_{i'j}l_{jk}l_{ki''}+l_{i'k}l_{kj}l_{ji''}$) 
but does not change under self-intersections of $K_j$ and $K_k$, then it would follow similarly to the proof of 
Proposition \ref{brunnianity}(c) that Conjecture \ref{irred-conj} is false for $m=4,5$.%
\footnote{A related observation can be found in \cite{A22}*{Theorem 12(1)}.
In this connection it has to be noted that \cite{A22}*{Theorem 12(2)}, asserting that a certain invariant $\Delta_{(1,2);\,(3,4)}(L)$
is not a function of invariants of proper sublinks of $L$, is proved incorrectly.
Its proof is based on the assertion that $\Delta_{(1,2);\,(3,4)}(L)$, which is a function of the coefficients of $\mho_L(z_1,z_2,z_3,z_4)$ 
of total degrees $\le 4$, jumps by $\pm 1$ under a $C_3$-move involving $4$ distinct components of the link.
This assertion (which is stated without proof) is false, since all the said coefficients are zero for every link $L$ with vanishing pairwise 
linking numbers by a well-known result of Traldi and Levine (see Remark \ref{traldi-levine}).}
Now it follows from Remark \ref{symmetry} that such an invariant $\omega^i(K_i,K_j,K_k)$ does not exist.
But it turns out that something similar does exist: one can show that Milnor's invariant $\bar\mu_{iijk}(L)=\bar\mu_{iikj}(L)$ 
(see \S\ref{milnor}) does not change under self-intersections of $K_j$ and $K_k$ and jumps by $l_{i'j}l_{ki''}+\delta_{iijk}(L)\Z=l_{i'k}l_{ji''}+\delta_{iijk}(L)\Z$ under 
a positive self-intersection of $K_i$.
\end{remark}

\begin{proposition} \label{brunn4} For $m=4$, $\lambda(L)\omega(L)$ is not a function of invariants of proper sublinks of $L$.
\end{proposition}

\begin{figure}[h]
\includegraphics[width=0.85\linewidth]{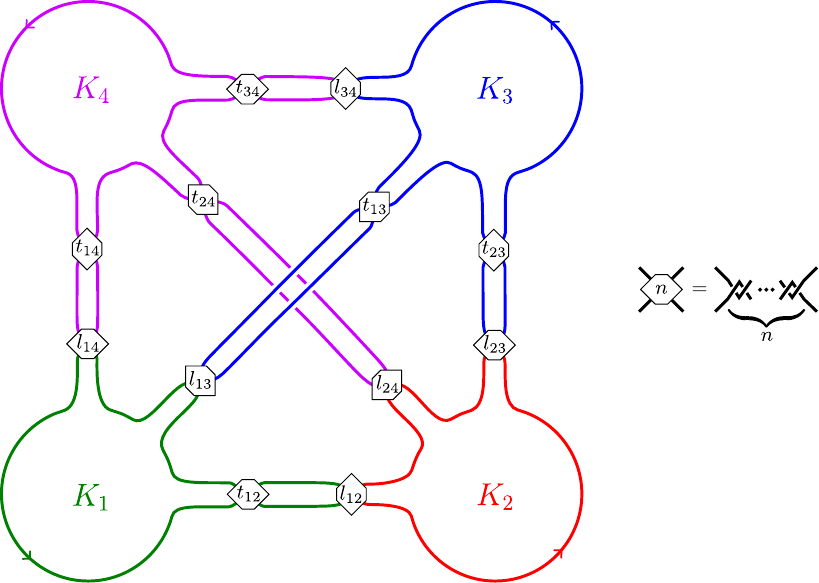}
\caption{The link $S\Big(\begin{smallmatrix} l_{12}& l_{13} & l_{14}\\ & l_{23} & l_{24}\\ & & l_{34}\end{smallmatrix},
\begin{smallmatrix} t_{12}& t_{13} & t_{14}\\ & t_{23} & t_{24}\\ & & t_{34}\end{smallmatrix}\Big)$.}
\label{standard}
\end{figure}

\begin{proof}
Let us consider the standard twisted $4$-component link 
$L:=S\Big(\begin{smallmatrix} l_{12}& l_{13} & l_{14}\\ & l_{23} & l_{24}\\ & & l_{34}\end{smallmatrix},
\begin{smallmatrix} t_{12}& t_{13} & t_{14}\\ & t_{23} & t_{24}\\ & & t_{34}\end{smallmatrix}\Big)$ (see Figure \ref{standard}).
Thus $L=(K_1,K_2,K_3,K_4)$ has $\lk(K_i,K_j)=l_{ij}$ and $t_{ij}$ full twists added to the clasp between $K_i$ and $K_j$.
It follows similarly to Remark \ref{symmetry} that when we remove one component from $L$, 
the resulting $3$-component link $(K_i,K_j,K_k)$ depends only on its pairwise linking numbers 
$l_{ij}$, $l_{jk}$, $l_{ik}$ and the total twisting $t_{ij}+t_{jk}+t_{ik}$.
In particular, if $t_{ij}+t_{jk}+t_{ik}=0$ whenever $i<j<k$, then $L$ has the same proper sublinks as 
the standard untwisted link $L':=S\Big(\begin{smallmatrix} l_{12}& l_{13} & l_{14}\\ & l_{23} & l_{24}\\ & & l_{34}\end{smallmatrix},
\begin{smallmatrix} 0& 0 & 0\\ & 0 & 0\\ & & 0\end{smallmatrix}\Big)$ with the same linking matrix.
On the other hand, formula (\ref{jump-formula}) implies that
\[\omega(L)-\omega(L')=\sum_{i<j}t_{ij}l_{ij}l_{kl}\big(l_{ik}l_{jl}+l_{il}l_{jk}\big),\]
where $i,j,k,l$ are pairwise distinct.

\begin{figure}[h]
\includegraphics[width=           \linewidth]{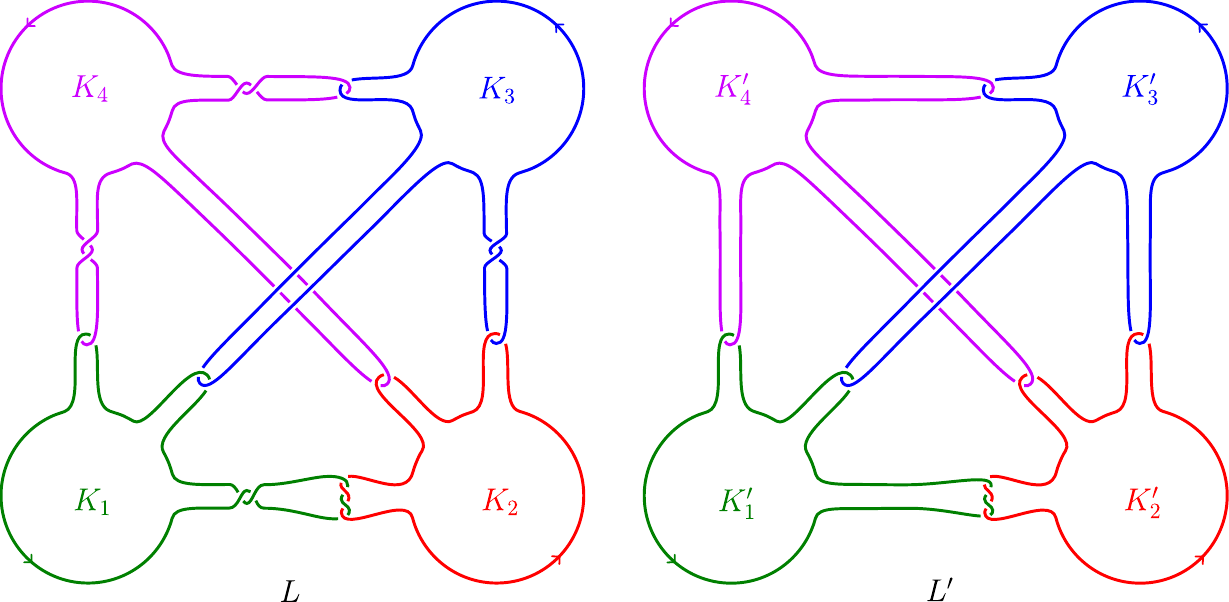}
\caption{Links $L=S\Big(\begin{smallmatrix} 2& 1 & 1\\ & 1 & 1\\ & & 1\end{smallmatrix},
\begin{smallmatrix} 1& 0 & -1\\ & -1 & 0\\ & & 1\end{smallmatrix}\Big)$ and 
$L'=S\Big(\begin{smallmatrix} 2& 1 & 1\\ & 1 & 1\\ & & 1\end{smallmatrix},
\begin{smallmatrix} 0& 0 & 0\\ & 0 & 0\\ & & 0\end{smallmatrix}\Big)$.}
\label{standardspecial}
\end{figure}
Let us set $t_{12}=t_{34}=1$, $t_{14}=t_{23}=-1$, $t_{13}=t_{24}=0$, $l_{12}=2$ and all other $l_{ij}=1$ (see Figure \ref{standardspecial}).
Then the condition $t_{ij}+t_{jk}+t_{ik}=0$ whenever $i<j<k$ is satisfied, so that $L$ and $L'$ have the same proper sublinks.
On the other hand, $\omega(L)-\omega(L')=4+4-3-3+0+0=2$, so $\lambda(L)\omega(L)-\lambda(L')\omega(L')=4$.
\end{proof}

\subsection{Braidability modulo invariants of lower type} \label{bml}
Given a link $L=(K_1,\dots,K_m)$, let $L_{p_1,q_1;\,\cdots;\,p_m,q_m}$ be obtained from $L$ by replacing each $K_i$ with its $(p_i,q_i)$-cable.

\begin{example} \label{casson-ex}
If $c_1(K)$ is the Casson knot invariant (= the coefficient of $\nabla_K(z)$ at $z^2$), then $c_1(K_{p,q})=p^2c_1(K)+\frac{(p^2-1)(q^2-1)}{24}$
(see Corollary \ref{casson}).
Let us note that if $p$ and $q$ are coprime, then $24\mid(p^2-1)(q^2-1)$.%
\footnote{Indeed, $3$ cannot divide both $p$ and $q$.
If for instance $3$ does not divide $p$, then it divides $(p+1)(p-1)$.
Also $p$ and $q$ cannot both be even.
If for instance $p$ is odd, then $(p+1)(p-1)$ is divisible by $8$.}
\end{example}

\begin{example} \label{beta-ex}
If $L$ is a $2$-component link, and $\beta(L)$ is, as before, the coefficient of $\bar\nabla_L(z)=\nabla_L/(\nabla_{K_1}\nabla_{K_2})$ 
at $z^3$, then
\[\beta(L_{p_1,q_1;\,p_2,q_2})=p_1^2p_2^2\beta(L)+
\tfrac{l}2\Big(q_1p_2^2l\tfrac{p_1(p_1+1)(p_1-1)}6+q_2p_1^2l\tfrac{p_2(p_2-1)(p_2+1)}6+\tfrac{p_1p_2(p_1p_2-1)(p_1p_2l^2+1)}6\Big),\]
where $l=\lk(L)$ (see Corollary \ref{beta-cable}(a)).
Let us note that the contents of the big parentheses is always integer,%
\footnote{If $3\nmid l$, then $l^2\equiv 1\pmod 3$, but always $2\mid p(p-1)$ and $3\mid p(p-1)(p'+1)$ where $p'\equiv p\pmod 3$.}
and moreover it is even if $l$ is odd, as long as each $q_i$ is coprime to $p_i$.%
\footnote{If $p_1$ and $p_2$ are both even or both odd, or if $4\mid p_2$, then each of the $3$ summands is even.
If $p_1$ is odd and $p_2\equiv 2\pmod 4$, then $q_2$ is odd, so the first summand is even and the latter two are both odd.}
\end{example}

For a nonzero integer $n$ let $B_n$ be the group of braids on $|n|$ strands, oriented upwards if $n>0$ and downwards if $n<0$.
Given braids $b_1,\dots,b_m$ such that each $b_i\in B_{p_i}$ and its closure is a knot, let $L_{b_1\dots b_m}$ be the 
$(p_1,\dots,p_m)$-braiding of $L$ obtained by replacing each $K_i$ with the oriented closure of $b_i$ sitting in a tubular neighborhood 
of $K_i$.

Let $v$ be a type $n$ invariant of $m$-component links.
We call $v$ {\it $(k_1,\dots,k_m)$-braidable modulo type $n-1$ invariants} 
if for every $p_1,\dots,p_m\in\Z\but\{0\}$ there exists a type $n-1$ invariant $v'_{p_1\dots p_m}$ of $m$-component links 
such that
\[v(L_{p_1,1;\,\cdots;\,p_m,1})=p_1^{k_1}\cdots p_m^{k_m}v(L)+v'_{p_1\dots p_m}(L)\] 
for every $m$-component link $L$.
It is not hard to see (cf.\ \S\ref{fti-cabling-section}) that $v$ is $(k_1,\dots,k_m)$-braidable modulo type $n-1$ invariants if and only if 
for every $p_1,\dots,p_m\in\Z\but\{0\}$ and for every $m$-tuple of braids $b_i\in B_{p_i}$ whose closures are knots there exists a type $n-1$ invariant $v_{b_1\dots b_m}'$ of $m$-component links such that 
\[v(L_{b_1\dots b_m})=p_1^{k_1}\cdots p_m^{k_m}v(L)+v_{b_1\dots b_m}'(L)\]
for every $m$-component link $L$.

\S\ref{fti-cabling-section} contains a necessary and sufficient condition, in terms of unitrivalent diagrams, for a type $n$ invariant 
of $m$-component links to be $(k_1,\dots,k_m)$-braidable modulo type $n-1$ invariants.
In the case $m=1$ this criterion was essentially proved (but not stated) in \cite{KSA} by means of constructions which have some significance 
for the theory of finite type invariants of knots.
In \S\ref{fti-cabling-section} we briefly review these constructions and comment on how they extend to the case of links.

As a consequence of our criterion, we get

\begin{mainthm} \label{fti-cabling-corollary}
Let $v$ be a $\Q$-valued type $n$ invariant of $m$-component links.
In parts (a) and (c), assume further that it is not a type $n-1$ invariant.

(a) If $v$ is $(k_1,\dots,k_m)$-braidable modulo type $n-1$ invariants, then $k_1+\dots+k_m\le 2n$ and each $k_i\le n$.

(b) When $k_1+\dots+k_m=2n$, the invariant $v$ is $(k_1,\dots,k_m)$-braidable modulo type $n-1$ invariants if and only if 
$v=L+v'$, where $v'$ is a type $n-1$ invariant and $Lz_1^{k_1}\cdots z_m^{k_m}$ is a polynomial in 
$l_{ij}z_iz_j$, the $l_{ij}$ being the pairwise linking numbers.

(c) When $m=1$, the invariant $v$ is $n$-braidable modulo type $n-1$ invariants if and only if $v=M+v'$, where $v'$ 
is a type $n-1$ invariant and $Mz^n$ is a polynomial in the terms $c_iz^{2i}$ of the Conway polynomial 
(in particular, $n$ must be even).
\end{mainthm}

Let us call a type $n$ invariant $v$ of $m$-component links {\it weakly $(k_1,\dots,k_m)$-cableable} if
\[v(L_{p_1,1;\,\cdots;\,p_m,1})=p_1^{k_1}\cdots p_m^{k_m}v(L)\] 
for every $p_1,\dots,p_m\in\Z\but\{0\}$ and for every $m$-component link $L$.

\begin{corollary}
Let $v$ be a non-zero weakly $(k_1,\dots,k_m)$-cableable type $n$ invariant.
Then 

(a) $k_1+\dots+k_m\le 2n$ and each $k_i\le n$;

(b) if $k_1+\dots+k_m=2n$, then $vz_1^{k_1}\cdots z_m^{k_m}$ is a polynomial in $l_{ij}z_iz_j$, 
the $l_{ij}$ being the pairwise linking numbers.
\end{corollary}

\begin{proof}[Proof. (a)]
If $k_1=\dots=k_m=0$, there is nothing to prove, since $n\ge 0$ is tacitly assumed when speaking of type $n$ invariants.
So we may assume that some $k_i\ne 0$.
Then by considering $p_i=2$ and $p_j=1$ for $j\ne i$ we see that $v$ cannot be a non-zero constant.
Since $v$ is known to be non-zero, we get that it is not a type $0$ invariant.
Hence there exists an $n'\le n$ such that $v$ is of type $n'$ and not of type $n'-1$.
Since $v$ is weakly $(k_1,\dots,k_m)$-cableable, it is $(k_1,\dots,k_m)$-braidable modulo type $n'-1$ invariants.
Then by Theorem \ref{fti-cabling-corollary}(a) $k_1+\dots+k_m\le 2n'\le 2n$ and each $k_i\le n'\le n$.
\end{proof}

\begin{proof}[(b)]
Since $v$ is weakly $(k_1,\dots,k_m)$-cableable, it is $(k_1,\dots,k_m)$-braidable modulo type $n-1$ invariants.
Let $L$ and $v'$ be given by Theorem \ref{fti-cabling-corollary}(b) and assume that $v'$ is non-zero.
Then $L$ is easily seen to be weakly $(k_1,\dots,k_m)$-cableable (see Example \ref{lk}).
Hence $v'=v-L$ is also weakly $(k_1,\dots,k_m)$-cableable.
If $k_1=\dots=k_m=0$, there is nothing to prove, since a constant is a polynomial with free term only.
So we may assume that some $k_i\ne 0$.
Then arguing as in the proof of (a), we get that $v'$ is not a type $0$ invariant.
Hence there exists an $n'\le n-1$ such that $v'$ is of type $n'$ and not of type $n'-1$.
Since $v'$ is weakly $(k_1,\dots,k_m)$-cableable, it is $(k_1,\dots,k_m)$-braidable modulo type $n'-1$ invariants.
Then by Theorem \ref{fti-cabling-corollary}(a) $k_1+\dots+k_m\le 2n'\le 2(n-1)$, which contradicts the hypothesis.
\end{proof}

Since a $k$-cableable invariant is weakly $(k,\dots,k)$-cableable, we get

\begin{corollary}\label{lower-bound} 
There are no nonzero $k$-cableable invariants of $m$-component links of type $<\frac{km}2$;
and every such invariant of type $\frac{km}2$ is a polynomial in the pairwise linking numbers.
\end{corollary}

Let us recall in this connection that the invariant $\bar{\bar\omega}$ of Theorem \ref{main-1} is 
an $(m+1)$-cableable invariant of $m$-component links, $m\ge 3$, which is not a function of the pairwise linking numbers.
It is of type $1+\frac{m(m+1)}2$, and by Corollary \ref{lower-bound} there is no such invariant 
of type $\le\frac{m(m+1)}2$.

\begin{remark}
It would be interesting to prove for links something of the sort that Theorem \ref{fti-cabling-corollary}(c) says about knots.
See \cite{Mell}, \cite{MT}, \cite{MV} for some related work.
\end{remark}

\section{Conway potential function} \label{cpf-section}

Let us discuss the Conway potential function $\Omega_L(x_1,\dots,x_m)$ of the $m$-component link $L$ in a bit more detail here.
We have $\Omega_L(x_1,\dots,x_m)\in\Z[x_1^{\pm1},\dots,x_m^{\pm1}]$ for $m>1$ and $(x-x^{-1})\Omega_K(x)\in\Z[x^{\pm1}]$
for a knot $K$. 
If the components of $L$ are colored in $n$ colors according to a coloring function $\chi\:\{1,\dots,m\}\to\{1,\dots,n\}$, 
then $\Lambda:=(L,\chi)$ is called a {\it colored link}, or more specifically an {\it $n$-colored link}, and $\Omega_\Lambda(x_1,\dots,x_n)$ 
denotes $\Omega_L(x_{\chi(1)},\dots,x_{\chi(m)})$.
Thus we have precisely one variable for each color. 
Moreover, identifying colors corresponds to equating the corresponding variables;
that is, if $\Lambda'=(L,q\chi)$ for some map $q\:\{1,\dots,n\}\to\{1,\dots,n'\}$, then $\Omega_{\Lambda'}(x_1,\dots,x_{n'})=\Omega_\Lambda(x_{q(1)},\dots,x_{q(n)})$.

\begin{lemma} \label{4.1} {\rm (compare \cite{Con}, \cite{Ki})}
Let $\Lambda$ be an $n$-colored link with $m$ components.
 
(a) $\Omega_\Lambda(x_1,\dots,x_n)=\Omega_\Lambda(-x_1^{-1},\dots,-x_n^{-1})$.

(b) For $m>1$ the total degree of every nonzero term of $\Omega_\Lambda$ has the same parity as $m$.

(c) For $m>1$ the exponent of $x_i$ in every nonzero term of $\Omega_\Lambda(x_1,\dots,x_n)$ has the same parity as 
$l_i^\Lambda+m_i^\Lambda$, where $m_i^\Lambda$ is the number of components of the $i$-colored sublink $L_i$
of $\Lambda$ and $l_i^\Lambda=\lk(L_i,\,\Lambda\but L_i)$.
\end{lemma}

The following proof is based on Hartley's definition of $\Omega_\Lambda$ \cite{Ha}, which we do not review here.
An alternative proof of Lemma \ref{4.1}, based on Cimasoni's construction of $\Omega_\Lambda$, can be found in 
\cite{M24-2}.
%\cite{M24-2}*{Corollary \ref{con:4.1'}}.

\begin{proof}[Proof. (c)]
It suffices to prove the assertion in the case where all the components of $\Lambda$ have distinct colors (so each $m_i^\Lambda=1$).
In this case Hartley \cite{Ha} defines $\Omega_L(x_1,\dots,x_m)$ as a normalized version of his sign-refined Alexander polynomial 
$\Delta_L(t_1,\dots,t_m)$, which is well-defined up to multiplication by monomials $t_1^{i_1}\dots t_m^{i_m}$.
Namely, according to \cite{Ha}*{(1.1)}, $\Omega_K(x)=x^{\mu}\Delta_K(x^2)/(x-x^{-1})$ for a knot $K$, and for a link $L$ with $m>1$ components, 
$\Omega_L(x_1,\dots,x_m)=x_1^{\mu_1}\dots x_m^{\mu_m}\Delta_L(x_1^2,\dots,x_m^2)$, where the integers $\mu,\mu_1,\dots,\mu_m$ 
are uniquely determined by the symmetry relation $\Omega_L(x_1,\dots,x_m)=(-1)^m\Omega_L(x_1^{-1},\dots,x_m^{-1})$
of \cite{Ha}*{(5.5)}.
Let us note that this symmetry relation is equivalent to 
$\Delta_L(t_1^{-1},\dots,t_m^{-1})=(-1)^m t_1^{\mu_1}\dots t_m^{\mu_m}\Delta_L(t_1,\dots,t_m)$, where each $t_i=x_i^2$
and where the sign of $\Delta_L$ is irrelevant.
In this form it is proved in \cite{TF}*{Corollary ~3} along with the addendum that the parity of each $\mu_i$ is opposite to that of $l_i^\Lambda$.

For an alternative proof that the parity of each $\mu_i$ is opposite to that of $l_i^\Lambda$ directly from the definition of $\mu_i$ 
in \cite{Ha}*{(2.4)} we observe that this definition, which depends on the choice of a plane diagram of $L$, implies that $\mu_i$ 
has the same parity as $l_i^\Lambda+\delta_i+\sigma_i$, where $\sigma_i$ is the number of positively oriented Seifert circles minus 
the number of negatively oriented Seifert circles of the plane diagram $D_i$ of the $i$th component of $L$ and $\delta_i$ is the number 
of double points of the plane curve $D_i$.
Now the parity of $\delta_i$ is opposite to that of the turning number of $D_i$ \cite{Wh}*{Theorem 2}, and the turning number
of $D_i$ is easily seen to be equal to $\sigma_i$.
\end{proof}

\begin{proof}[(b)] This follows easily from (c).
\end{proof}

\begin{proof}[(a)]
For links with $>1$ components this follows from (b) and from the symmetry relation 
$\Omega_\Lambda(x_1,\dots,x_n)=(-1)^m\Omega_\Lambda(x_1^{-1},\dots,x_n^{-1})$, which as explained in the proof of (c) 
is a consequence of \cite{TF}*{Corollary ~3} (see also \cite{Ha}*{(5.5)}).
For knots the assertion follows from $\Omega_K(x)=\nabla_K(x-x^{-1})/(x-x^{-1})$.
\end{proof}

\section{Conway potential function expanded in Conway's variables} \label{mho-section}

In the one-variable case $\Omega_\Lambda(x)$ can be expressed in the form $(x-x^{-1})^{-1}\nabla_\Lambda(x-x^{-1})$, 
where $\nabla_\Lambda(z)\in\Z[z]$ is called the {\it Conway polynomial}.
This follows immediately%
\footnote{If $\Omega'=(x-x^{-1})\Omega\in\Z[x^{\pm1}]$ satisfies $\Omega'(x)=\Omega'(-x^{-1})$, then its terms come 
in pairs $a_kx^k+(-1)^ka_kx^{-k}$, which can be expressed as $a_k(x-x^{-1})^k$ up to terms of lower (in absolute value) degrees.}
from the symmetry $\Omega_\Lambda(x)=\Omega_\Lambda(-x^{-1})$.
As is well-known, $\nabla_\Lambda$ is very easy to deal with, as it is fully determined by two simple axioms 
($\nabla_{\includegraphics[width=0.4cm]{1+p.pdf}}(z)-\nabla_{\includegraphics[width=0.4cm]{1-p.pdf}}(z)=z\,\nabla_{\includegraphics[width=0.4cm]{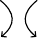}}(z)$
and $\nabla_{\includegraphics[height=0.4cm]{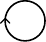}}(z)=1$)
which imply in particular that its coefficients are finite type invariants (see e.g.\ \cite{M24-1}*{\S\ref{fti:conway}}).

In the case of $n>1$ variables, $\Omega_\Lambda(x_1,\dots,x_n)$ still satisfies the symmetry 
\[\Omega_\Lambda(x_1,\dots,x_n)=\Omega_\Lambda(-x_1^{-1},\dots,-x_n^{-1})\]
(see Lemma \ref{4.1}(a)) but cannot be expressed%
\footnote{For instance, if $U$ denotes the unknot and $U^n$ its $(1,n)$-cable (so that $\lk(U,U^n)=n$), then
$\Omega_{(U,U^2)}(x,y)=xy+x^{-1}y^{-1}$ and $\Omega_{(U,U^{-2})}(x,y)=-xy^{-1}-x^{-1}y$, which are not polynomials 
in $x-x^{-1}$ and $y-y^{-1}$ since they are not invariant under the substitution $x\mapsto x$, $y\mapsto-y^{-1}$.}
as a polynomial in $x_1-x_1^{-1},\,\dots,\,x_n-x_n^{-1}$ in general.
However, we may treat $z=x-x^{-1}$ as a quadratic equation in $x$, select one of the two roots $x(z)=\frac{z}2\pm\sqrt{1+\frac{z^2}4}$ and
expand the radical according to Newton's binomial formula $(1+t)^r=1+rt+\frac{r(r-1)}2t^2+\dots$. 
This formula is understood here in the context of {\it formal} power series --- which actually seems to be how it was understood by Newton himself (for a detailed treatment see \cite{Ni}, where $r\in\Q$; see also \cite{Sam} for the more general case $r\in\mathbb C$, which we 
do not need here).
Thus each $x(z_i)$ becomes a formal power series in $z_i$ (with the same choice of the root for each $i$) which is invertible 
(since its free term is $\pm1$), and so can be substituted for $x_i$ in $\Omega_\Lambda(x_1,\dots,x_n)$.
Thus we get $\Omega_\Lambda\big(x(z_1),\dots,x(z_n)\big)=\mho_\Lambda(z_1,\dots,z_n)$ for some $\mho_\Lambda\in\Q[[z_1,\dots,z_n]]$, as long as $n>1$.
In the case $n=1$ we set $\mho_\Lambda(z)=z^{-1}\nabla_\Lambda(z)$, so that $\Omega_\Lambda(x)=(x-x^{-1})^{-1}\nabla_\Lambda(x-x^{-1})=\mho_\Lambda(x-x^{-1})$.
Since identifying colors corresponds to equating the corresponding variables for $\Omega$ (see \S\ref{cpf-section}), the same is true of $\mho$,
including the case $n=1$.

The $2$-variable power series $\mho_\Lambda$ appears in the literature already in the early 90s \cite{HK}, \cite{Kai} under the name 
``two-variable Conway polynomial''.
We refrain from using this terminology, not only because $\mho_\Lambda$ is not a polynomial (to be honest), but also because there is 
another object which better deserves this name (see \cite{M24-2}).

\begin{theorem}\label{mho} The following holds for any $n$-colored link $\Lambda$ with $m$ components.

(a) For $n>1$, both choices of the root $x(z)$ lead to the same formal power series $\mho_\Lambda$.

(b) The coefficients of $\mho_\Lambda$ are finite type invariants; specifically, the coefficient at a term 
of total degree $k$ is of type $k+1$.

(c) The total degree of every nonzero term of $\mho_\Lambda$ has the same parity as $m$.

(d)  For $n>1$, $\mho_\Lambda(4y_1,\dots,4y_n)\in\Z[[y_1,\dots,y_n]]$.
\end{theorem}

\begin{proof}[Proof. (a)] The Galois group of the quadratic equation $z=x-x^{-1}$ acts on its roots by $x\mapsto-x^{-1}$ 
(since the action defined by this formula takes the equation to itself, or alternatively by using Vieta's formula).
Now the assertion follows from Lemma \ref{4.1}(a).
\end{proof}

\begin{proof}[(b)] First we use (a) to ensure that the free term of $x(z)$ is $1$ rather than $-1$.
Now the assertion is proved similarly to the result of Dynnikov \cite{Dy} and H. Murakami \cite{MuH}, who show that 
the coefficients of the power series $\Omega_\Lambda\big(f_1(h_1),\dots,f_n(h_n)\big)$, where each $f_i(h)=e^{h/2}$, are 
finite type invariants of $\Lambda$.
In fact the argument of \cite{MuH}*{proof of Lemma 3.2} works to show the same for arbitrary formal power series $f_1,\dots,f_n$ 
such that $f_i(0)=1$.
(The argument of \cite{Dy} is very similar, but less detailed.)
It must be noted that there is a misprint in the proof of \cite{MuH}*{proof of Lemma 3.2}: apart from the last row, which is correct, 
all $*$'s in the matrix $\tilde D^{(ij)}(\mathcal L^d)$ must be $0$'s.
\end{proof}

\begin{proof}[(c)] Let $x_i=x(z_i)$. 
Since $-(x-x^{-1})=x^{-1}-x$, we have $x(-z_i)=x_i^{-1}$. 
Hence
\begin{multline*}
\mho_\Lambda(-z_1,\dots,-z_n)=\Omega_\Lambda(x_1^{-1},\dots,x_n^{-1})=\Omega_\Lambda(-x_1,\dots,-x_n)\\
=(-1)^m\Omega_\Lambda(x_1,\dots,x_n)=(-1)^m\mho_\Lambda(z_1,\dots,z_n),
\end{multline*}
where the second and the third equalities follow from parts (a) and (b) of Lemma \ref{4.1}, respectively.
\end{proof}

\begin{proof}[(d)] We have $x(z)=\frac{z}2\pm\sqrt{1+\frac{z^2}4}$ and $(1+t)^r=\sum_{k=0}^\infty\dfrac{r(r-1)\cdots(r-k+1)}{k!}t^k$.
Hence $x(4y)=2y\pm\sqrt{1+4y^2}$.
To prove that the power series expansion of $(1+4s)^{1/2}$ has integer coefficients, it suffices to do the same for
\[(1+4s)^{-1/2}=\sum_{k=0}^\infty2^k\frac{(-1)(-3)\cdots(1-2k)}{k!}s^k
=\sum_{k=0}^\infty(-2)^k\frac{1\cdot 3\cdot 5\cdots(2k-1)}{k!}s^k.\]
We have $1\cdot 3\cdot 5\cdots(2k-1)=\dfrac{(2k-1)!}{2\cdot 4\cdot 6\cdots(2k-2)}=\dfrac{(2k-1)!}{2^{k-1}(k-1)!}$.
Hence \[(1+4s)^{-1/2}=1+2\sum_{k=1}^\infty(-1)^k\binom{2k-1}{k}s^k.\]
\end{proof}

\begin{remark} Similarly to the proof of (d) we have \[(1+4s)^{1/2}=1+4s\sum_{k=1}^\infty\frac{(-1)^k}{k+1}\binom{2k-1}{k}s^k,\]
but the integrality of the latter power series may be not obvious from the way it is written, as it is known that 
$(k+1)\nmid\binom{2k-1}k$ for $k=1,3,7$.
However, from the proof of (d) we know that in reality its coefficients are not only integral, but also divisible by $2$, apart from 
the very first one.
Hence we see that $(k+1)\mid2\binom{2k-1}k$.
The referee adds to this that $\frac2{k+1}\binom{2k-1}k=\frac1{k+1}\binom{2k}k$ are actually the Catalan numbers.
\end{remark}

\section{Lucas and Fibonacci polynomials}

Let $F_n$ and $L_n$, $n\in\Z$, be the Fibonacci and Lucas polynomials, defined by the same recursion 
with different initial values:
\[\begin{cases} F_{n+1}(z)=zF_n(z)+F_{n-1}(z),\\
F_0(z)=0,\\ F_1(z)=1; \end{cases}\qquad\qquad
\begin{cases}L_{n+1}(z)=zL_n(z)+L_{n-1}(z),\\
L_0(z)=2,\\ L_1(z)=z. \end{cases}\]

\begin{lemma} \label{lucas-fibonacci}
Let $z=x-x^{-1}$. 
Then 
\[(a)\ L_n(z)=\begin{cases}x^n-x^{-n},&\text{if $n$ is odd},\\
x^n+x^{-n},&\text{if $n$ is even};\end{cases}
\qquad
(b)\ F_n(z)=\begin{cases}\dfrac{x^n+x^{-n}}{x+x^{-1}},&\text{if $n$ is odd},\\
\dfrac{\vphantom{\int\limits^y}x^n-x^{-n}}{x+x^{-1}},&\text{if $n$ is even}.\end{cases}\]
\end{lemma}

This is well-known, but we include a short proof for the reader's convenience.

\begin{proof}[Proof. (a)] Let $\tilde L_n(z)=x^n+(-1)^nx^{-n}$, where $x$ is a fixed solution of the quadratic equation $z=x-x^{-1}$.
Then it is easily checked that $\tilde L_0(z)=2$, $\tilde L_1(z)=z$ and $\tilde L_{n+1}(z)=z\tilde L_n(z)+\tilde L_{n-1}(z)$.
Hence $\tilde L_n\in\Z[z]$ and $\tilde L_n=L_n$.
\end{proof}

\begin{proof}[(b)]
Similarly, setting $\tilde F_n(z)=\big(x^n-(-1)^nx^{-n}\big)/(x+x^{-1})$, we get that $\tilde F_n=F_n$.
\end{proof}

Let $\Lambda_n(z)=\begin{cases}
L_n(z),&\text{if $n$ is odd}\\
\bar zF_n(z),&\text{if $n$ is even}
\end{cases}$ and 
$\Phi_n(z)=\begin{cases}
\bar zF_n(z),&\text{if $n$ is odd}\\
L_n(z),&\text{if $n$ is even,}
\end{cases}$
\smallskip

\noindent
where $\bar z=\sqrt{z^2+4}$ (the positive root).
Like before, the root is understood as an element of $\Q[[z]]$, but in fact, much of what follows also makes sense in the smaller ring $\Z[z][\sqrt{z^2+4}]$.

\begin{proposition} Let $z=x-x^{-1}$. 
Then 
\smallskip

(a) $x^n-x^{-n}=\Lambda_n(z)$,
\smallskip

(b) $x^n+x^{-n}=\Phi_n(z)$.
\end{proposition}

\begin{proof} This follows from Lemma \ref{lucas-fibonacci} since, writing $\bar z=x+x^{-1}$, we have $\bar z^2=z^2+4$.
\end{proof}

\begin{corollary} (a) $\Phi_n(z)^2=\Lambda_n(z)^2+4$.
\smallskip

(b) $\Lambda_{pq}(z)=\Lambda_p\big(\Lambda_q(z)\big)$.
\smallskip

(c) $\Phi_{pq}(z)=\Phi_p\big(\Lambda_q(z)\big)$.
\end{corollary}

\begin{lemma} \label{expansion}
$\Lambda_n(z)=nz+\frac{n^3-n}{24}z^3+\dots$ and $\Phi_n(z)=2+\frac{n^2}4z^2+\dots$.
\end{lemma}

\begin{proof} It is well-known that for $n\ge 1$
\[F_n(z)=\sum_{k=0}^{\lfloor\frac{n-1}2\rfloor}\binom{n-1-k}{k}z^{n-1-2k}, \qquad\qquad
L_n(z)=\sum_{k=0}^{\lfloor\frac n2\rfloor}\frac{n}{n-k}\binom{n-k}{k}z^{n-2k}.\]
Hence
\[F_n=\begin{cases}
\frac{n}2z+\frac{n^3-4n}{48}z^3+\dots,&\text{$n$ even}\\
1+\frac{n^2-1}{8}z^2+\dots,&\text{$n$ odd,}
\end{cases}\qquad\qquad
L_n=\begin{cases}
2+\frac{n^2}{4}z^2+\dots,&\text{$n$ even}\\
nz+\frac{n^3-n}{24}z^3+\dots,&\text{$n$ odd.}
\end{cases}\]
Now $\bar z=2\sqrt{1+\frac{z^2}4}=2\Big(1+\frac12\cdot\frac{z^2}4+\dots\Big)$, and the assertion follows.
\end{proof}

Hereafter we will use the notation $[n]=\{1,\dots,n\}$.

\begin{lemma} \label{decomposition}
$\displaystyle x_1^{k_1}\cdots x_n^{k_n}=\frac{1}{2^n}\sum_{S\subset[n]}\prod_{s\in S}(x_s^{k_s}-x_s^{-k_s})\prod_{s\in[n]\but S}(x_s^{k_s}+x_s^{-k_s})$.
\end{lemma}

\begin{proof} We have $x_s^{k_s}=\frac12\big((x_s^{k_s}+x_s^{-k_s})+(x_s^{k_s}-x_s^{-k_s})\big)$.
Hence \[\prod\limits_{s=1}^n x_s^{k_s}=\frac{1}{2^n}\prod\limits_{s=1}^n\big((x_s^{k_s}+x_s^{-k_s})+(x_s^{k_s}-x_s^{-k_s})\big).\]
\end{proof}

Let us denote the cardinality of a set $S$ by $|S|$. Let 
\begin{align*}\Lambda_{k_1,\dots,k_n}(z_1,\dots,z_n)&=\frac{1}{2^{n-1}}\sum\limits_{\substack{S\subset[n],\\ |S|\text{ \rm odd}}}
\prod\limits_{s\in S}\Lambda_{k_s}(z_s)\prod\limits_{s\in[n]\but S}\Phi_{k_s}(z_s),\\
\Phi_{k_1,\dots,k_n}(z_1,\dots,z_n)&=\frac{1}{2^{n-1}}\sum\limits_{\substack{S\subset[n],\\ |S|\text{ \rm even}}}
\prod\limits_{s\in S}\Lambda_{k_s}(z_s)\prod\limits_{s\in[n]\but S}\Phi_{k_s}(z_s).
\end{align*}

\begin{proposition} Let $z_i=x_i-x_i^{-1}$. 
Then 
\smallskip

(a) $x_1^{k_1}\cdots x_n^{k_n}-x_1^{-k_1}\cdots x_n^{-k_n}=\Lambda_{k_1,\dots,k_n}(z_1,\dots,z_n)$,
\smallskip

(b) $x_1^{k_1}\cdots x_n^{k_n}+x_1^{-k_1}\cdots x_n^{-k_n}=\Phi_{k_1,\dots,k_n}(z_1,\dots,z_n)$.
\end{proposition}

\begin{proof} Since $n\mapsto-n$ sends $x^n+x^{-n}$ to itself but reverses the sign of $x^n-x^{-n}$, the summands in the expression of 
Lemma \ref{decomposition} for $x_1^{k_1}\cdots x_n^{k_n}$ equal those for $x_1^{-k_1}\cdots x_n^{-k_n}$
up to a sign. 
The sign is positive precisely when $|S|$ is even.
\end{proof}

\begin{corollary} (a) $\Phi_{k_1,\dots,k_n}(z_1,\dots,z_n)^2=\Lambda_{k_1,\dots,k_n}(z_1,\dots,z_n)^2+4$.
\smallskip

(b) $\Lambda_{p_1q_1,\dots,p_nq_n}(z_1,\dots,z_n)=\Lambda_{q_1,\dots,q_n}\big(\Lambda_{p_1}(z_1),\dots,\Lambda_{p_n}(z_n)\big)$
\smallskip

(c) $\Phi_{p_1q_1,\dots,p_nq_n}(z_1,\dots,z_n)=\Phi_{q_1,\dots,q_n}\big(\Lambda_{p_1}(z_1),\dots,\Lambda_{p_n}(z_n)\big)$.
\end{corollary}

\section{Terms of the lowest degree}

\begin{theorem}[Torres--Hartley \cite{Ha}*{(5.3)}] \label{torres}
Let $L=(K_1,\dots,K_m,Q)$ be a link, where $m\ge 1$, and let $q_i=\lk(K_i,Q)$. 
Let $L'=(K_1,\dots,K_m)$.
Then \[\Omega_L(x_1,\dots,x_m,1)=(x_1^{q_1}\cdots x_m^{q_m}-x_1^{-q_1}\cdots x_m^{-q_m})\,\Omega_{L'}(x_1,\dots,x_m).\]
\end{theorem}

\begin{corollary} \label{torres'}
Let $L=(K_1,\dots,K_m,Q)$ be a link, where $m\ge 1$, and let $q_i=\lk(K_i,Q)$. 
Let $L'=(K_1,\dots,K_m)$.
Then \[\mho_L(z_1,\dots,z_m,0)=\Lambda_{q_1,\dots,q_m}(z_1,\dots,z_m)\,\mho_{L'}(z_1,\dots,z_m).\]
\end{corollary} 

The following result can be seen to be equivalent to Traldi's and Buryak's formulas for the terms of $\Omega_L(1+v_1,\dots,1+v_m)$ 
of total degrees $\le m-1$ in the $v_i$ \cite{Tr2}*{Theorem 8.2}, \cite{Bu} and to Levine's formula for the terms of 
$\Delta_L(1+u_1,\dots,1+u_m)$ of total degrees $\le m-1$ in the $u_i$, \cite{Le2}*{Corollary 1.6, case $k=2$}.

\begin{theorem} \label{hhh}
Let $L=(K_1,\dots,K_m)$ be a link and let $l_{ij}=\lk(K_i,K_j)$.
Then 
\[\mho_L(z_1,\dots,z_m)=\sum_T\prod_{\{i,j\}\in E(T)}l_{ij}\prod_{v\in V(T)}z_v^{\deg_T(v)-1}+\text{\rm(terms of total degrees $\ge m$)},\tag{$*_m$}\]
where $T$ runs over all spanning trees of the complete graph $\K_m$ with vertices $1,\dots,m$ and $V(T)$ ($=V(\K_m)$) and $E(T)$ denote the sets
of its vertices and edges, respectively.
\end{theorem}

Let us note that $T$ has $m$ vertices and Euler characteristic $1$ (since it is contractible).
Hence it has $m-1$ edges, and consequently the sum of the degrees of its vertices is $2m-2$.
Thus $\prod_{v\in V(T)}z_v^{\deg_T(v)-1}$ has total degree $m-2$ regardless of the choice of $T$.

\begin{proof} Since $\mho_K(z)=z^{-1}+\text{\rm(terms of degrees $\ge 1$)}$ for a knot $K$, the assertion ($*_1$) holds.

Assuming that ($*_{m-1}$) holds, we will prove ($*_m$).
Since every term of total degree $<m$ does not involve at least one of the $m$ variables, it suffices to prove for each $i=1,\dots,m$ 
the assertion $(*_m)|_{z_i=0}$ obtained from ($*_m$) by setting $z_i=0$ on both sides of the equation.
By symmetry it suffices to prove only $(*_m)|_{z_m=0}$.

By Corollary \ref{torres'} $\mho_L(z_1,\dots,z_{m-1},0)=\Lambda_{l_{1m},\dots,l_{m-1,m}}(z_1,\dots,z_{m-1})\,\mho_{L'}(z_1,\dots,z_{m-1})$,
where $L'=(K_1,\dots,K_{m-1})$.
Moreover, by the induction hypothesis
\[\mho_{L'}(z_1,\dots,z_{m-1})=\sum_{T'}\prod_{\{i,j\}\in E(T')}l_{ij}\prod_{v\in V(T')}z_v^{\deg_{T'}(v)-1}+\text{\rm(terms of total degrees $\ge m-1$)},\]
where $T'$ runs over all spanning trees of $\K_{m-1}$.
Since
\[\Lambda_{l_{1m},\dots,l_{m-1,m}}(z_1,\dots,z_{m-1})=l_{1m}z_1+\dots+l_{m-1,m}z_{m-1}+\text{\rm(terms of total degrees $\ge 3$)}\]
and, as noted above, $\prod_{v\in V(T')}z_v^{\deg_{T'}(v)-1}$ has total degree $m-3$ for each $T'$, we get that
\[\mho_L(z_1,\dots,z_{m-1},0)=(l_{1m}z_1+\dots+l_{m-1,m}z_{m-1})\sum_{T'}\prod_{\{i,j\}\in E(T')}l_{ij}\prod_{v\in V(T')}z_v^{\deg_{T'}(v)-1}+\dots
\tag{$*$}\]
up to terms of total degrees $\ge m$.

On the other hand, setting $z_m=0$ in the right hand side of ($*_m$) amounts to considering only those spanning trees $T$ whose vertex $m$ is a leaf.
Such trees are precisely the unions of arbitrary spanning trees of $\K_{m-1}$ and edges of the form $\{i,m\}$.
It follows that the right hand side of $(*_m)|_{z_m=0}$ equals the right hand side of ($*$).
\end{proof}

\begin{corollary}[Hosokawa, Hartley, Hoste] \label{hhh2}
\[\nabla_L(z)=\Bigg(\sum_T\prod_{\{i,j\}\in E(T)}l_{ij}\Bigg)z^{m-1}+\text{\rm(terms of total degrees $\ge m+1$)},\]
where $T$ runs over all spanning trees of $\K_m$.
\end{corollary}

See \cite{Le}*{Proposition 3.2} for a short proof of Corollary \ref{hhh2} and \cite{Masb} for the references on the subject.

\begin{example} (a) If $L$ is a $2$-component link, then
\[\mho_L(u,v)=\lk(L)+\text{\rm(terms of total degrees $\ge 2$)}.\]

(b) If $L=(K_1,K_2,K_3)$ and $l_{ij}=\lk(K_i,K_j)$, then
\[\mho_L(z_1,z_2,z_3)=\!\!\sum_{(i,j,k)\in\langle 3\rangle!}\!\! l_{ij}l_{ik}z_i+\text{\rm(terms of total degrees $\ge 3$)}.\]
\end{example}

\begin{remark} \label{traldi-levine}
It follows from the results of Traldi \cite{Tr2}*{Theorem 8.2} and Levine \cite{Le2}*{Corollary 1.6} that for every $m$-component link $L$ with 
vanishing $\bar\mu$-invariants of length $\le r$ the coefficients of $\mho_L(z_1,\dots,z_m)$ of total degrees $\le r(m-1)-2$ are zero, and 
the coefficients of total degree $r(m-1)-1$ are explicitly determined by the $\bar\mu$-invariants of $L$ of length $r+1$.
In more detail, \cite{Le2}*{Corollary 1.6} asserts that for such links the coefficients of $\Delta_L(1+u_1,\dots,1+u_m)$ of total degrees $\le r(m-1)-2$ 
in the $u_i$ are zero, and the coefficients of total degree $r(m-1)-1$ in the $u_i$ are explicitly determined by the $\bar\mu$-invariants of $L$ 
of length $r+1$ (up to a sign; but the sign is given by \cite{Tr2}*{Theorem 8.2}).
But we have $\mho_L(z_1,\dots,z_m)=\Omega_L(x_1,\dots,x_m)$, where $z_i=x_i-x_i^{-1}$, and 
$\Omega_L(x_1,\dots,x_m)=x_1^{\mu_1}\dots x_m^{\mu_m}\Delta_L(x_1^2,\dots,x_m^2)$ (see the proof of Lemma \ref{4.1}(c)).
Also $u_i=x_i^2-1=z_ix_i=z_ix(z_i)$, where $x(z)$ is a power series in $z$ with constant term $1$, and the assertion follows.
\end{remark}

\section{Conway's first identity}

\begin{proposition}\label{conwayI'}
\[\mho_{\includegraphics[width=0.6cm]{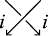}}(z_1,\dots,z_n)-\mho_{\includegraphics[width=0.6cm]{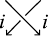}}(z_1,\dots,z_n)
=z_i\,\mho_{\includegraphics[width=0.6cm]{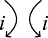}}(z_1,\dots,z_n),\]
where the three $n$-colored links agree outside a PL $3$-ball and behave as shown inside this ball, with the first two related by
a positive intersection between components of color $i$.
\end{proposition}

\begin{proof} This is immediate from the skein relation
\[\Omega_{\includegraphics[width=0.6cm]{1+.pdf}}(x_1,\dots,x_n)-\Omega_{\includegraphics[width=0.6cm]{1-.pdf}}(x_1,\dots,x_n)
=(x_i-x_i^{-1})\,\Omega_{\includegraphics[width=0.6cm]{1o.pdf}}(x_1,\dots,x_n),\]
which is due to Alexander \cite{Al}*{Fig.~2} and Kauffman \cite{Kau} in the one-variable case and to Conway \cite{Con}, 
Cooper \cite{Coo2}*{p.~112} and Hartley \cite{Ha}*{proof of 4.2} in the general case.
\end{proof}

An invariant $v$ of $n$-colored links is said to be of {\it colored type} $k$ \cite{M24-1} 
if its standard extension to singular links, defined by 
$v\big(\raisebox{-3pt}{\includegraphics[width=0.5cm]{1x.pdf}}\big)
=v\big(\raisebox{-3pt}{\includegraphics[width=0.5cm]{1+p.pdf}}\big)-v\big(\raisebox{-3pt}{\includegraphics[width=0.5cm]{1-p.pdf}}\big)$, 
vanishes on all singular links with $>k$ intersections of same-colored components.
In the monochromatic case (where the coloring function is constant) this is the same as being a type $k$ invariant of links (in the usual sense), 
and in the case where all components are colored in distinct colors, this is the same as being a colored type $k$ invariant of (non-colored) 
links.

An invariant $v$ of $n$-colored links is said to be of {\it type} $(k_1,\dots,k_n)$ \cite{M24-1} if its standard extension to singular links 
vanishes on all singular links with $k_1+1$ double points involving only components of color $1$,
on all singular links with $k_2+1$ double points involving only components of color $2$, and so on.
It is easy to see that every type $(k_1,\dots,k_m)$ invariant is a colored type $k_1+\dots+k_m$ invariant; 
and every colored type $n$ invariant is a type $(n,\dots,n)$ invariant.
Thus $v$ is of colored finite type if and only if it is of type $(k_1,\dots,k_m)$ for some $k_1,\dots,k_m$.

\begin{proposition} \label{coloredft} Let $\Lambda$ be an $n$-colored link.

(a) The coefficient of $\mho_\Lambda$ at a term of total degree $k$ is of colored type $k+2-n$.

(b) If $n>1$, then the coefficient of $\mho_\Lambda$ at $z_1^{k_1}\cdots z_n^{k_n}$ is of type $(k_1,\dots,k_n)$.
\end{proposition}

A version of (b) appears in \cite{ShYa}*{Lemma 2.5}.

Let us note that the case $n=2$ in (a) follows from (b).

\begin{proof}[Proof. (b)] Since $n>1$, $\mho_\Lambda$ has no terms of negative degree.
Now the assertion follows by a well-known argument (see \cite{M24-1}*{proof of Lemma \ref{fti:conway-coefficients}})
using Proposition \ref{conwayI'} in place of the skein relation for the Conway polynomial.
\end{proof}

\begin{proof}[(a)] Given an $n$-colored link $\Lambda'$, it must have at least $n$ components, so Theorem \ref{hhh} 
implies that every nonzero term of $\mho_{\Lambda'}$ has total degree at least $n-2$.
Given this, the same argument as in (b) shows that the coefficient of $\nabla_L$ at a term of total degree $n-2+i$ 
is of colored type $i$.
\end{proof}

\section{Two-component links} \label{2-comp-section}

For a link $L=(K_1,\dots,K_m)$ let \[\bar\mho_L(z_1,\dots,z_m)=\frac{\mho_L(z_1,\dots,z_m)}{\nabla_{K_1}(z_1)\cdots\nabla_{K_m}(z_m)}.\]

\begin{proposition} \label{1-reduced}
Let $L=(Q,K_1,\dots,K_m)$ be a link, and let $q_i=\lk(Q,K_i)$. 
Then
\[\bar\mho_L(z,0,\dots,0)=z^{-1}\Lambda_{q_1}(z)\cdots\Lambda_{q_m}(z).\]
\end{proposition}

\begin{proof} By an inductive application of Corollary \ref{torres'}, using that $\Lambda_n(0)=0$ and $\Phi_n(0)=2$
\[\mho_L(z,0,\dots,0)=\Lambda_{q_1}(z)\cdots\Lambda_{q_m}(z)\,\mho_Q(z).\]
Since $\mho_Q(z)=z^{-1}\nabla_Q(z)$, we get
\[\bar\mho_L(z,0,\dots,0)=\frac{\mho_L(z,0,\dots,0)}{\nabla_{Q}(z)\nabla_{K_1}(0)\cdots\nabla_{K_m}(0)}=z^{-1}\Lambda_{q_1}(z)\cdots\Lambda_{q_m}(z).\]
\end{proof}

For a link $L=(K_1,\dots,K_m)$ we also have 
\[\bar\nabla_L(z)=\frac{\nabla_L(z)}{\nabla_{K_1}(z)\cdots\nabla_{K_m}(z)}=z^{m-1}\Big(\alpha(L)+\beta(L)z^2+\dots\Big).\]

\begin{corollary} \label{2-comp}
Let $L$ be a $2$-component link with $l=\lk(L)$.
Then \[\bar\mho_L(u,v)=l+\omega(L)uv+\tfrac{l^3-l}{24}(u^2+v^2)+\text{\rm(terms of total degrees $\ge 4$)},\]
where $\omega(L)=\beta(L)-\frac{l^3-l}{12}$.
\end{corollary}

Let us note that $\omega(L)$ is also the coefficient of $\mho_L(u,v)$ at $uv$, due to the absence of linear terms in $\nabla_{K_i}$.

\begin{proof} By Proposition \ref{1-reduced} $\bar\mho_L(u,0)=l+\tfrac{l^3-l}{24}u^2+\text{\rm(terms of degrees $\ge 4$)}$.
Hence \[\bar\mho_L(u,v)=l+\omega uv+\tfrac{l^3-l}{24}(u^2+v^2)+\text{\rm(terms of total degrees $\ge 4$)}\]
for some $\omega\in\Q$.
Since $\bar\mho_L(z,z)=z^{-1}\bar\nabla_L(z)$, we have $\omega+\tfrac{l^3-l}{12}=\beta(L)$.
\end{proof}

Let us recall Traldi's expansion of $\Omega_L$.
By Lemma \ref{4.1}(c) $(xy)^{\lk(L)-1}\Omega_L(x,y)$ is a Laurent polynomial in even powers of $x$ and $y$.
Expanding $x^{\pm 2}$ and $y^{\pm 2}$ as power series in $s:=1-x^2$ and $t:=1-y^2$ respectively, we get that 
$(xy)^{\lk(L)-1}\Omega_L(x,y)=T_L(s,t)$, where $T_L\in\Z[[s,t]]$.
This is a slightly modified version of the power series of \cite{Tr2}, which in our notation is $T_L(-s,-t)$.
Let $\tau_{i,j}$ be the coefficient of $T_L(s,t)$ at $s^it^j$ and let $\bar\tau_{i,j}$ be its residue class modulo 
the gcd of all $\tau_{k,l}$ with $k\le i$, $l\le j$ and $(k,l)\ne(i,j)$.

\begin{remark}
The coefficients of $T_L$ are of interest due to their relation with $\bar\mu$-invariants.
In particular, let $\bar\mu_{[i,j]}(L)=\bar\mu_{\underbrace{\scriptstyle1\dots1}_i\underbrace{\scriptstyle2\dots2}_j}(L)$
(see \S\ref{milnor} for the definition of $\bar\mu_{i_1\dots i_k}(L)$).
Then by \cite{Tr2}*{Corollary 8.4} $(-1)^{i+j}\bar\tau_{i,j}=(-1)^j\bar\mu_{[i+1,j+1]}(L)$.
Hence $\bar\tau_{i,j}=(-1)^i\bar\mu_{[i+1,\,j+1]}(L)$.
\end{remark}

\begin{proposition} \label{traldi} Let $L$ be a $2$-component link with $l=\lk(L)$.
Then, modulo the ideal $J$ generated by $s^2$, $t^2$ and terms of total degree $\ge 3$,
\[T_L(s,t)=l+\tfrac{l(l-1)}2(s+t)+\tau(L)st+\dots\] 
where $\tau(L)=\beta(L)+\tfrac{l(l-1)(l-2)}6$.
\end{proposition}

\begin{proof}
We have $u^2=(x-x^{-1})^2=x^{-2}(x^2-1)^2=(1-s)^{-1}s^2$ and similarly $v^2=(1-t)^{-1}t^2$.
Hence $u=s(1-s)^{-1/2}=s+\frac12s^2+\frac38s^3+\dots$ and similarly $v=t+\frac12t^2+\frac38t^3+\dots$.
Therefore $uv=st+\dots$ up to terms of total degree $\ge 3$.
Now $\Omega_L(x,y)=\mho_L(u,v)=\nabla_{K_1}(u)\nabla_{K_2}(v)\bar\mho_L(u,v)=l+\omega uv+\dots$ up to multiples of $u^2$, $v^2$ and terms
of total degree $\ge 3$.
Thus $\Omega_L(x,y)=l+\omega st+\dots$ modulo $J$.

Also we have $x=(1-s)^{1/2}=1-\frac12s-\frac18s^2+\dots$ and similarly $y=1-\frac12t-\frac18t^2+\dots$.
Hence $xy=1-\frac12(s+t)+\frac14st+\dots$ modulo $J$.
Then \[(xy)^n\equiv1-\tfrac n2(s+t)+\tfrac n4st+\tfrac{n^2-n}8(s+t)^2+\dots\equiv1-\tfrac n2(s+t)+\tfrac{n^2}4st+\dots\pmod J.\]
Thus $T_L(s,t)=(xy)^{l-1}\Omega_L(x,y)=l+\omega st+\frac{l(l-1)}2(s+t)+\frac{l(l-1)^2}4st+\dots$ modulo $J$.
In particular, we get $\tau(L)=\omega(L)+\tfrac{l(l-1)^2}4=\beta(L)+\frac{l(l-1)(l-2)}6$.
\end{proof}

\begin{corollary} \label{beta-mu} Let $L$ be a $2$-component link with $l=\lk(L)$.
Then
\smallskip

(a) $\bar\mu_{1122}(L)\equiv-\beta(L)-\tfrac{l(l-1)(l-2)}6\pmod{\delta_{1122}(L)}$;
\smallskip

(b) $\delta_{1122}(L)=\gcd\big(l,\frac{l(l-1)}2\big)=\begin{cases}
l&\text{if $l$ is odd,}\\
\frac l2&\text{if $l$ is even.}
\end{cases}$ 
\end{corollary}

\begin{remark} Corollary \ref{beta-mu} corrects the formula of \cite{MR1}*{Appendix}, where the sign in $-\beta(L)$ is missing, 
the additional summand $-\tfrac{l(l-1)(l-2)}6$ is missing, and the modulus is wrongly given as $l$.
But the idea of the proof in \cite{MR1}*{Appendix} is correct, so let us discuss each of these errors.

(a) The modulus error was caused by wrongly assuming that the identity $\bar\mu_{112}(L)=\bar\mu_{122}(L)=0$ for links $L$ of
linking number $0$ carries over to all $2$-component links.
It is easy to check directly that these $\bar\mu$-invariants do not vanish identically for links of nonzero linking numbers 
(see below) and in fact, it is rather well-known that $\bar\mu_{112}(L)=\bar\mu_{122}(L)\equiv\frac{l(l-1)}2\bmod l$ 
(cf.\ \cite{Tr2}*{p.~249} and \cite{Mi2}*{formula (21)}).

(b) The lack of the additional summand $-\tfrac{l(l-1)(l-2)}6$ was caused by wrongly assuming that $\bar\mu_{1122}$ 
vanishes on a certain link $H_l$ satisfying $\lk(H_l)=l$ and $\beta(H_l)=0$, namely, the Hopf link where one component 
is replaced by its $(l,-1)$-cable (see \cite{KL}*{Figure 2, on the left}).
This error is remedied by correctly computing $\bar\mu_{1122}(H_l)$ (note that $\beta(H_l)$ is indeed zero, 
due to $\nabla_{H_l}(z)=lz$, which is easy to check from the skein relation for the Conway polynomial).
Namely, the fundamental group of $H_l$ is generated by two meridians, $x$ and $y$ (these are among the Wirtinger
generators for the link diagram in \cite{KL}), and the corresponding longitudes 
can be written as $\lambda_1=y^lx^{-y^{l-1}}\cdots x^{-y^2}x^{-y}x^{l-1}=(yx^{-1})^{l-1}yx^{l-1}$ and $\lambda_2=x^yx^{y^2}\cdots x^{y^l}=(y^{-1}x)^ly^l$ respectively, using the notation $a^b=b^{-1}ab$ and $a^{-b}=(a^{-1})^b$.
From this one easily finds that 
\begin{align*}
\mu_{1122}(H_l)&=l(l-1)+\dots+3\cdot 2+2\cdot 1=\tfrac{l(l-1)(l+1)}3\equiv-\tfrac{l(l-1)(l-2)}6\pmod l,\\ \mu_{2211}(H_l)&=(l-1)^2+\dots+2^2+1^2=\tfrac{(l-1)l(2l-1)}6\equiv-\tfrac{l(l-1)(l-2)}6\pmod{\tfrac{l(l-1)}2}
\end{align*}
(the latter is to double-check, as $\bar\mu_{1122}=\bar\mu_{2211}$ \cite{Mi2}*{formula (21)}) and incidentally
\begin{align*}
\mu_{112}(H_l)&=\tbinom l2=\tfrac{l(l-1)}2,\\
\mu_{122}(H_l)&=l^2-(1+2+\dots+l-1)=\tfrac{l(l+1)}2\equiv\tfrac{l(l-1)}2\pmod l.
\end{align*}

(c) The origin of the sign error is less obvious, but most likely it crept in when switching from right-normed to
left-normed commutators.
It is easy to check the sign directly: for instance, $\bar\mu_{1122}(W)=-1$ is computed for the Whitehead link $W$ 
already in Milnor's original paper \cite{Mi2}*{Figure 1, case $m=1$} (and one can easily check it) but $\beta(W)=+1$ 
due to $\nabla_{W}(z)=z^3$ (which is easy to check from the skein relation for the Conway polynomial).
In fact, it is rather well-known that $\bar\mu_{1122}(L)=-\beta(L)$ when $\lk(L)=0$ \cite{Co2}*{\S4} (the sign is different
in \cite{Co2} because it uses a different version of $\nabla_L$, as explained in \cite{Tr2}*{\S2 and p.~254}); 
see also \cite{Co1}*{\S9}, \cite{SB}, \cite{Co3}.
\end{remark}

\section{Three-component links} \label{3-comp-section}

For a $2$-component link $L$ let $\ell=\sqrt{\lk(L)}\in [0,\infty)\cup i[0,\infty)$ and let
\[\bar\mho'_L(u,v)=\frac{\bar\mho_L(\ell u,\ell v)}{\ell^2},\] 
where the cancellation of the denominator (which is possible since $\bar\mho_L(0,0)=\ell^2$ and $\bar\mho_L$ has no linear terms) 
is understood to precede the evaluation of $\ell$. 
(This precedence matters when $\ell=0$.)
Thus $\bar\mho'_L(0,0)=1$, the quadratic terms of $\bar\mho'_L$ are the same as those of $\bar\mho_L$, and the remaining
terms of $\bar\mho'_L$ vanish when $\ell=0$.

For a link $L=(K_1,\dots,K_m)$ let $L_{ij}=(K_i,K_j)$, $\ell_{ij}=\sqrt{\lk(L_{ij})}$, $\ell=\prod_{i<j}\ell_{ij}$, and let
\[\bar{\bar\mho}_L(z_1,\dots,z_m)=\frac{\ell^{2-m}\,\bar\mho_L(\ell z_1,\dots,\ell z_m)}
{\prod_{i<j}\bar\mho'_{L_{ij}}(\frac{\ell}{\ell_{ij}}z_i,\,\frac{\ell}{\ell_{ij}}z_j)},\]
where the cancellation of the factor $\ell^{2-m}$ (which is possible since every term of $\bar\mho_L$ is of total degree $\ge m-2$)
and the cancellation of the denominator in each $\frac{\ell}{\ell_{ij}}$ are understood to precede the evaluations of $\ell$ and $\ell_{ij}$.
For brevity we may also write
\[\bar{\bar\mho}_L(z_1,\dots,z_m)=\frac{\ell^{4-m}\,\bar\mho_L(\ell z_1,\dots,\ell z_m)}
{\prod_{i<j}\bar\mho_{L_{ij}}(\ell z_i,\ell z_j)},\]
where the evaluations of $\ell$ and $\ell_{ij}$ are understood to be preceded by the cancellations as described above.
(This precedence matters only in the case $\ell=0$.)

\begin{proposition} \label{2-reduced}
Let $L=(P,Q,K_1,\dots,K_m)$ be a link, $p_i=\lk(P,K_i)$, $q_i=\lk(Q,K_i)$ and $\ell=\sqrt{\lk(P,Q)\prod_i p_iq_i\,\prod_{i<j}\lk(K_i,K_j)}$.
Then
\smallskip

(a) $\bar\mho_L(u,v,0,\dots,0)=\Lambda_{p_1,q_1}(u,v)\cdots\Lambda_{p_m,q_m}(u,v)\,\bar\mho_{(P,Q)}(u,v)$;
\smallskip

(b) $\displaystyle\bar{\bar\mho}_L(u,v,0,\dots,0)
=\lk(P,Q)\prod_i\Bigg(\frac{\bar\Phi_{p_i}(\ell u)}{\bar\Lambda_{p_i}(\ell u)}q_iv
+p_iu\frac{\bar\Phi_{q_i}(\ell v)}{\bar\Lambda_{q_i}(\ell v)}\Bigg)$,
\smallskip

\noindent
where $\bar\Lambda_p(z)=\dfrac{\Lambda_p(z)}{pz}=1+\frac{p^2-1}{24}z^2+\dots$
and $\bar\Phi_p(z)=\dfrac{\Phi_p(z)}{2}=1+\frac{p^2}{8}z^2+\dots$.
\end{proposition}

\begin{proof}
By an inductive application of Corollary \ref{torres'}, using that $\Lambda_n(0)=0$ and $\Phi_n(0)=2$
\[\mho_L(u,v,0,\dots,0)=\Lambda_{p_1,q_1}(u,v)\cdots\Lambda_{p_m,q_m}(u,v)\,\mho_{(P,Q)}(u,v).\]
This implies (a).
In turn, from (a) and Proposition \ref{1-reduced} we get
\begin{align*}
\bar{\bar\mho}_L(u,v,0,\dots,0)&=\frac{\ell^{-m}\,\bar\mho_L(\ell u,\ell v,0,\dots,0)\,\lk(P,Q)\,\prod_i p_iq_i\,\prod_{i<j}\lk(K_i,K_j)}
{\bar\mho_{(P,Q)}(\ell u,\ell v)\,
\prod_i\bar\mho_{(P,K_i)}(\ell u,0)\,\bar\mho_{(Q,K_i)}(\ell v,0)\,\prod_{i<j}\bar\mho_{(K_i,K_j)}(0,0)}\\
&=\lk(P,Q)\frac{\ell^{-m}\,\Lambda_{p_1,q_1}(\ell u,\ell v)\cdots\Lambda_{p_m,q_m}(\ell u,\ell v)}
{\bar\Lambda_{p_1}(\ell u)\bar\Lambda_{q_1}(\ell v)\cdots\bar\Lambda_{p_m}(\ell u)\bar\Lambda_{q_m}(\ell v)\,}.
\end{align*}
This implies (b), using that $\Lambda_{p,q}(x,y)=\Lambda_p(x)\bar\Phi_q(y)+\bar\Phi_p(x)\Lambda_q(y)$.
\end{proof}

For a link $L=(K_1,\dots,K_m)$ with $L_{ij}=(K_i,K_j)$, $\lambda(L)=\prod_{i>j}\lk(L_{ij})$ and $\ell=\sqrt{\lambda(L)}$ we also have 
\[\bar{\bar\nabla}_L(z)=\frac{\ell^{3-m}\,\bar\nabla_L(\ell z)}{\prod_{i<j}(\ell z)^{-1}\bar\nabla_{L_{ij}}(\ell z)},\]
where cancellation again precedes evaluation.
It is easy to see that 
\[\bar{\bar\nabla}_L(z)=z^{m-1}\Big(\alpha(L)+\bar\beta(L)z^2+\dots\Big),\]
where $\displaystyle\bar\beta(L)=\lambda(L)\beta(L)-\alpha(L)\sum_{i<j}\frac{\lambda(L)}{\lk(L_{ij})}\beta(L_{ij})$.
Let us note that $\bar\beta(L)=0$ if two or more of the $\lk(L_{ij})$ vanish.

\begin{corollary} \label{3-comp-barbar}
Let $L=(K_1,K_2,K_3)$ be a $3$-component link with $l_{ij}=\lk(K_i,K_j)$ and
$\lambda=l_{12}l_{23}l_{31}$.
Then
\[\bar{\bar\mho}_L(z_1,z_2,z_3)=\sum_{(i,j,k)\in\langle 3\rangle!} l_{ij}l_{ik}z_i+\sum_{(i,j,k)\in[3]!}l_{ij}l_{jk}\tfrac{2l_{ik}^2+1}{24}\lambda z_i^2z_j+
\bar\omega(L)z_1z_2z_3+\dots\]
up to terms of total degrees $\ge 5$, where $\bar\omega(L)=\bar\beta(L)-\sum\limits_{(i,j,k)\in\langle 3\rangle!}l_{ij}l_{jk}\tfrac{2l_{ik}^2+1}{12}\lambda$.
\end{corollary}

\begin{proof} Let us write $p=l_{13}$, $q=l_{23}$, $r=l_{12}$ and $\ell=\sqrt{\lambda}$. 
By Proposition \ref{2-reduced}(b)
\begin{align*}\bar{\bar\mho}_L(u,v,0)&
=r\Bigg(\frac{\bar\Phi_p(\ell u)}{\bar\Lambda_p(\ell u)}qv+pu\frac{\bar\Phi_q(\ell v)}{\bar\Lambda_q(\ell v)}\Bigg)\\
&=rqv\Big(1+\tfrac{p^2}{8}\ell^2u^2\Big)\Big(1-\tfrac{p^2-1}{24}\ell^2u^2\Big)
+rpu\Big(1+\tfrac{q^2}{8}\ell^2v^2\Big)\Big(1-\tfrac{q^2-1}{24}\ell^2v^2\Big)+\dots
\end{align*}
up to terms of total degrees $\ge 5$.
Hence
\[\bar{\bar\mho}_L(u,v,0)=rpu+rqv+rq\Big(\tfrac{p^2}{8}-\tfrac{p^2-1}{24}\Big)\lambda u^2v+rp\Big(\tfrac{q^2}{8}-\tfrac{q^2-1}{24}\Big)\lambda uv^2+\dots.\]
up to terms of total degrees $\ge 5$.
Therefore
\[\bar{\bar\mho}_L(z_1,z_2,z_3)=\sum_{(i,j,k)\in\langle 3\rangle!} l_{ij}l_{ik}z_i
+\sum_{(i,j,k)\in[3]!}l_{ij}l_{jk}\tfrac{2l_{ik}^2+1}{24}\lambda z_i^2z_j+\bar\omega z_1z_2z_3+\dots\]
up to terms of total degrees $\ge 5$ for some $\bar\omega\in\Q$.
Since $\bar{\bar\mho}_L(z,z,z)=z^{-1}\bar{\bar\nabla}_L(z)$, 
\[\bar\omega+\sum_{(i,j,k)\in[3]!}l_{ij}l_{jk}\tfrac{2l_{ik}^2+1}{24}\lambda=\bar\beta(L).\]
\end{proof}

Let \[\gamma(L)=\beta(L)-\sum_{(i,j,k)\in[3]!}\lk(K_i,K_j)\beta(K_i,K_k),\]
where $[n]!$ denotes the set $\{\sigma(1,\dots,n)\mid\sigma\in S_n\}$ of all ordered $n$-tuples of 
pairwise distinct elements of $[n]=\{1,\dots,n\}$.

\begin{remark}\label{gamma}
The invariant $\gamma(L)$ was noted by the author in 2003 for its remarkably simple crossing change formula
for a self-intersection of $K_i$:
\[\gamma(L_+)-\gamma(L_-)=l_{i'j}l_{jk}l_{ki''}+l_{i''j}l_{jk}l_{ki'},\]
where $(i,j,k)\in\langle3\rangle!$, $l_{ij}=\lk(K_i,K_j)$ and $K_{i'}$, $K_{i''}$ are the lobes of the singular component
\cite{M03}*{p.~11}, \cite{M24-1}*{Proposition \ref{fti:gamma}}.
One consequence of this formula is that neither $\gamma(L)$ nor $\lambda(L)\gamma(L)$ is a function of any invariants of 
proper sublinks of $L$, where $\lambda(L)$ is the product of the pairwise linking numbers of $L$ 
\cite{M24-1}*{Corollary \ref{fti:brunnianity}}.

The original 2005 construction of Akhmetiev's analytic invariant $M$ \cite{A05} was presumably motivated in part by $\gamma$.
While $\gamma$ is only hinted at in \cite{A05}*{\S1.1}, it is explicitly mentioned in most other Akhmetiev's papers dealing with $M$ 
(namely, in \cite{A11}, \cite{A12}, \cite{A14}, \cite{A16}, \cite{A20}, \cite{A21}, \cite{A21'}, \cite{A22}, \cite{A24}, \cite{AS}) 
and it was certainly known to Akhmetiev at the time of writing \cite{A05} (in fact he learned about it from me no later than summer 2003).
\end{remark}

\begin{corollary} \label{3-comp-bar}
Let $L=(K_1,K_2,K_3)$ be a $3$-component link and let $l_{ij}=\lk(K_i,K_j)$ and $\omega_{ij}=\omega(K_i,K_j)$.
Then 
\begin{multline*}
\bar\mho_L(z_1,z_2,z_3)=\sum_{(i,j,k)\in\langle 3\rangle!} l_{ij}l_{ik}\Big(z_i+\tfrac{l_{ij}^2+l_{ik}^2-2}{24}z_i^3\Big)
+\sum_{(i,j,k)\in[3]!}\Big(\omega_{ij}l_{ik}+l_{ij}l_{jk}\tfrac{3l_{ik}^2+l_{ij}^2-1}{24}\Big)z_i^2z_j\\
+\omega(L)z_1z_2z_3+\text{\rm(terms of total degrees $\ge 5$)},
\end{multline*}
where $\omega(L)=\gamma(L)-\frac14\sum\limits_{(i,j,k)\in\langle 3\rangle!} l_{ij}l_{jk}l_{ik}^2$.
\end{corollary}

\begin{proof} Let us write $p=l_{13}$, $q=l_{23}$ and $r=l_{12}$.

By Proposition \ref{2-reduced}(a) $\bar\mho_L(u,v,0)=\Lambda_{p,q}(u,v)\,\bar\mho_{(K_1,K_2)}(u,v)$.
Here
\begin{align*}\Lambda_{p,q}(u,v)&=\tfrac12\big(\Lambda_p(u)\Phi_q(v)+\Phi_p(u)\Lambda_q(v)\big)\\
&=pu\Big(1+\tfrac{p^2-1}{24}u^2\Big)\Big(1+\tfrac{q^2}{8}v^2\Big)
+qv\Big(1+\tfrac{p^2}{8}u^2\Big)\Big(1+\tfrac{q^2-1}{24}v^2\Big)+\dots
\end{align*}
up to terms of total degrees $\ge 5$ and by Corollary \ref{2-comp}
\[\bar\mho_{(K_1,K_2)}(u,v)=r\Big(1+\tfrac{r^2-1}{24}(u^2+v^2)\Big)+\omega_{12}uv+\text{\rm(terms of total degrees $\ge 4$)}.\]
Hence
\begin{align*}\bar\mho_L(u,v,0)=(r+\omega_{12}uv)(pu+qv)
&+pr\Big(\tfrac{r^2-1}{24}+\tfrac{p^2-1}{24}\Big)u^3+pr\Big(\tfrac{r^2-1}{24}+\tfrac{q^2}{8}\Big)uv^2\\
&+qr\Big(\tfrac{r^2-1}{24}+\tfrac{q^2-1}{24}\Big)v^3+qr\Big(\tfrac{r^2-1}{24}+\tfrac{p^2}{8}\Big)u^2v+\dots
\end{align*}
up to terms of total degrees $\ge 5$.
Therefore
\begin{align*}
\bar\mho_L(z_1,z_2,z_3)&=\sum_{(i,j,k)\in\langle 3\rangle!}l_{ij}l_{ik}\Bigg(z_i+\Big(\tfrac{l_{ij}^2-1}{24}+\tfrac{l_{ik}^2-1}{24}\Big)z_i^3\Bigg)\\
&+\sum_{(i,j,k)\in[3]!}\Bigg(\omega_{ij}l_{ik}+l_{ij}l_{jk}\Big(\tfrac{l_{ij}^2-1}{24}+\tfrac{l_{ik}^2}{8}\Big)\Bigg)z_i^2z_j
+\omega z_1z_2z_3+\dots
\end{align*}
up to terms of total degrees $\ge 5$ for some $\omega\in\Q$.
Since $\bar\mho_L(z,z,z)=z^{-1}\bar\nabla_L(z)$ and
\[\sum_{(i,j,k)\in\langle 3\rangle!}l_{ij}l_{ik}\Big(\tfrac{l_{ij}^2-1}{24}+\tfrac{l_{ik}^2-1}{24}\Big)
+\sum_{(i,j,k)\in[3]!}l_{ij}l_{jk}\tfrac{l_{ij}^2-1}{24}=\sum_{(i,j,k)\in[3]!}\big(\beta_{ij}-\omega_{ij}\big)l_{ik},\]
we have 
\[\omega+\sum_{(i,j,k)\in[3]!}\Big(\beta_{ij}l_{ik}+l_{ij}l_{jk}\tfrac{l_{ik}^2}8\Big)=\beta(L).\]
\end{proof}

\begin{remark} In the case where the pairwise linking numbers are all nonzero, $\bar\mho_L$ can be recovered from
$\bar{\bar\mho}_L$.
In this case Corollary \ref{3-comp-bar} can be deduced from Corollaries \ref{2-comp} and \ref{3-comp-barbar}.
By checking that this rather different computation yields the same answer as the proof of Corollary \ref{3-comp-bar} 
the reader might be able to convince himself that there are no errors in the computations of the present section.
(Beware that the definitions of $\bar\beta$ and $\gamma$ involve $\beta_{ij}$ while the formulas 
of Corollaries \ref{2-comp} and \ref{3-comp-bar} involve $\omega_{ij}$.)
\end{remark}

\section{$m$-component links}

For an $m$-component link $L$ let $\omega(L)$ be the coefficient of $\bar\mho_L(z_1,\dots,z_m)$ at $z_1\dots z_m$.

A {\it rooted forest} is a graph whose every component is a rooted tree (that is, a tree with a distinguished vertex)
containing at least one edge.%
\footnote{The convention of excluding the isolated vertices is not really needed for the main statements but is
convenient for the proofs.}

\begin{theorem} \label{forests}
Let $L=(K_1,\dots,K_m)$ be an $m$-component link, for an $S\subset[m]$ let $L_S=(K_{s_1},\dots,K_{s_n})$, 
where $S=\{s_1,\dots,s_n\}$, and let $l_{ij}=\lk(L_{\{i,j\}})$.
Then
\[\omega(L)=\beta(L)-\Bigg(\sum_{S\subsetneqq[m]}\omega(L_S)\sum_F\prod_{\{i,j\}\in E(F)} l_{ij}\Bigg)+\text{\rm (a polynomial in $l_{ij}$)},\]
where $F$ runs over all rooted forests with all roots in $S$ and with the non-roots being precisely all the elements of $[m]\but S$.
\end{theorem}

The summands corresponding to singleton subsets $S$ can be removed from the sum since $\omega(L_{\{v\}})=0$.

\begin{lemma} \label{forests'}
Let $\omega_{i_1\dots i_m}(L)$ be the coefficient of $\bar\mho_L(z_1,\dots,z_m)$ at the monomial $z_1^{i_1}\dots z_m^{i_m}$.
If $i_1+\dots+i_m=m$ and $(i_1,\dots,i_m)\ne (1,\dots,1)$, then
\[\omega_{i_1\dots i_m}(L)=\Bigg(\sum_{S\subsetneqq[m]}\omega(L_S)\sum_F\prod_{\{i,j\}\in E(F)} l_{ij}\Bigg)+\text{\rm (a polynomial in $l_{ij}$)},\]
where $F$ runs over all rooted forests with all roots in $S$, with the non-roots being precisely all the elements 
of $[m]\but S$, and with 
\[\deg_F v=\begin{cases}i_v-1,&\text{if $v$ is a root;}\\
i_v+1,&\text{if $v$ is not a root.}
\end{cases}\]
\end{lemma}

\begin{proof}[Proof of Theorem \ref{forests} modulo Lemma \ref{forests'}] 
Since $\beta(L)$ is the coefficient of $\bar\nabla_L(z)$ at $z^{m+1}$ and $\bar\nabla_L(z)=z\bar\mho_L(z,\dots,z)$,
we have $\beta(L)=\sum_{i_1+\dots+i_m=m}\omega_{i_1\dots i_m}(L)$.
Here $\omega_{1,\dots,1}=\omega(L)$, and each of the remaining terms is given by Lemma \ref{forests'}.
The condition on degrees of vertices in the statement of Lemma \ref{forests'} implies that $(i_1,\dots,i_m)$ can be reconstructed from 
any forest $F$ that appears in the formula of Lemma \ref{forests'} for $\omega_{i_1\dots i_m}(L)$.
So no forest appears in two different instances of this formula, corresponding to distinct tuples $(i_1,\dots,i_m)$.
Thus we may simply add up all these instances together, discarding the condition on degrees of vertices.
This yields the desired formula.
\end{proof}

\begin{proof}[Proof of Lemma \ref{forests'}]
Since $i_1+\dots+i_m=m$, each $i_j\ge 0$ and $(i_1,\dots,i_m)\ne(1,\dots,1)$, at least one $i_j$ must be zero.
Let us choose some specific $j$ with this property.
Then to compute $\omega_{i_1\dots i_m}(L)$ we may use the following relation:
\[\bar\mho_L(z_1,\dots,z_m)|_{z_j=0}=\Lambda_{l_{j1},\dots,\hat l_{jj},\dots,l_{jm}}(z_1,\dots,\hat z_j,\dots,z_m)\,
\bar\mho_{L_{[m]\but\{j\}}}(z_1,\dots,\hat z_j,\dots,z_m)\]
(the hat accent indicates omission), which follows from Corollary \ref{torres'} by dividing both sides by 
$\nabla_{K_1}(z_1)\cdots\widehat{\nabla_{K_j}(z_j)}\cdots\nabla_{K_m}(z_m)$.
Since
\[\Lambda_{l_{j1},\dots,\hat l_{jj},\dots,l_{jm}}(z_1,\dots,\hat z_j,\dots,z_m)=
\sum_{k\in[m]\but\{j\}}l_{jk}z_k+\text{\rm(terms of total degrees $\ge 3$)}\]
and the coefficients of all terms of $\bar\mho_{L_{[m]\but\{j\}}}(z_1,\dots,\hat z_j,\dots,z_m)$ of degrees $\le m-3$ are polynomials in $l_{il}$
(see Theorem \ref{hhh}), we obtain
\[\omega_{i_1\dots i_m}(L)=
\sum_{k\in[m]\but\{j\},\ i_k\ge 1}l_{jk}\omega_{i_1,\dots,i_k-1,\dots,\hat i_j,\dots,i_m}(L_{[m]\but\{j\}})
+\text{(a polynomial in $l_{il}$)}.\]
Now our computation forks into subcomputations, each computing one summand of the latter sum.
If $(i_1,\dots,i_k-1,\dots,\hat i_j,\dots,i_m)$ happens to be $(1,\dots,1)$, then our summand is 
$l_{jk}\omega_{(1,\dots,1)}(L_{[m]\but\{j\}})=l_{jk}\omega(L_{[m]\but\{j\}})$.
The latter expression can be considered to be ``already computed'' for the purposes of proving Lemma \ref{forests'}, so we feel that our job is done 
and stop the subcomputation.
Else at least one of $i_1,\dots,i_k-1,\dots,\hat i_j,\dots,i_m$ must be zero, and in this case we may proceed to compute 
$\omega_{i_1,\dots,i_k-1,\dots,\hat i_j,\dots,i_m}(L_{[m]\but\{j\}})$ in the same fashion as above.
In fact, in order to start doing this we must select some index $j'\ne j$ such that either $j'\ne k$ and $i_{j'}=0$ or $j'=k$ and $i_k-1=0$.
In order to simplify matters we will consider the second case to be preferential; that is, if it is available to us
(i.e.\ if $i_k=1$), then we always go for it.
A successive series of such ``lucky'' choices $j'=k;\,j''=k';\,j'''=k'';\dots;j^{(r)}=k^{(r-1)}$ following the series of selections 
$k,k',\dots,k^{(r-1)}$ of forked subcomputations may end up with one of two outcomes: (i) the last-selected forked subcomputation stops 
like above at $\omega_{(1,\dots,1)}(L_S)$ for some $S\subset [m]$; (ii) it does not stop yet, but the ``lucky'' choice is no longer available 
(that is, $i_{k^{(r)}}\ge 2$).
Either way, the (possibly void) series of successive ``lucky'' choices following our initial choice of $j$ (which was not ``lucky'') 
ends up with the expression
\[l_{jk}l_{j'k'}\cdots l_{j^{(r)}k^{(r)}}
\omega_{i_1,\dots,i_{k^{(r)}}-1,\dots,\hat i_j,\dots,\hat i_{j^{(r)}},\dots,i_m}(L_{[m]\but\{j,j',\dots,j^{(r)}\}}).\]
Let us note that $l_{jk}l_{j'k'}\cdots l_{j^{(r)}k^{(r)}}$ can also be written as $l_{jj'}l_{j'j''}\cdots l_{j^{(r-1)}j^{(r)}}l_{j^{(r)}k^{(r)}}$,
so the factors of this product correspond to the edges of a broken line (i.e.\ a tree homeomorphic to $[0,1]$), namely, the broken line
with vertices $j,j',\dots,j^{(r)},k^{(r)}$.

In the case (ii) we may continue our last-selected forked subcomputation with an ``unlucky'' choice.
Every additional (possibly void) series of ``lucky'' choices following such an ``unlucky'' choice will again yield an additional broken line.
When the computation eventually stops at $\omega_{(1,\dots,1)}(L_S)$ for some $S\subset [m]$, the vertices of the resulting broken lines will cover
all of $[m]\but S$ (since otherwise we would be able to continue the computation starting from any uncovered vertex, which in fact means that it 
could not have stopped).
The initial vertex of each of these broken lines uniquely corresponds to some $j\in [m]$ such that $i_j=0$.
Its terminal vertex corresponds to some $j\in [m]$ such that $i_j\ge 2$; and for every such $j$ there are precisely $i_j-1$ of the broken lines
whose terminal vertices correspond to it.
Each intermediate vertex of each broken line uniquely corresponds to some $j\in [m]\but S$ such that $i_j\ge 1$.
Thus only terminal vertices can be in $S$; and every terminal vertex in $[m]\but S$ must be the same element of $[m]$ as 
some intermediate vertex of some other broken line.

Upon taking the union of the broken lines and declaring those terminal vertices that are in $S$ to be ``roots'', we obtain a rooted forest $F$.
The final expression that our last-selected forked subcomputation ends up with is $\omega(L_S)\prod_{\{i,j\}\in E(F)} l_{ij}$.
If we make our ``unlucky'' choices (including the very first one) in a different order (but continue to go for the ``lucky'' choice whenever possible), 
we may get a different collection of broken lines, but their union will still be the same rooted forest; so in particular we still get the same expression
$\omega(L_S)\prod_{\{i,j\}\in E(F)} l_{ij}$.
Conversely, from $F$ we can reconstruct all its leaves (that is, those degree one vertices that are not roots) and hence our entire
last-selected forked subcomputation, up to reordering the ``unlucky'' choices.
In particular, any subcomputation that forks away from ours in at least one vertex will always lead to a distinct rooted forest --- regardless of whether 
its ``unlucky'' choices were the same as in ours or different.
It can be seen from this that each rooted forest contributes precisely one term to the expression for $\omega_{i_1\dots i_m}(L)$.
\end{proof}

\begin{corollary} \label{forests''}
$\omega(L)=\sum_\Lambda P_\Lambda\beta(\Lambda)+Q$, where $\Lambda$ runs over all sublinks of $L$ and each $P_\Lambda$ as well as $Q$
are polynomials in the pairwise linking numbers of $L$.
\end{corollary}

It is not hard to derive an explicit formula for $P_\Lambda$ from the statement of Theorem \ref{forests}.
One can also obtain an explicit formula for $Q$ similarly to the proof of Theorem \ref{forests}.

\section{Cables} \label{cables section}

\begin{theorem}[Cimasoni \cite{Ci}*{Corollary 3.4}] \label{cimasoni}
Let $L=(K_1,\dots,K_m)$ be a link with $l_{ij}=\lk(K_i,K_j)$ and let $L'$ be obtained from $L$ by replacing some $K_i$ with its $(p,q)$-cable. 
Then \[\Omega_{L'}(x_1,\dots,x_m)=\frac{T^p-T^{-p}}{T-T^{-1}}\,\Omega_L(x_1,\dots,x_{i-1},x_i^p,x_{i+1},\dots,x_m),\]
where $T=x_i^q\prod_{j\ne i}x_j^{l_{ij}}$.
\end{theorem}

It should be mentioned that Seifert proved Theorem \ref{cimasoni} for knots, and Sumners--Woods proved it up to a sign 
(see references in \cite{Ci} or \cite{Tu0}).
See also \cite{BZ}*{Proposition 8.23} for an approach based on the fundamental group.

By an inductive application of Theorem \ref{cimasoni} we obtain

\begin{corollary} \label{cimasoni'}
Let $L=(K_1,\dots,K_m)$ be a link with $l_{ij}=\lk(K_i,K_j)$ and let $L'$ be obtained from $L$ by replacing each $K_i$ with its $(p_i,q_i)$-cable. 
Then \[\Omega_{L'}(x_1,\dots,x_m)=\frac{T_1^{p_1}-T_1^{-p_1}}{T_1-T_1^{-1}}\cdots\frac{T_m^{p_1}-T_m^{-p_1}}{T_m-T_m^{-1}}\,
\Omega_L(x_1^{p_1},\dots,x_m^{p_m}),\]
where each $T_i=x_i^{q_i}\prod_{j\ne i}x_j^{l_{ij}p_j}$.
\end{corollary}

\begin{proof} To illustrate the argument let us discuss the case $m=2$. 
Writing $L'=(K_1',K_2')$ and $l=\lk(K_1,K_2)$, by Theorem \ref{cimasoni}
\[\Omega_{L'}(x,y)=\frac{T_1^{p_1}-T_1^{-p_1}}{T_1-T_1^{-1}}\,\Omega_{(K_1,K_2')}(x^{p_1},y)=
\frac{T_1^{p_1}-T_1^{-p_1}}{T_1-T_1^{-1}}\,\frac{T_2^{p_2}-T_2^{-p_2}}{T_2-T_2^{-1}}\,\Omega_L(x^{p_1},y^{p_2}),\]
where $T_1=x^{q_1}y^{\lk(K_1,K_2')}=x^{q_1}y^{lp_2}$ and $T_2=y^{q_2}(x^{p_1})^l=y^{q_2}x^{lp_1}$.
\end{proof}

\begin{corollary} \label{cimasoni''}
Let $L=(K_1,\dots,K_m)$ be a link and let $L'$ be obtained from $L$ by replacing each $K_i$ with its $(p_i,q_i)$-cable. 
Let $\lambda_{ij}=\begin{cases}\lk(K_i,K_j),&\text{ if $i\ne j$},\\
q_i/p_i,&\text{ if $i=j$}.
\end{cases}$
Then \[\mho_{L'}(z_1,\dots,z_m)=p_1\bar\Lambda_{p_1}(S_1)\cdots p_m\bar\Lambda_{p_m}(S_m)\,
\mho_L\big(\Lambda_{p_1}(z_1),\dots,\Lambda_{p_m}(z_m)\big),\]
where $S_i=\Lambda_{\lambda_{i1}p_1,\dots,\lambda_{im}p_m}(z_1,\dots,z_m)$.
\end{corollary} 

\begin{corollary} \label{cimasoni'''}
Let $K$ be a knot and let $K'$ be its $(p,q)$-cable. 
Then \[\nabla_{K'}(z)=\frac{\bar\Lambda_{pq}(z)}{\bar\Lambda_p(z)\bar\Lambda_q(z)}\,\nabla_K\big(\Lambda_p(z)\big).\]
\end{corollary}

\begin{proof} We have $\nabla_{K'}(z)=z\mho_{K'}(z)$ and similarly $\nabla_K\big(\Lambda_p(z)\big)=\Lambda_p(z)\mho_K\big(\Lambda_p(z)\big)$.
Hence by Corollary \ref{cimasoni''}
\[\nabla_{K'}(z)=zp\bar\Lambda_p\big(\Lambda_q(z)\big)\frac{\nabla_K\big(\Lambda_p(z)\big)}{\Lambda_p(z)}=
\frac{z\Lambda_p\big(\Lambda_q(z)\big)}{\Lambda_q(z)\Lambda_p(z)}\nabla_K\big(\Lambda_p(z)\big),\]
where $\Lambda_p\big(\Lambda_q(z)\big)=\Lambda_{pq}(z)=pqz\bar\Lambda_{pq}(z)$ and $\Lambda_q(z)\Lambda_p(z)=pzqz\bar\Lambda_q(z)\bar\Lambda_p(z)$.
\end{proof}

\begin{proposition} \label{cabling}
Let $K$ be a knot and let $K'$ be its $(p,q)$-cable. 
Then
\[\nabla_{K'}(z)=\Big(1+\tfrac{(p^2-1)(q^2-1)}{24}z^2\Big)\,\nabla_K(pz)+\text{\rm (terms of degrees $\ge 4$)}.\]
\end{proposition}

\begin{proof}  By Lemma \ref{expansion} 
\[\frac{\bar\Lambda_{pq}(z)}{\bar\Lambda_p(z)\bar\Lambda_q(z)}=\frac{1+\tfrac{p^2q^2-1}{24}z^2+\dots}
{\big(1+\tfrac{p^2-1}{24}z^2+\dots\big)\big(1+\tfrac{q^2-1}{24}z^2+\dots\big)}=1+\tfrac{(p^2-1)(q^2-1)}{24}z^2+\dots\]
up to terms of degrees $\ge 4$.
On the other hand,
\[\nabla_K\big(\Lambda_p(z)\big)=\nabla_K\big(pz+\tfrac{p^3-p}{24}z^3+\dots\big)=\nabla_K(pz)+\text{\rm (terms of degrees $\ge 4$)},\]
since $\nabla_K$ involves no terms of degree $1$.
\end{proof}

\begin{corollary} \label{casson} Let $K$ be a knot and let $K'$ be its $(p,q)$-cable. 
Let $c_1(K)$ denote the coefficient of $\nabla_K$ at $z^2$.
Then \[c_1(K')=p^2c_1(K)+\tfrac{(p^2-1)(q^2-1)}{24}.\]
\end{corollary}

\begin{proposition} \label{cabling'}
Let $L=(K_1,\dots,K_m)$ be a link with $l_{ij}=\lk(K_i,K_j)$ and let $L'$ be obtained from $L$ by replacing each $K_i$ with its $(p_i,q_i)$-cable.
Let $I_m$ be the ideal of $\Q[[z_1,\dots,z_m]]$ generated by all monomials of degrees $\ge m+2$ and
by those monomials of degree $m$ that are divisible by $z_i^3$ for some $i$.
\bigskip

(a) \ $\displaystyle\mho_{L'}(z_1,\dots,z_m)\equiv p_1\cdots p_m\Bigg(1+\sum_{i=1}^m\tfrac{p_i^2-1}{24}A_i\Bigg)\,
\mho_L(p_1z_1,\dots,p_mz_m)\pmod{I_m},$
\smallskip

\noindent
where $A_i=\Big(q_iz_i+\sum_{j\ne i}l_{ij}p_jz_j\Big)^2$.
\bigskip

(b) \ $\displaystyle\bar\mho_{L'}(z_1,\dots,z_m)\equiv p_1\cdots p_m\Bigg(1+\sum_{i=1}^m\tfrac{p_i^2-1}{24}B_i\Bigg)\,
\bar\mho_L(p_1z_1,\dots,p_mz_m)\pmod{I_m},$
\smallskip

\noindent
where $B_i=\Big(q_iz_i+\sum_{j\ne i}l_{ij}p_jz_j\Big)^2-(q_i^2-1)z_i^2$.
\end{proposition}
\medskip

\begin{proof}[Proof. (a)] We start from the formula of Corollary \ref{cimasoni''}.
By Lemma \ref{expansion} 
\[S_i=\lambda_{1i}p_1z_1+\dots+\lambda_{mi}p_mz_m+\text{\rm (terms of total degrees $\ge 3$)}\] and
$\bar\Lambda_{p_i}(S_i)=1+\tfrac{p_i^2-1}{24}S_i^2+\dots$
up to terms of degrees $\ge 4$ in $S_i$.
Hence
\[\bar\Lambda_{p_i}(S_i)=1+\tfrac{p_i^2-1}{24}(\lambda_{1i}p_1z_1+\dots+\lambda_{mi}p_mz_m)^2+\dots\]
up to terms of total degrees $\ge 4$ in $z_1,\dots,z_m$.
By Theorem \ref{hhh} $\mho_L$ involves only terms of total degrees $\ge m-2$, so we get
\[\mho_{L'}(z_1,\dots,z_m)=p_1\cdots p_m\Bigg(1+\sum_{i=1}^m\tfrac{p_i^2-1}{24}A_i\Bigg)\,
\mho_L\big(\Lambda_{p_1}(z_1),\dots,\Lambda_{p_m}(z_m)\big)+\dots\]
up to terms of degrees $\ge m+2$ and
\[\mho_L\big(\Lambda_{p_1}(z_1),\dots,\Lambda_{p_m}(z_m)\big)=\mho_L(p_1z_1,\dots,p_mz_m)+\dots\]
up to terms of degrees $\ge m$ that are divisible by $z_i^3$ for some $i$.
\end{proof}

\begin{proof}[(b)] This follows from (a) and Proposition \ref{cabling}.
\end{proof}

\begin{corollary} \label{beta-cable}
Let $L=(K_1,K_2)$ be a $2$-component link and let $l=\lk(L)$.

(a) Let $L'$ be obtained from $L$ by replacing each $K_i$ with its $(p_i,q_i)$-cable.
Then
\[\beta(L')=(p_1p_2)^2\beta(L)+p_1p_2l\Big(q_1p_2l\tfrac{p_1^2-1}{12}+p_1q_2l\tfrac{p_2^2-1}{12}+(1+p_1p_2l^2)\tfrac{p_1p_2-1}{12}\Big).\]

(b) Let $L''$ be a $(p_1,p_2)$-satellite of $L$.
Then
\[\bar\mu_{1122}(L'')\equiv(p_1p_2)^2\bar\mu_{1122}(L)\pmod{\delta_{1122}(L)}.\]
\end{corollary}

\begin{proof}[Proof. (a)] In the notation of Proposition \ref{cabling'},
\[\bar\mho_{L'}(u,v)=p_1p_2\Big(1+\tfrac{p_1^2-1}{24}B_1+\tfrac{p_2^2-1}{24}B_2\Big)\bar\mho_L(p_1u,\,p_2v)+\text{\rm (terms of total degrees $\ge 4$)},\]
where $B_1=(q_1u+lp_2v)^2-(q_1^2-1)u^2$ and $B_2=(lp_1u+q_2v)^2-(q_2^2-1)v^2$.
Hence \[\omega(L')=(p_1p_2)^2\omega(L)+2p_1p_2l^2\Big(q_1p_2\tfrac{p_1^2-1}{24}+p_1q_2\tfrac{p_2^2-1}{24}\Big).\]
The assertion of (a) follows from this and the relation $\omega(L)=\beta(L)-\frac{\lk(L)^3-\lk(L)}{12}$ of Corollary \ref{2-comp}.
\end{proof}

\begin{proof}[(b)] By Proposition \ref{traldi} and Corollary \ref{beta-mu}(a)
$\bar\mu_{1122}(L)\equiv-\tau(L)\pmod{\delta_{1122}(L)}$.
Now $\tau(L)=\omega(L)+\frac{\lk(L)(\lk(L)-1)^2}4$ by the proof of Proposition \ref{traldi}, so
from the proof of (a) we get 
\[\tau(L')=(p_1p_2)^2\tau(L)+\tfrac l2\Big((p_1p_2l^2-1)\tfrac{p_1p_2(p_1p_2-1)}2
+q_1p_2l\tfrac{p_1(p_1+1)(p_1-1)}6+p_1q_2l\tfrac{p_2(p_2+1)(p_2-1)}6\Big).\]
The three fractions are the binomial coefficients, so they are integer.
Hence the contents of the big parentheses is always integer.
Therefore $\tau(L')=(p_1p_2)^2\tau(L)\pmod{l/2}$ when $l$ is even.
When $l$ is odd, it turns out that the contents of the big parentheses is always even (by considering the residues of $p_1$ and $p_2$ modulo 2,
compare Example \ref{beta-ex}, or alternatively using that $\tau(L)$ and $\tau(L')$ are integer).
Hence $\tau(L')=(p_1p_2)^2\tau(L)\pmod{l}$ when $l$ is odd.
Thus we get that $\tau(L')=(p_1p_2)^2\tau(L)\pmod{\delta_{1122}}$ (see Corollary \ref{beta-mu}(b)).
On the other hand, the crossing change formula for $\beta$ (see e.g.\ \cite{M24-1}*{Proposition \ref{fti:beta-jump}}) easily implies that
$\beta(L'')\equiv\beta(L')\pmod l$.
Since $\tau(L)=\beta(L)+\tfrac{l(l-1)(l-2)}6$ (see Proposition \ref{traldi}) and $\lk(L')=\lk(L'')$, we get that
$\tau(L'')\equiv\tau(L')\pmod l$.
It follows that $\tau(L')=(p_1p_2)^2\tau(L)\pmod{\delta_{1122}}$ and consequently
$\bar\mu_{1122}(L'')\equiv(p_1p_2)^2\bar\mu_{1122}(L)\pmod{\delta_{1122}(L)}$.
\end{proof}

\begin{proposition} \label{cabling''}
In the notation of Proposition \ref{cabling'} let $\lambda=\prod_{i<j}l_{ij}$, $p=p_1\cdots p_m$ and 
let $q$ be either of the two square roots of $p^{m-1}$.
Then
\[\displaystyle\bar{\bar\mho}_{L'}(z_1,\dots,z_m)\equiv \frac{p}{q^{m-2}}\Bigg(1+\sum_{i=1}^m\tfrac{p_i^2-1}{24}p^{m-1}\lambda C_i\Bigg)\,
\bar{\bar\mho}_L(p_1qz_1,\dots,p_mqz_m)\pmod{I_m},\]
where $C_i=2\sum\limits_{\{j,k\}\subset[m]\but\{i\}}l_{ij}l_{ik}p_jp_kz_jz_k-(m-2)z_i^2$.
\end{proposition}

Let us note that $C_i$ improves over $A_i$ and $B_i$ in that it does not depend on the $q_i$.

Also, the total degrees of all terms of $\bar{\bar\mho}_L$ have the same parity as $m$, and therefore
none of these terms actually involves a non-integer power of $p$.
Moreover, upon multiplying any of these terms by the factor of $\frac p{q^{m-2}}$, we always get an even power of $q$.
Thus the right hand side does not depend on the choice of $q$ among the two square roots of $p^{m-1}$.

\begin{proof} Let $\ell=\sqrt{\lambda}\in [0,\infty)\cup i[0,\infty)$ and $\ell'=\sqrt{p^{m-1}\lambda}\in [0,\infty)\cup i[0,\infty)$.
Since the formula to be proved is known to be independent of the choice of $q$, it suffices to prove it for just one choice.
Let us choose $q$ to be $\frac{\ell'}{\ell}$.
(Thus this $q$ may belong to $-i[0,\infty)$, and when $m$ is odd, it equals $|p^{(m-1)/2}|$ rather than $p^{(m-1)/2}$.)
Let $L_{ij}=(K_i,K_j)$ and $L'_{ij}=(K_i',K_j')$, where $L'=(K_1',\dots,K_m')$.
Then
\begin{align*}
\displaystyle\bar{\bar\mho}_{L'}(z_1,\dots,z_m)
&=\frac{(q\ell)^{4-m}\,\bar\mho_{L'}(q\ell z_1,\dots,q\ell z_m)}{\prod_{i<j}\bar\mho_{L'_{ij}}(q\ell z_i,\,q\ell z_j)}\\
&\underset{I_m}\equiv q^{4-m}\frac{p}{p^{m-1}}\Bigg(1+\sum_{i=1}^m\tfrac{p_i^2-1}{24}q^2\ell^2(B_i-B'_i)\Bigg)\,
\frac{\ell^{4-m}\,\bar\mho_L(p_1q\ell z_1,\dots,p_mq\ell z_m)}{\prod_{i<j}\bar\mho_{L_{ij}}(p_iq\ell z_i,\,p_jq\ell z_j)}\\
&=\frac{p}{q^{m-2}}\Bigg(1+\sum_{i=1}^m\tfrac{p_i^2-1}{24}p^{m-1}\lambda C_i\Bigg)\,\bar{\bar\mho}_L(p_1qz_1,\dots,p_mqz_m),
\end{align*}
where $B'_i=\sum_{j\ne i}\Big((q_iz_i+l_{ij}p_jz_j)^2-(q_i^2-1)z_i^2\Big)$.
\end{proof}

\begin{theorem} \label{3-comp}
Let $L=(K_1,K_2,K_3)$ be a $3$-component link, and let $L'$ be obtained from $L$ by replacing each $K_i$ 
with its $(p_i,q_i)$-cable.
Let $\bar\omega(L)$ be the coefficient of $\bar{\bar\mho}_L(u_1,u_2,u_3)$ at $u_1u_2u_3$ and let $\bar{\bar\omega}(L)=\bar\omega(L)-\frac{1}{12}\sum_{(i,j,k)\in\langle 3\rangle!}l_{ij}^3l_{ik}^3l_{jk}$,
where $l_{ij}=\lk(K_i,K_j)$.
Then \[\bar{\bar\omega}(L')=(p_1p_2p_3)^4\bar{\bar\omega}(L).\]
\end{theorem}

\begin{proof} By Proposition \ref{cabling''}
\[\bar{\bar\mho}_{L'}(z_1,z_2,z_3)=\Bigg(1+\sum_{i=1}^3\tfrac{p_i^2-1}{24}p^2\lambda C_i\Bigg)\bar{\bar\mho}_L(p_1pz_1,\,p_2pz_2,\,p_3pz_3)
+\dots\]
up to terms of total degrees $\ge 5$ and terms involving $z_i^3$, where $C_i=2l_{ij}l_{ik}p_jp_kz_jz_k-(m-2)z_i^2$, with $j$ and $k$ 
chosen so that $(i,j,k)\in\langle 3\rangle!$.
Since $\bar\omega(L)$ is the coefficient of $\bar{\bar\mho}_L(u_1,u_2,u_3)$ at $u_1u_2u_3$, whereas by Corollary \ref{3-comp-barbar} 
the linear part of $\bar{\bar\mho}_L(u_1,u_2,u_3)$ is $\sum_{(i,j,k)\in\langle 3\rangle!}l_{ij}l_{ik}u_i$, we obtain
\[\bar\omega(L')=p^4\Bigg(\bar\omega(L)+\sum_{(i,j,k)\in\langle 3\rangle!}\tfrac{p_i^2-1}{12}l_{ij}^2l_{ik}^2\lambda\Bigg).\]
Writing $P_i(L)=\frac{\lambda}{12}l_{ij}^2l_{ik}^2$, with $j$ and $k$ chosen so that $(i,j,k)\in\langle 3\rangle!$, it is easy to see that $P_i(L')=p_i^2p^4 P_i(L)$.
Hence $P_i(L')=p^4\Big(P_i(L)+\tfrac{p_i^2-1}{12}l_{ij}^2l_{ik}^2\lambda\Big)$.
Therefore setting $\bar{\bar\omega}(L)=\bar\omega(L)-\sum_{i=1}^3 P_i(L)$, we obtain $\bar{\bar\omega}(L')=p^4\bar{\bar\omega}(L)$.
\end{proof}

For an $m$-component link $L=(K_1,\dots,K_m)$ let
\[\bar{\bar{\bar\mho}}_L(z_1,\dots,z_m)=
\frac{\bar{\bar\mho}_L(z_1,\dots,z_m)}{\prod\limits_{(i,j,k)\in\langle m\rangle^{\underline{\!\langle 3\rangle\!}}}\big(1+\frac{1}{12}l_{ij}l_{ik}\lambda z_jz_k\big)},\]
where $l_{ij}=\lk(K_i,K_j)$, $\lambda=\prod_{i<j}l_{ij}$ and $\langle m\rangle^{\underline{\!\langle 3\rangle\!}}$ denotes the set of all injections 
$\langle 3\rangle\to\langle m\rangle$ that respect the cyclic order.%
\footnote{In other words, triples $(i,j,k)$ of elements of $\{1,\dots,m\}$ such that either $i\<j\<k$ or $j\<k\<i$ or $k\<i\<j$.}
Let $\bar{\bar\omega}(L)$ be the coefficient of $\bar{\bar{\bar\mho}}_L(z_1,\dots,z_m)$ at $z_1\cdots z_m$.

\begin{theorem} \label{cables}
Let $L=(K_1,\dots,K_m)$, where $m\ge 3$, and let $L'$ be obtained from $L$ by replacing each $K_i$ 
with its $(p_i,q_i)$-cable.
Then \[\bar{\bar\omega}(L')=(p_1\cdots p_m)^{m+1}\bar{\bar\omega}(L).\]
\end{theorem}

\begin{proof}
Let $\bar\omega(L)$ be the coefficient of $\bar{\bar\mho}_L(z_1,\dots,z_m)$ at $z_1\cdots z_m$.
Arguing similarly to the proof of Theorem \ref{3-comp}, it is easy to see that, in the notation of Theorem \ref{hhh} and Proposition \ref{cabling''},
\[\bar\omega(L')=p^2q^2\Bigg(\bar\omega(L)+
\lambda\sum_T\prod_{\{v,v'\}\in E(T)}l_{vv'}\sum_{i\notin\partial T}\tfrac{p_i^2-1}{12}\prod_{j\in\partial T}l_{ij}\Bigg),\]
where $T$ runs over those spanning trees of $K_m$ that are homeomorphic to $[0,1]$.
On the other hand, we have
\[\bar{\bar\omega}(L)=\bar\omega(L)-\tfrac1{12}\lambda\sum_T\prod_{\{v,v'\}\in E(T)}l_{vv'}\sum_{i\notin\partial T}\prod_{j\in\partial T}l_{ij},\]
where $T$ keeps its meaning from the previous formula.
Arguing similarly to the proof of Theorem \ref{3-comp}, we get that \[\bar{\bar\omega}(L')=p^2q^2\bar{\bar\omega}(L).\]
\end{proof}

\section{Satellites and crossing changes}

For an $m$-component link $L$ let $\omega(L)$, $\bar\omega(L)$ and $\bar{\bar\omega}(L)$ denote respectively
the coefficients of $\bar\mho_L(z_1,\dots,z_m)$, $\bar{\bar\mho}_L(z_1,\dots,z_m)$ and $\bar{\bar{\bar\mho}}_L(z_1,\dots,z_m)$
at $z_1\cdots z_m$. 

\begin{lemma} \label{omega} 
For an $m$-component link $L=(K_1,\dots,K_m)$ let $l_{ij}=\lk(K_i,K_j)$, $\omega_{ij}=\omega(K_i,K_j)$ and $\lambda=\prod_{i<j}l_{ij}$.
Then
\bigskip

(a) $\displaystyle\bar\omega(L)=\lambda\omega(L)-\sum_{\{j,k\}\subset[m]}\frac{\lambda}{l_{jk}}\omega_{jk}
\sum_{(i_1,\dots,i_{m-2})\in([m]\but\{j,k\})!}l_{ji_1}l_{i_1i_2}\cdots l_{i_{m-3}i_{m-2}}l_{i_{m-2}k}$
\bigskip

(b) $\displaystyle\bar{\bar\omega}(L)=\bar\omega(L)-\tfrac1{12}\lambda\sum_{(i,j,k)\in\langle m\rangle^{\underline{\!\langle 3\rangle\!}}}l_{ij}l_{ik}
\sum_{(i_1,\dots,i_{m-2})\in([m]\but\{j,k\})!}l_{ji_1}l_{i_1i_2}\cdots l_{i_{m-3}i_{m-2}}l_{i_{m-2}k}$
\end{lemma}
\medskip

\begin{proof} This follows easily from Theorem \ref{hhh}, using additionally Corollary \ref{2-comp} for the purposes of proving (a).
Indeed, the coefficient of $\mho_L(z_1,\dots,z_m)$, and hence also of $\bar\mho_L(z_1,\dots,z_m)$, at $z_1\cdots\hat z_j\cdots\hat z_k\cdots z_m$
(the hat accents indicate omission of the letters) is given by the sum of those terms in the formula of Theorem \ref{hhh} that correspond to trees 
with leaves in the vertices $j$ and $k$ and with all other vertices being of degree $2$.
\end{proof}

\begin{corollary} \label{denominator} $2^{m-1}\cdot 3\,\bar{\bar\omega}(L)\in\Z$.
\end{corollary}

\begin{proof} We have $\Omega_L(x_1,\dots,x_m)\in\Z[x_1,\dots,x_m]$ and
$\{x_i,-x_i^{-1}\}=\frac{z_i}2\pm\sqrt{1+\frac{z_i^2}4}$, where the power series for $\sqrt{1+\frac{z_i^2}4}$
involves no terms that are linear in $z_i$.
Hence the coefficient $\omega(L)$ of $\mho_L(z_1,\dots,z_m)$ at $z_1,\dots,z_m$ equals $\frac{1}{2^m}$ 
times an alternating sum of the coefficients of $\Omega_L(x_1,\dots,x_m)$ at all monomials of the form 
$x_1^{\epsilon_1}\cdots x_m^{\epsilon_m}$, where each $\epsilon_i\in\{1,-1\}$.
Since $\Omega_L(x_1,\dots,x_m)=\Omega_L(-x_1^{-1},\dots,-x_m^{-1})$, 
the coefficients at $x_1^{\epsilon_1}\cdots x_m^{\epsilon_m}$ and $x_1^{-\epsilon_1}\cdots x_m^{-\epsilon_m}$
are the same possibly up to a sign.
Hence the said alternating sum must be even, and so $2^{m-1}\omega(L)\in\Z$.
Now the assertion follows from Lemma \ref{omega}, taking into account that $\bar{\bar\omega}(L)$ is only defined when $m\ge 3$.
\end{proof}

\begin{proposition} \label{jump}
Suppose that $(m+1)$-component links $L_+=(K_{0+},K_1,\dots,K_m)$ and $L_-=(K_{0-},K_1,\dots,K_m)$ differ 
by a positive self-intersection of the $0$th component.
Also let $K_0$ be the singular knot that occurs at the time of this self-intersection and let $K_{0'}$ and $K_{0''}$ be its smoothed lobes.
Let $l_{ij}=\lk(K_i,K_j)$ and $\lambda=\prod_{\{i,j\}\subset\{0\}\cup[m]}l_{ij}$.
Then
\bigskip

(a) $\displaystyle\omega(L_+)-\omega(L_-)=\sum_{(i_1,\dots,i_m)\in[m]!}l_{0'i_1}l_{i_1i_2}\cdots l_{i_{m-1}i_m}l_{i_m0''}$
\bigskip

(b) $\displaystyle\bar\omega(L_+)-\bar\omega(L_-)=\sum_{(i_1,\dots,i_m)\in[m]!}l_{0'i_1}l_{i_1i_2}\cdots l_{i_{m-1}i_m}l_{i_m0''}
\Big(\lambda-\frac{\lambda}{l_{0i_1}}l_{0'i_1}-\frac{\lambda}{l_{0i_m}}l_{0''i_m}\Big)$
\bigskip

(c) $\bar{\bar\omega}(L_+)-\bar{\bar\omega}(L_-)=\bar\omega(L_+)-\bar\omega(L_-)$.
\end{proposition}
\medskip

It should be noted that, taking into account Corollary \ref{3-comp-bar}, part (a) generalizes the formula of Remark \ref{gamma}.
The case $m=2$ of part (b) bears some similarity to \cite{A21}*{formula (28), as applied in the proofs of Theorems 17 and 14.1}.

\begin{proof}[Proof. (a)] Since each $\nabla_{K_i}(z_i)$, $i=0\pm,1,\dots,m$, is of the form $1+\text{\rm (terms of degrees $\ge 2$)}$,
$\omega(L_\pm)$ equals the coefficient of $\mho_{L_\pm}(z_0,\dots,z_m)$ at $z_0\cdots z_m$.
By Proposition \ref{conwayI'} we have $\mho_{L_+}(z_0,\dots,z_m)-\mho_{L_-}(z_0,\dots,z_m)=z_0\mho_{L_0}(z_0,\dots,z_m)$ with
$L_0=(K_{0'},K_{0''},K_1,\dots,K_m)$.
Hence $\omega(L_+)-\omega(L_-)$ is the coefficient of $\mho_{L_0}(z_0,\dots,z_m)$ at $z_1\cdots z_m$.
This coefficient is given by the sum of those terms in the formula of Theorem \ref{hhh} that correspond to trees with leaves
in the vertices $0'$ and $0''$ and with all other vertices being of degree $2$.
\end{proof}

\begin{proof}[(b)] It follows from (a) and Lemma \ref{omega}(a) that
\begin{align*}
\bar\omega(L_+)-\bar\omega(L_-)&=\sum_{(i_1,\dots,i_m)\in[m]!}l_{0'i_1}l_{i_1i_2}\cdots l_{i_{m-1}i_m}l_{i_m0''}\lambda\\
&-\sum_{k\in[m]}\frac{\lambda}{l_{0k}}l_{0'k}l_{0''k}\sum_{(j_1,\dots,j_{m-1})\in([m]\but\{k\})!}l_{0j_1}l_{j_1j_2}\cdots l_{j_{m-2}j_{m-1}}l_{j_{m-1}k}.
\end{align*}
Since $l_{0j_1}=l_{0'j_1}+l_{0''j_1}$, each summand containing $l_{0j_1}$ can be replaced by two similar summands, one containing $l_{0'j_1}$
and another $l_{0''j_1}$ in place of $l_{0j_1}$.
Now set $(i_1,\dots,i_m)=(j_1,\dots,j_{m-1},k)$ in the summands containing $l_{0'j_1}$ and $(i_1,\dots,i_m)=(k,j_{m-1},\dots,j_1)$ in 
the summands containing $l_{0''j_1}$.
This yields the desired expression.
\end{proof}

\begin{proof}[(c)] By Lemma \ref{omega}(b) $\bar{\bar\omega}(L)-\bar\omega(L)$ is a function of the pairwise linking numbers of the components of $L$,
which are invariant under link homotopy of $L$.
\end{proof}

\begin{theorem} \label{satellites}
Let $L=(K_1,\dots,K_m)$ be a link of $m\ge 3$ components, and let $L^P$ and $L^Q$ be its satellites modeled respectively 
on patterns $P=(P_1,\dots,P_m)$ and $Q=(Q_1,\dots,Q_m)$, where each of the $P_i$ and $Q_i$ is a knot in $T:=S^1\x D^2$.
Suppose that each $P_i$ is homotopic to the corresponding $Q_i$.
Then $\bar\omega(L^P)=\bar\omega(L^Q)$ and $\bar{\bar\omega}(L^P)=\bar{\bar\omega}(L^Q)$.
\end{theorem}

\begin{proof} Arguing by induction, it suffices to consider the case where $P_i=Q_i$ for all indices $i$ except one.
By symmetry we may assume that this one index is $1$.
The given homotopy between $P_1$ and $Q_1$ may be assumed to contain only transversal self-intersections.
Without loss of generality there is just one, and (by symmetry) it is positive.
Let $X$ be the intermediate singular knot, and let $X'$ and $X''$ be its smoothed lobes.
Let $p$, $p'$ and $p''$ respectively be their winding numbers in $S^1\x D^2$ 
(that is, $[X]=p[S^1\x\{0\}]\in H_1(T)$ and similarly for $p'$ and $p''$).
Thus $p=p'+p''$.

Writing $L^P=(K^P_1,\dots,K^P_m)$ and $L^Q=(K_1^Q,\dots,K_m^Q)$, we may assume that $K_i^P=K_i^Q$ for $i\ge 2$.
Let us write $Z_i=K_i^P=K_i^Q$ for $i\ge 2$.
The homotopy between $P_1$ and $Q_1$ yields a homotopy between $K_1^P$ and $K_1^Q$ within a regular neighborhood $T_1$ of $K_1$
in the complement of $Z_2\cup\dots\cup Z_m$.
Let $Z_1$ be the intermediate singular knot of this homotopy, and let $Z_{1'}$ and $Z_{1''}$ be its smoothed lobes.
Writing $l_{ij}=\lk(Z_i,Z_j)$ and $k_i=\lk(K_1,Z_i)$, we have
$l_{1i}=pk_i$, $l_{1'i}=p'k_i$ and $l_{1''i}=p''k_i$ for each $i\ge 2$.
Writing $\lambda=\prod_{\{i,j\}\subset[m]}l_{ij}$, we then get that
\[\frac{\lambda}{l_{1i}}l_{1'i}+\frac{\lambda}{l_{1j}}l_{1''j}=\frac{\lambda}{pk_i}p'k_i+\frac{\lambda}{pk_j}p''k_j=\frac{p'+p''}p\lambda=\lambda\]
for arbitrary $i,j\in\{2,\dots,m\}$ (this can be seen to work also when some or all of $k_i$, $k_j$ and $p$ are zero).
Now the assertion follows from Proposition \ref{jump}(b,c).
\end{proof}

\begin{corollary} \label{satellites'}
Let $L=(K_1,\dots,K_m)$, where $m\ge 3$, and let $L'$ be a $(p_1,\dots,p_m)$-satellite of $L$.
Then \[\bar{\bar\omega}(L')=(p_1\cdots p_m)^{m+1}\bar{\bar\omega}(L).\]
\end{corollary}

\begin{proof} By Theorem \ref{cables} the desired assertion holds when each $p_i\ne 0$ and $L'$ is a $(p_1,\dots,p_m)$-cable of $L$.
Also, when some $p_i=0$ and the $i$th component of $L'$ is split from the rest of the link by an embedded sphere, $\mho_{L'}=0$.
Consequently $\bar{\bar\mho}_{L'}=0$, whence $\bar{\bar{\bar\mho}}_{L'}=0$, so in particular $\bar{\bar\omega}(L')=0$.
Thus for every tuple $(p_1,\dots,p_m)$ the desired assertion holds for at least one $(p_1,\dots,p_m)$-satellite $L'$ of $L$.
Hence by Theorem \ref{satellites} it holds for every $(p_1,\dots,p_m)$-satellite of $L$.
\end{proof}

\begin{proposition} \label{lhlt}
$\bar{\bar\omega}$ for $m$-component links is not a sum of a link homotopy invariant and an invariant of type $m(m+1)/2$.
\end{proposition}

\begin{proof} Let $L=(K_0,\dots,K_{m-1})$ be an arbitrary $m$-component link.
For any pair $(K_i,K_j)$ of components of $L$ we may create an intersection between them as follows: choose an arc between a point on $K_i$
and a point on $K_j$, disjoint elsewhere from the link, and push fingers from $K_i$ and $K_j$ along this arc until they intersect.
In particular, let $L'$ be a link obtained in this way from $L$ by introducing a self-intersection of $K_0$.
This transforms $K_0$ into a wedge of two circles; let us call them $K_{0'}$ and $K_{0''}$.
Now the same procedure can be applied to add any number $k$ of double points to $L'$ by starting from $k$ pairwise disjoint arcs,
which are disjoint also from the self-intersection point of $K_0$.
In this way we can create double points between any number of prescribed pairs of the circles $K_{0'},K_{0''},K_1,\dots,K_{m-1}$.
Let $L''$ be a link obtained from $L'$ by creating double points between the following pairs:
\begin{itemize}
\item a double point between $K_i$ and $K_j$ whenever $i\ne j$, $i,j\ge 1$;
\item an additional double point between $K_i$ and $K_{i+1}$ whenever $1\le i\le m-2$;
\item two double points between $K_{0'}$ and $K_1$;
\item two double points between $K_{m-1}$ and $K_{0''}$;
\item a double point between $K_0$ and $K_i$ whenever $2\le i\le m-2$.%
\footnote{It does not matter if this one involves $K_{0'}$ or $K_{0''}$.}
\end{itemize}
Thus $L''$ contains, in terms of the original components $K_0,\dots,K_{m-1}$: (i) a self-intersection of $K_0$; (ii) an intersection between 
each unordered pair of distinct components; (iii) an additional intersection between those components whose indices differ by $1$ modulo $m$.
Altogether we have $1+\frac{m(m-1)}2+m=1+\frac{m(m+1)}2$ double points.

Thus any type $\frac{m(m+1)}2$ invariant, when extended to singular links, vanishes on $L''$.
Also any link homotopy invariant, when extended to singular links, vanishes on any singular link that has at least one self-intersection 
of a component, and in particular on $L''$.

On the other hand, it is not hard to see from Proposition \ref{jump} that the standard extension of $\bar{\bar\omega}$ to singular links, 
also denoted $\bar{\bar\omega}$, does not vanish on $L''$.
In more detail, by Proposition \ref{jump}(b,c) the value of $\bar{\bar\omega}$ on any singular link with a self-intersection of $K_0$ 
and no other double points equals
\[\sum_{(i_1,\dots,i_{m-1})\in[m-1]!}l_{0'i_1}l_{i_1i_2}\cdots l_{i_{m-2}i_{m-1}}l_{i_{m-1}0''}
\Big(\lambda-\frac{\lambda}{l_{0i_1}}l_{0'i_1}-\frac{\lambda}{l_{0i_{m-1}}}l_{0''i_{m-1}}\Big).\]
The summands with $(i_1,\dots,i_{m-1})\ne (1,\dots,m-1)$, when extended to singular links with additional double points, clearly vanish
on $L''$.
The remaining summands 
\begin{align*}
&l_{0'1}l_{12}\cdots l_{m-2,\,m-1}l_{m-1,\,0''}\lambda,\\
&l_{0'1}l_{12}\cdots l_{m-2,\,m-1}l_{m-1,\,0''}\frac{\lambda}{l_{01}}l_{0'1} \text{ and}\\
&l_{0'1}l_{12}\cdots l_{m-2,\,m-1}l_{m-1,\,0''}\frac{\lambda}{l_{0,\,m-1}}l_{0'',\,m-1}
\end{align*}
each involve $m$ repeating linking numbers, each of them with one repeat.
Since the second difference of $t^2$ is $2$, every repeated linking number contributes a factor of $2$, and so
each of the three summands assumes value $2^m$ on $L''$.
Altogether $\bar{\bar\omega}(L'')=-2^m$.
\end{proof}

\section{Cabling and finite type invariants} \label{fti-cabling-section}

Let $\A^\fr_{m,n}$ be the vector space of formal $\Q$-linear combinations of chord diagrams on $m$ circles with $n$ chords 
modulo the $4$-term relations, and let $\A_{m,n}$ be its quotient by the $1$-term relations (see \cite{CDM} for the definitions).
The quotient map $q\:\A^\fr_{m,n}\to\A_{m,n}$ has a natural section $\iota\:\A_{m,n}\to\A^\fr_{m,n}$ by the following lemma.

\begin{lemma} \label{1T}
The subspace $(1T)$ of $\A^\fr_{m,n}$ generated by the classes of the $1$-term relators is the kernel of a natural projection 
$\phi\:\A^\fr_{m,n}\to\A^\fr_{m,n}$.
\end{lemma}

\begin{proof} The desired map $\phi$ is constructed in \cite{KSA}*{\S4} (see also \cite{CDM}*{\S4.4.5}) in the case $m=1$.
In the general case the same argument works to show that the subspace $(1T_i)$ of $\A^\fr_{m,n}$ generated by the classes of
the $1$-term relators for the $i$th component is the kernel of a natural projection $\phi_i\:\A^\fr_{m,n}\to\A^\fr_{m,n}$.
The $\phi_i$ are easily seen to commute, and their composition $\phi:=\phi_1\cdots\phi_m$ is a projection with 
$\ker\phi=(1T_1)+\dots+(1T_m)=(1T)$.
\end{proof}

It is well-known that $\A^\fr_{m,n}$ is isomorphic to the vector space $\B_{m,n}$ of formal $\Q$-linear combinations of $m$-colored
degree $n$ unitrivalent diagrams (also known as ``open Jacobi diagrams'' \cite{CDM}, ``Chinese characters'' \cite{BN} and 
``symmetrized Feynman diagrams'' \cite{KSA}) modulo the antisymmetry relations, the IHX relations and the link relations 
\cite{BN}*{Theorems 6 and 8}, \cite{BGRT2}*{Theorem 3} (see also \cite{CDM}*{Theorems 5.3.1, 5.7.1 and \S5.10.1}, \cite{Mell}).
The isomorphism $\chi\:\B_{m,n}\to\A^\fr_{m,n}$ sends a given unitrivalent diagram to the arithmetic mean \cite{BGRT1}, \cite{BGRT2}, 
\cite{Mell}, \cite{CDM} (correcting the earlier sum \cite{BN}, \cite{KSA}, \cite{Kri}) of all ways of attaching its univalent vertices to the 
Wilson loops (respecting the colors), which ways in turn represent certain linear combinations of chord diagrams via the STU relation.

The subspace $\iota(\A_{m,n})$ of $\A^\fr_{m,n}$ is easy to describe in the language of $\B_{m,n}$:

\begin{lemma} \label{struts}
$\chi^{-1}\iota(\A_{m,n})$ is the subspace of $\B_{m,n}$ generated by the classes of those unitrivalent diagrams that contain no 
monochromatic struts (i.e.\ components that have two univalent vertices of the same color and no trivalent vertices).
\end{lemma}

\begin{proof} The case $m=1$ is proved in \cite{KSA}*{\S4} (see also \cite{Kri}*{Lemma 2.5} for a corrected statement; compare \cite{CDM}*{\S5.8.2}).
In the general case the same argument works to show that $\chi^{-1}\big(\phi_i(\A^\fr_{m,n})\big)$ is generated by the classes of 
those unitrivalent diagrams that contain no struts of color $i$ (see the proof of Lemma \ref{1T} concerning the $\phi_i$).
But $\iota(\A_{m,n})=\bigcap_{i=1}^m\phi_i(\A^\fr_{m,n})$.
\end{proof}

Given an invariant $v$ of $m$-component links and an $m$-tuple of braids $b_i\in B_{p_i}$ whose closures are knots, 
let $v_{b_1\dots b_m}$ be the invariant of $m$-component links defined by $v_{b_1\dots b_m}(L)=v(L_{b_1\dots b_m})$, where 
$L_{b_1\dots b_m}$ is obtained from $L$ by replacing the $i$th component by the closure of $b_i$ for each $i$.
It is not hard to show that if $v$ is a type $n$ invariant, then so is $v_{b_1\dots b_m}$ (see \cite{CDM}*{\S9.2.2}; see also 
\cite{BN}*{Exercise 3.14 and \S7.2}).
Moreover, it is clear from the same argument that given another $m$-tuple of braids $b'_i\in B_{p_i}$ whose closures are knots, 
$v_{b_1\dots b_m}-v_{b'_1\dots b'_m}$ is a type $n-1$ invariant.
This implies the equivalence of two definitions in \S\ref{bml} of a type $n$ invariant being $(k_1,\dots,k_m)$-braidable modulo type $n-1$ invariants.

For a $\Q$-valued type $n$ invariant $v$ of $m$-component links let $s(v)$ denote its symbol, i.e.\ the function $\A_{m,n}\to\Q$ 
given by the extension of $v$ to singular links with $n$ double points, and let $s^\fr(v)$ be the composition 
$\A^\fr_{m,n}\xr{q}\A_{m,n}\xr{s(v)}\Q$.
It is not hard to show that $s^\fr(v_{b_1\dots b_m})=s^\fr(v)\circ\psi_{p_1,\dots,p_m}$
(and consequently $s(v_{b_1\dots b_m})=s(v)\circ q\psi_{p_1,\dots,p_m}\iota$) for a certain linear map 
$\psi_{p_1,\dots,p_m}\:\A^\fr_{m,n}\to\A^\fr_{m,n}$ (see \cite{KSA}*{\S3}, \cite{CDM}*{\S9.2.2, \S9.2.4}; see also \cite{BN}*{Exercise 3.14 and \S7.2}).

It is not hard to see that if $\B_{k_1,\dots,k_m;\,n}$ is the subspace of $\B_{m,n}$ generated by the classes of
all unitrivalent diagrams of degree $n$ with $k_i$ univalent vertices of color $i$, then $\chi^{-1}\psi_{p_1,\dots,p_m}\chi\:\B_{m,n}\to\B_{m,n}$ 
sends every $[D]\in\B_{k_1,\dots,k_m;\,n}$ to $p_1^{k_1}\cdots p_m^{k_m}[D]$ \cite{KSA}*{Theorem 3} (the proof is for knots but it also works 
for links; see also \cite{CDM}*{\S9.2.3, \S9.2.5}).
Since the antisymmetry relations, the IHX relations and the link relations are homogeneous in the numbers of univalent vertices of each color,
$\B_{m,n}$ is the direct sum of the eigenspaces $\B_{k_1,\dots,k_m;\,n}$.

Lemma \ref{struts} and the above description of $\chi^{-1}\psi_{p_1,\dots,p_m}\chi$ imply that the latter sends $\chi^{-1}\iota(\A_{m,n})$ into itself.
Moreover, since the antisymmetry relations, the IHX relations and the link relations do not mix those unitrivalent diagrams that 
contain monochromatic struts with those that don't, $\chi^{-1}\iota(\A_{m,n})$ 
is the direct sum of the eigenspaces $\B_{k_1,\dots,k_m;\,n}\cap\chi^{-1}\iota(\A_{m,n})$ of the restriction
$\chi^{-1}\psi_{p_1,\dots,p_m}\chi|_{\chi^{-1}\iota(\A_{m,n})}$.
In other words, we have $\A_{m,n}=\bigoplus_{k_1,\dots,k_m} q\chi(\B_{k_1,\dots,k_m;\,n})$,
and each $q\chi(\B_{k_1,\dots,k_m;\,n})$ is an eigenspace of $q\psi_{p_1,\dots,p_m}\iota$ with eigenvalue $p_1^{k_1}\cdots p_m^{k_m}$.

\begin{theorem} \label{criterion}
A $\Q$-valued type $n$ invariant $v$ of $m$-component links is $(k_1,\dots,k_m)$-braidable modulo type $n-1$ invariants
if and only if $s(v)$ vanishes on $q\chi(\B_{i_1,\dots,i_m;\,n})$ for all $(i_1,\dots,i_m)\ne(k_1,\dots,k_m)$.
\end{theorem}

\begin{proof} By the above if $s(v)$ vanishes on $q\chi(\B_{i_1,\dots,i_m;\,n})$ for all $(i_1,\dots,i_m)\ne(k_1,\dots,k_m)$,
then $s(v)\circ q\psi_{p_1,\dots,p_m}\iota=p_1^{k_1}\cdots p_m^{k_m}s(v)$ for all $p_1,\dots,p_m\in\Z\but\{0\}$.
Conversely, given any $(i_1,\dots,i_m)\ne(k_1,\dots,k_m)$, there exist $p_1,\dots,p_m\in\Z\but\{0\}$ such that 
$p_1^{i_1}\cdots p_m^{i_m}\ne p_1^{k_1}\cdots p_m^{k_m}$.
(For instance, if $i_j\ne k_j$, set $p_j=2$ and $p_i=1$ for all $i\ne j$.)
Then $s(v)\circ q\psi_{p_1,\dots,p_m}\iota=p_1^{k_1}\cdots p_m^{k_m}s(v)$ implies that
$s(v)$ vanishes on $q\chi(\B_{i_1,\dots,i_m;\,n})$.

Given an $m$-tuple of braids $b_i\in B_{p_i}$ whose closures are knots, we also have
$s(v)\circ q\psi_{p_1,\dots,p_m}\iota=s(v_{b_1\dots b_m})$.
And on the other hand we have $p_1^{k_1}\cdots p_m^{k_m}s(v)=s(p_1^{k_1}\cdots p_m^{k_m}v)$.
Hence $s(v)\circ q\psi_{p_1,\dots,p_m}\iota=p_1^{k_1}\cdots p_m^{k_m}s(v)$ is equivalent to
$s(v_{b_1\dots b_m})=s(p_1^{k_1}\cdots p_m^{k_m}v)$.
The latter in turn holds if and only if
$v_{b_1\dots b_m}-p_1^{k_1}\cdots p_m^{k_m}v$ is a type $n-1$ invariant,
or equivalently $v(L_{b_1\cdots b_m})-p_1^{k_1}\cdots p_m^{k_m}v(L)$ is a type $n-1$ invariant of $L$.
\end{proof}

\begin{corollary} \label{fti-cabling-corollary2}
Let $v$ be a $\Q$-valued type $n$ invariant of $m$-component links.
In parts (a) and (c), assume further that it is not a type $n-1$ invariant.

(a) If $v$ is $(k_1,\dots,k_m)$-braidable modulo type $n-1$ invariants, then $k_1+\dots+k_m\le 2n$ and each $k_i\le n$.

(b) When $k_1+\dots+k_m=2n$, the invariant $v$ is $(k_1,\dots,k_m)$-braidable modulo type $n-1$ invariants if and only if 
$v=L+v'$, where $v'$ is a type $n-1$ invariant and $Lz_1^{k_1}\cdots z_m^{k_m}$ is a polynomial in 
$l_{ij}z_iz_j$, the $l_{ij}$ being the pairwise linking numbers.

(c) When $m=1$, the invariant $v$ is $n$-braidable modulo type $n-1$ invariants if and only if $v=M+v'$, where $v'$ 
is a type $n-1$ invariant and $Mz^n$ is a polynomial in the terms $c_iz^{2i}$ of the Conway polynomial 
(in particular, $n$ must be even).
\end{corollary}

\begin{proof}[Proof. (a)]
Let $[D]\in\B_{k_1,\dots,k_m;\,n}\cap\chi^{-1}\iota(\A_{m,n})$.
Then $k_i$ is the number of $i$-colored univalent vertices of $D$ and $n$ is half the total number of vertices of $D$,
so $k_1+\dots+k_m\le 2n$.

If a trivalent vertex of $D$ is adjacent to at least two same-colored univalent vertices, 
then $2[D]=0$ due to the antisymmetry relation (see \cite{KSA}*{formula (26)}).
Over $\Q$ we then also have $[D]=0$.
So if $[D]\ne 0$, then the vertices adjacent to distinct univalent vertices of color $i$ are distinct from each other.
Also they cannot be univalent vertices of color $i$ because $D$ contains no monochromatic struts (see Lemma \ref{struts}).
Hence there are at least $k_i$ of such vertices, which are distinct from the $k_i$ univalent vertices of color $i$. 
Thus $k_i+k_i\le 2n$.
\end{proof}

\begin{proof}[(b)] The ``if'' assertion is easy.
To prove the ``only if'' assertion, let $[D]\in\B_{k_1,\dots,k_m;\,n}\cap\chi^{-1}\iota(\A_{m,n})$.
Since $k_1+\dots+k_m=2n$, all vertices of $D$ are univalent.
Since $D$ contains no monochromatic struts, it consists entirely of dichromatic struts.
Corresponding to every such strut is the linking number of the respective components of the link.
Let $\lambda_D$ be the product of all these linking numbers.
Then $\lambda_D$ is a type $n$ invariant (see \cite{M24-1}*{Corollary \ref{fti:product}}).

Since for every $[D]\in\B_{k_1,\dots,k_m;\,n}\cap\chi^{-1}\iota(\A_{m,n})$, the diagram $D$ contains 
no trivalent vertices, while each summand of each IHX relation, each anti-symmetry relation
and each link relation contains a trivalent vertex, a basis of 
$\B_{k_1,\dots,k_m;\,n}\cap\chi^{-1}\iota(\A_{m,n})$ is given by the set $S_{k_1,\dots,k_m}$ 
of the classes $[D]$ of all diagrams $D$ that have $k_i$ vertices of color $i$ and consist 
entirely of dichromatic struts.
Hence $q\chi(S_{k_1,\dots,k_m})$ is a basis of the isomorphic space $q\chi(\B_{k_1,\dots,k_m;\,n})$.
Given a diagram $D\in S_{k_1,\dots,k_m}$, let $D^*$ be the linear map $\A_{m,n}\to\Q$ 
which assumes $1$ on $q\chi([D])$ and $0$ on the other elements of $q\chi(S_{k_1,\dots,k_m})$ as well as 
on the other summands of the decomposition $\A_{m,n}=\bigoplus_{i_1,\dots,i_m} q\chi(\B_{i_1,\dots,i_m;\,n})$.
Let us recall that $q\chi([D])$ is the arithmetic mean of all chord diagrams obtained by attaching
the struts of $D$ to $m$ circles respecting the colors.
It is easy to see that the symbol $s(\lambda_D)$ assumes the same value on each of these diagrams 
and $0$ on all other chord diagrams.
Namely, if $a_{ij}$ is the number of edges of $D$ with vertices of colors $i$ and $j$, then 
$\lambda_D=\prod_{i<j}l_{ij}^{a_{ij}}$ and the said value $a_D$ equals $\prod_{i<j}a_{ij}!$.
Hence $s(\lambda_D)=a_DD^*$.

Finally let us recall that we are given an invariant $v$ which is $(k_1,\dots,k_m)$-braidable 
modulo type $n-1$ invariants.
By Theorem \ref{criterion} the symbol $s(v)$ is a linear combination $s(v)=\sum_D c_DD^*$ over all 
$D\in S_{k_1,\dots,k_m}$.
Let $\lambda_v=\sum_D \frac{c_D}{a_D}\lambda_D$.
Then $s(\lambda_v)=s(v)$.
Hence $v-\lambda_v$ is a type $n-1$ invariant.
\end{proof}

\begin{proof}[(c)] The ``if'' assertion follows from Theorem \ref{criterion} and \cite{Kri}*{Theorem 2.10}.
The ``only if'' assertion follows from \cite{Kri}*{Lemma 2.11} similarly to the above proof of part (b).
\end{proof}

\begin{remark}
A formula for the behavior of the Kontsevich integral of a link under $(p,q)$-cabling of a given component is known \cite{BLT}*{Theorem 1},
\cite{CDM}*{\S9.3}.
\end{remark}

\section{Milnor's invariants} \label{milnor}

Let us briefly recall the definition of Milnor's $\bar\mu$-invariants \cite{Mi2}.
We start with a link $L=(K_1,\dots,K_m)$.
Writing $\pi_L=\pi_1(S^3\but L)$, the homomorphism $b_L\:F_m\to\pi_L$ from the free group given by a choice of meridians to the components of $L$ 
descends modulo the $q$th lower central term to an epimorphism $\bar b_L\:F_m\to\pi_L/\gamma_q\pi_L$ for every finite $q$ \cite{Mi2}*{Theorem ~4}.
Writing $M\:F_m\to\Z\left<\left<x_1,\dots,x_m\right>\right>$ for the Magnus expansion, $\mu_{j_1\dots j_ni}(L)$ is defined as 
the coefficient of $M(\bar\lambda_i)$ at $x_{j_1}\cdots x_{j_n}$, where $\bar\lambda_i\in F_m$ is some representative of the class 
$[\lambda_i]\in\pi_L/\gamma_q$ of a longitude $\lambda_i$ of $K_i$ for some $q>n$ \cite{Mi2}*{p.\ 291}.
Finally, $\bar\mu_{i_1\dots i_k}(L)$ is the residue class of $\mu_{i_1\dots i_k}(L)$ modulo 
the gcd $\delta_{i_1\dots i_k}(L)$ of all $\mu_{j_1\dots j_l}(L)$ where $j_1,\dots,j_l$ ranges over all ordered subsequences
of $i_1,\dots,i_k$ with $2\le l<k$ \cite{Mi2}*{p.\ 292 and assertion ~(21)}.

\begin{remark} \label{mu-modified}
(a) In the case where $i\notin\{j_1,\dots,j_n\}$, the integer $\mu_{j_1\dots j_ni}(L)$ can be 
redefined in a slightly different way.
Namely, writing $\Lambda=(K_1,\dots,\hat K_i,\dots,K_m)$, where the hat denotes omission, $\mu_{j_1\dots j_ni}(L)$ equals the coefficient 
of $M'(\tilde l_i)$ at $x_{j_1}\cdots x_{j_n}$, where $M'\:F_{m-1}\to\Z\left<\left<x_1,\dots,\hat x_i,\dots,x_m\right>\right>$ is the Magnus expansion and $\tilde l_i\in F_{m-1}$ is some representative of the class $[l_i]\in\pi_\Lambda/\gamma_q\pi_\Lambda$ of $l_i$ for some $q>n$.%
\footnote{To see that this new definition of $\mu_{j_1\dots j_ni}(L)$ for the case $i\notin\{j_1,\dots,j_n\}$ is equivalent to 
the original one, it suffices to consider the commutative diagram
\[\begin{CD}
\pi_L/\gamma_q\pi_L@<\bar b_L<<F_m@>M>>\Z\left<\left<x_1,\dots,x_m\right>\right>\\
@Vp_iVV@Vq_iVV@Vr_iVV\\
\pi_\Lambda/\gamma_q\pi_\Lambda@<\bar b_\Lambda<<F_{m-1}@>M'>>\Z\left<\left<x_1,\dots,\hat x_i,\dots,x_m\right>\right>,
\end{CD}\]
where $p_i$ kills all meridians of $K_i$ (i.e.\ sends them to $1$), $q_i$ kills the $i$th generator (i.e.\ sends it to $1$),
and $r_i$ kills $x_i$ (i.e.\ sends it to $0$ --- not to $1$!!!); and observe that $r_i$ preserves the coefficients of all monomials 
that do not involve $x_i$.}
In fact here $l_i$ can be taken to be a loop going along some path from the basepoint to a point of $K_i$, then traversing $K_i$,
and finally returning back to the basepoint along the same path.

(b) From the modified definition in (a) it is easy to see that $\bar\mu$-invariants with pairwise distinct indices are link homotopy invariants 
(by considering link homotopies with support in one component).
This yields an alternative proof of \cite{Mi2}*{Theorem 8}.
\end{remark}

\begin{theorem} \label{mu-satellite}
Let $L'$ be a $(p_1,\dots,p_m)$-satellite of an $m$-component link $L$.
Then $\delta_{i_1\dots i_k}(L)$ divides $\delta_{i_1\dots i_k}(L')$ and
\[\bar\mu_{i_1\dots i_k}(L')\equiv p_1^{n_1}\cdots p_m^{n_m}\bar\mu_{i_1\dots i_k}(L)\pmod{\delta_{i_1\dots i_k}(L)},\] 
where $n_j$ is the number of occurrences of $j$ among $i_1,\dots,i_k$.
\end{theorem}

See Corollary \ref{beta-cable}(b) for an alternative proof of Theorem \ref{mu-satellite} for $\bar\mu_{1122}$ 
(using the Conway potential function).

\begin{example} \label{wh-hopf} 
Both the Whitehead link $W$ and the unlink $U$ are $(1,0)$-satellites of the Hopf link $H$, but
we have $\bar\mu_{1122}(W)=1$ and $\bar\mu_{1122}(U)=0$.
This does not contradict Theorem \ref{mu-satellite} as $\bar\mu_{1122}(H)$ takes values in the zero group.
\end{example}

\begin{proof}[Proof of Theorem \ref{mu-satellite}] We may assume that the theorem is known for proper subsequences of $i_1,\dots,i_k$.
Then the first assertion of the theorem for $i_1,\dots,i_k$ follows from the second assertion for
proper subsequences of $i_1,\dots,i_k$.

It remains to prove the second assertion.
By iterating it suffices to consider the case where $L'$ is obtained from $L$ by replacing 
just one component $K_l$ with a knot $K'_l$ which is a $p$-satellite of $K_l$.
Thus $K'_l$ lies in a regular neighborhood $T_l$ of $K_l$.
By cyclic symmetry \cite{Mi2}*{assertion (21)} we may assume that $i_k\ne l$.%
\footnote{We could have assumed $i_k=l$ as well, but this approach seems to work nicely (using the definition of 
Remark \ref{mu-modified}) only when $i_1,\dots,i_k$ are pairwise distinct.}

The remainder of the proof is similar to the proofs of Theorem 7 and assertion (13) in \cite{Mi2}.
In more detail, we consider the composition $\iota\:\pi_L\simeq\pi_1\big(S^3\but (L\cup T_l)\big)\xr{i_*}\pi_{L'}$, where
$i_*$ is the inclusion induced homomorphism.
Then $\iota$ sends a chosen meridian $\mu_l$ of $K_l$ to a product of meridians of $K'_l$ and their inverses.
Each of these meridians is conjugate to a chosen meridian $\mu'_l$ of $K'_l$.
Thus $\iota(\mu_l)=\prod_{i=1}^r u_i^{-1}(\mu'_l)^{\epsilon_i}u_i$ for some $r$, $\epsilon_i\in\{1,-1\}$ and $u_i\in\pi_{L'}$.
(If $K'_l$ is a $p$-braiding of $K_l$, we may choose $r=|p|$ and each $\epsilon_i=\sgn(p)$; in general, $\epsilon_1+\dots+\epsilon_r=p$.)
Also, modulo $\gamma_q$, each $u_i=\prod_{j=1}^{r_i}(\mu'_{h_{ij}})^{\epsilon_{ij}}$ for some $r_i$, $h_{ij}$ and $\epsilon_{ij}\in\{1,-1\}$.
Then $\iota$, when reduced modulo $\gamma_q$, lifts with respect to $\bar b_L$ and $\bar b_{L'}$ to a homomorphism $\phi\:F_m\to F_m$ 
sending each generator $\alpha_i$ for $i\ne l$ to itself and $\alpha_l$ to $\prod_{i=1}^rv_i^{-1}\alpha_l^{\epsilon_i}v_i$, 
where each $v_i=\prod_{j=1}^{r_i}\alpha_{h_{ij}}^{\epsilon_{ij}}$.
Since each $M(\alpha_i)=1+x_i$ in the Magnus ring, $M(v_i^{-1}\alpha_lv_i)=1+M(v_i)^{-1}x_lM(v_i)$
and $M(v_i^{-1}\alpha_l^{-1}v_i)=1-M(v_i)^{-1}(x_l-x_l^2+x_l^3-\dots)M(v_i)$.
Since $M(v_i)=\prod_{j=1}^{r_i}(1+x_{h_{ij}})^{\epsilon_{ij}}$, it is easy to see that
$M(v_i)^{-1}x_lM(v_i)=x_l+\text{(terms involving $x_lx_h$ or $x_hx_l$ for some $h$)}$.
Similarly $M(v_i)^{-1}(x_l-x_l^2+x_l^3-\dots)M(v_i)=x_l+\text{(terms involving $x_lx_h$ or $x_hx_l$)}$.
It follows that $M\big(\phi(\alpha_l)\big)=1+px_l+\text{(terms involving $x_lx_h$ or $x_hx_l$)}$.
Consequently, $M\big(\phi(\bar\lambda_i)\big)$ is obtained from $M(\bar\lambda_i)$ by substituting
$x_l$ with $px_l+\text{(terms involving $x_lx_h$ or $x_hx_l$)}$.
By \cite{Mi2}*{assertion (18)} the terms involving $x_lx_h$ or $x_hx_l$ are irrelevant for the purposes of computing 
$\mu_{j_1\dots j_ni}(L')$, $n<q$, and so we may as well substitute $x_l$ with $px_l$.
On applying this substitution to $x_{j_1}\dots x_{j_n}$, we get $p^{n_l}x_{j_1}\dots x_{j_n}$, using that $n_l$ equals the number 
of occurrences of $l$ among $j_1,\dots,j_n$ (rather than among $j_1,\dots,j_n,i$) due to $i\ne l$.
\end{proof}

\begin{corollary} \label{mu-satellite'}
Let $L'$ be a $(p_1,\dots,p_m)$-satellite of an $m$-component link $L$ such that $\delta_{i_1\dots i_k}(L)=0$.
Then $\delta_{i_1\dots i_k}(L')=0$ and
\[\bar\mu_{i_1\dots i_k}(L')=p_1^{n_1}\cdots p_m^{n_m}\bar\mu_{i_1\dots i_k}(L),\] 
where $n_j$ is the number of occurrences of $j$ among $i_1,\dots,i_k$.
\end{corollary}

\begin{remark}
(a) Corollary \ref{mu-satellite'} is easy to prove geometrically at least for
\begin{itemize} 
\item $\bar\mu_{123}$ (described geometrically in \cite{Co2}*{\S5});
\item $\bar\mu_{11\dots 1122}$ (described geometrically in \cite{Co1}, as shown in \cite{Co3}*{Theorem 6.10}, \cite{Ste});
\item $\bar\mu_{111222}$ and $\bar\mu_{1123}$ (described geometrically in \cite{M21}*{Remark 2.5}, as shown in \cite{Co3}).
\end{itemize}
Presumably, this approach can be extended to the case of arbitrary $\bar\mu$-invariants, which are described geometrically 
(in principle) in \cite{Co3}.

(b) To extend these geometric descriptions to the case $\delta_{i_1\dots i_k}(L)\ne 0$ (in order to also prove geometrically 
Theorem \ref{mu-satellite}), one can work with $\Z/\delta_{i_1\dots i_k}(L)$-manifolds (see \cite{BRS}*{Chapter III}). 
For example, a $\Z/2$-manifold is just a possibly non-orientable manifold.
This extension is not without work, but for instance in the case of $\bar\mu_{123}$, to extend the argument of \cite{Co2}*{\S5} 
one can use \cite{Le88}*{Figure 9} in place of \cite{Co2}*{Figure 5.3}.
\end{remark}

\section{Linked solenoids}

Let us recall that the limit $X_\infty$ of an inverse sequence $\dots\xr{\pi_2}X_1\xr{\pi_1}X_0$ of spaces is  the subset of
$\prod_{i=0}^\infty X_i$ consisting of all {\it threads}, i.e.\ sequences $(p_0,p_1,\dots)$ such that $\pi_i(p_i)=p_{i-1}$ for each $i\ge 1$.
The composition $X_k\xr{\pi_k}\dots\xr{\pi_{i+1}}X_i$ will be denoted $\pi^k_i$, also we use the notation
$\pi^\infty_j\:X_\infty\to X_j$ for the restriction of the coordinate projection $\prod_{i=0}^\infty X_i\to X_j$.
Further information about inverse limits can be found in \cite{M00}.

\begin{lemma}\label{Mardesic}
Let $K$ be the limit of an inverse sequence $\dots\xr{\pi_2}P_1\xr{\pi_1}P_0$ of compact metric spaces 
and let $Q$ be an ANR with a fixed metric on it.

(a) For every map $f\:K\to Q$ and each $\eps>0$ there exists a $j$ and a map $g\:P_j\to Q$ such that 
$f$ is $\eps$-close%
\footnote{Maps $\phi,\psi\:X\to Y$, where $Y$ is a metric space, are called {\it $\eps$-close} if
$d\big(\phi(x),\psi(x)\big)<\eps$ for each $x\in X$.}
to the composition $K\xr{\pi^\infty_j}P_j\xr{g}Q$.

(b) Suppose that $f,g\:P_i\to Q$ are maps such that the two compositions 
$K\xr{\pi^\infty_i}P_i\overset{f}{\underset{g}{\rightrightarrows}}Q$ are $\eps$-close.
Then for each $\delta>0$ there exists a $k$ such that the two compositions
$P_k\xr{\pi^k_i}P_i\overset{f}{\underset{g}{\rightrightarrows}}Q$ are
$(\eps+\delta)$-close.
\end{lemma}

See \cite{M-unif}*{Lemma 16.5} for a proof of a more general lemma on uniformly continuous maps.
It should be noted that the conclusion of Lemma \ref{Mardesic} resembles the definition of a resolution 
from shape theory (see \cite{M00}*{Theorem 20.1} and references there).
A version of (a) is found already in the Eilenberg--Steenrod book \cite[Theorem 10.11.9]{ES}.

In what follows we use the terminology introduced in Remark \ref{solenoids-remark}.

\begin{corollary} \label{taming} Let $\Lambda\:\Sigma\to S^3$ be a link of $m$ solenoids.
For each $\eps>0$ there exists an $r$ and a smooth link $L$ such that the composition 
$\Sigma\xr{\pi^\infty_r}mS^1\xr{L}S^3$ is $\eps$-close to $\Lambda$.
\end{corollary}

Given an inverse sequence of coverings $\dots\xr{\pi_1}S^1\xr{\pi_0}S^1$, each $\pi_i$ induces an isomorphism on rational $1$-cohomology, 
so the (\v Cech) cohomology group%
\footnote{See \cite{M2}*{\S4} concerning \v Cech cohomology of compact metric spaces; \cite{M00} treats \v Cech cohomology of arbitrary metric spaces.}
$H^1(\Sigma;\Q)\simeq\Q$ for every solenoid $\Sigma$.
Moreover, each projection $\pi^\infty_i\:\Sigma\to S^1$ induces an isomorphism on the rational $1$-cohomology.

Let $T\subset S^3$ be a disjoint union of tame solid tori.
$T$ is said to be a {\it thickening} of a topological link $\Lambda$ if the image $\Sigma$ of $\Lambda$ lies in the interior of $T$ and
the inclusion map $\Sigma\to T$ is a homotopy equivalence.
A topological link is called {\it thickenable} if it has a thickening (this terminology is due to F. Ancel).
On the other hand, we call $T$ a {\it $\Q$-thickening} of a link of solenoids $\Lambda$ if the image $\Sigma$ of $\Lambda$ lies in the interior of $T$ 
and the inclusion map $\Sigma\to T$ induces an isomorphism on rational $1$-cohomology.
A link of solenoids will be called {\it $\Q$-thickenable} if it has a $\Q$-thickening.

The equivalence relation on links of solenoids generated by ambient isotopy and the relation of having a common $\Q$-thickening
respecting the numbering of the components will be called {\it F$_\Q$-isotopy}.
The equivalence relation on topological links generated by ambient isotopy and the relation of having a common thickening
respecting the numbering and the orientations of the components will be called {\it F-isotopy} (compare \cite{KY}).
Proposition \ref{SFextension} below shows that in the case of smooth links this is equivalent to the definition of F-isotopy 
in \S\ref{intro}.

\begin{remark} 
It is shown in \cite{An}*{\S3, Step 1}, \cite{M24-3}*{Proposition \ref{rolf:insertion}} that every topological knot is thickenable
(and in particular is F-isotopic to the unknot).
The argument of \cite{M24-3}*{Proposition \ref{rolf:insertion}} (using that knotted solenoids are tame in the sense of 
Shtan'ko \cite{Sh2}*{Theorem 4}, \cite{McM}*{Theorem 3}) works to show that every knotted solenoid is $\Q$-thickenable 
(and in particular is F$_\Q$-isotopic to the unknot).
\end{remark}

\begin{proposition} \label{stabilization} (a) Every tubular topological link is thickenable.

(b) Every tubular link of solenoids is $\Q$-thickenable.
\end{proposition}

\begin{proof}[Proof. (b)] Let $\Sigma$ be the image of a tubular link of solenoids.
It suffices to show that $\Sigma$ is the intersection of a chain $\dots\subset T_1\subset T_0$, where every $T_i$ is a disjoint union of tame solid tori 
and each inclusion map $T_{i+1}\to T_i$ induces an isomorphism on rational $1$-cohomology.
To show this, it in turn suffices to consider the case where $\Sigma$ is connected.

In this case $\Sigma$, being the image of a tubular knotted solenoid, is the intersection of a chain $\dots\subset T_1\subset T_0$ of tame solid tori.
The rational cohomology group $H^1(\Sigma)$ is the colimit of the direct sequence $H^1(T_0)\xr{p_0} H^1(T_1)\xr{p_1}\dots$.
Since $\Sigma$ is a solenoid, $H^1(\Sigma)\ne 0$, so only finitely many of the $p_i$ can be zero.
By omitting some indices we may assume that all the $p_i$ are nonzero.
Then they all are isomorphisms.
\end{proof}

\begin{proof}[(a)] Similarly to (b).
\end{proof}

\begin{lemma} \label{approximation0} If $T$ is a $\Q$-thickening of a link of solenoids $\Lambda\:\Sigma\to S^3$,
then there exists an $\eps>0$ such that $T$ is also a $\Q$-thickening of every smooth link $L$ 
such that for some $r$ the composition $\Sigma\xr{\pi^\infty_r}mS^1\xr{L}S^3$ is $\eps$-close to $\Lambda$.
\end{lemma}

\begin{proof}
Since the image of $\Lambda$ lies in the interior of $T$, there exists an $\eps>0$ such that if a map $\Sigma\to S^3$ is $\eps$-close 
to $\Lambda$, then it is homotopic to $\Lambda$ with values in $T$.
Given a smooth link $L$ such that for some $r$ the composition $\Sigma\xr{\pi^\infty_r}mS^1\xr{L}S^3$ is $\eps$-close to $\Lambda$,
by our choice of $\eps$ this composition is homotopic to $\Lambda$ with values in $T$.
Since $\Lambda\:\Sigma\to T$ induces an isomorphism on rational $1$-cohomology, so does the composition $\Sigma\xr{\pi^\infty_r}mS^1\xr{L}T$.
Since $\pi^\infty_r$ also induces an isomorphism on rational $1$-cohomology, so does $L\:mS^1\to T$.
\end{proof}

\begin{proposition} \label{SFextension}
(a) Two smooth links are F-isotopic if and only if they are related by the equivalence relation on smooth links 
generated by ambient isotopy and the operation of replacing a given link with any of its $(1,\dots,1)$-satellites.

(b) Two smooth links are F$_\Q$-isotopic if and only if they are related by the equivalence relation on smooth links 
generated by ambient isotopy and the operation of replacing a given link with any of its $(p_1,\dots,p_m)$-satellites, where each $p_i\ne 0$.
\end{proposition}

\begin{proof}[Proof. (b)] Let $L$ and $L'$ be F$_\Q$-isotopic smooth links.
Then they are joined by a sequence of links of solenoids $L=\Lambda_{11},\Lambda_{12},\dots,\Lambda_{r1},\Lambda_{r2}=L'$,
where each $\Lambda_{i1}\:\Sigma_i\to S^3$ is ambient isotopic to $\Lambda_{i2}\:\Sigma_i\to S^3$ and each $\Lambda_{i2}$, $i<r$, 
has a common $\Q$-thickening $T_{i2}=T_{i+1,1}$ with $\Lambda_{i+1,1}$.
Let $T_{11}$ and $T_{r2}$ be some tubular neighborhoods of $L$ and $L'$ respectively.
Then each $T_{ij}$ is a $\Q$-thickening of $\Lambda_{ij}$. 

Let $\eps_{ij}$ be given by Lemma \ref{approximation0} for $\Lambda_{ij}$ and $T_{ij}$.
Since $S^3$ is compact, the ambient isotopy $H_{it}\:S^3\to S^3$ taking $\Lambda_{i1}$ onto $\Lambda_{i2}$ is uniformly continuous,
so there exists a $\delta_{i2}$ such that $H_{i1}$ sends $\delta_{i2}$-close points into $\frac12\eps_{i2}$-close ones.
By Corollary \ref{taming} there exists a smooth link $L_{i1}$ such that for some $r_i$ the composition 
$\Sigma_i\xr{\pi^\infty_{r_i}}mS^1\xr{L_{i1}}S^3$ is $\min(\eps_{i1},\delta_{i2})$-close to $\Lambda_{i1}$.
Let $L_{i2}$ be a smooth link, $\frac12\eps_{i2}$-ambient isotopic to the tame link $H_{i1}L_{i1}$.
Then each composition $\Sigma_i\xr{\pi^\infty_{r_i}}mS^1\xr{L_{ij}}S^3$ is $\eps_{ij}$-close to $\Lambda_{ij}$.

By Lemma \ref{approximation0} $T_{ij}$ is a $\Q$-thickening of $L_{ij}$.
Then $L_{ij}$ is a $(p_1,\dots,p_m)$-satellite of the core link of $T_{ij}$, where each $p_i\ne 0$.
But each $L_{i1}$ is ambient isotopic to $L_{i2}$, each $T_{i2}=T_{i+1,1}$ and the core links of $T_{11}$ and $T_{r2}$ are the original links $L$ and $L'$.
\end{proof}

\begin{proof}[(a)] Similarly to (b).
\end{proof}

\begin{lemma} \label{approximation} (a) Given a thickenable topological link $\Lambda$, there exists an $\eps>0$ such that
every topological link, $\eps$-close to $\Lambda$, is F-isotopic to $\Lambda$.

(b) Given a $\Q$-thickenable link of solenoids $\Lambda$, there exists an $\eps>0$ such that
every link of solenoids, $\eps$-close to $\Lambda$, is F$_\Q$-isotopic to $\Lambda$.
\end{lemma}

\begin{proof}[Proof. (b)] Let $T$ be a $\Q$-thickening of $\Lambda$.
Since the image of $\Lambda$ lies in the interior of $T$, there exists an $\eps>0$ such that if a link of solenoids $\Lambda'$ is 
$\eps$-close to $\Lambda$, then it is homotopic to $\Lambda$ with values in $T$.
But given such a homotopy, clearly $T$ is a $\Q$-thickening of $\Lambda'$.
\end{proof}

\begin{proof}[(a)] Similarly to (b).
\end{proof}

Two embeddings $\Lambda,\Lambda'\:\Sigma\to S^3$ are called {\it isotopic} if they are homotopic through embeddings.

\begin{corollary} (a) If two thickenable topological links are isotopic through thickenable topological links, then they are F-isotopic.

(b) If two $\Q$-thickenable links of solenoids are isotopic through $\Q$-thickenable links of solenoids, then they are F$_\Q$-isotopic.
\end{corollary}

\begin{proof}[Proof. (b)] Given an isotopy $\Lambda_t$ through $\Q$-thickenable links of solenoids, the isotopy class of the link 
of solenoids $\Lambda_t$, regarded as a function of $t$, is locally constant by Lemma \ref{approximation}(b) and hence constant due to 
the connectedness of $I$.

Let us recast this argument in more explicit terms.
For each $t\in I$ let $\eps_t$ be given by Lemma \ref{approximation}(b) for $\Lambda_t$, 
and let $U_t$ consist of all $s\in I$ such that $\Lambda_s$ is $\eps_t$-close to $\Lambda_t$.
Then $\{U_t\mid t\in I\}$ is an open cover of $I$, so it has a finite subcover.
Since $I$ is connected, the elements of this finite subcover can be inductively numbered $V_1,\dots,V_l$ so that $0\in V_1$, 
each $V_i\cap V_{i+1}\ne 0$ and $1\in V_l$.
Thus, writing $V_i=U_{t_i}$, there exist $0=s_0,s_1,\dots,s_l,s_{l+1}=1$ such that each $\Lambda_{s_i}$ is $\eps_{t_i}$-close to 
$\Lambda_{t_i}$ and $\eps_{t_{i-1}}$-close to $\Lambda_{t_{i-1}}$.
Then by Lemma \ref{approximation}(b) $\Lambda_{s_i}$ is F$_\Q$-isotopic to $\Lambda_{t_i}$ and to $\Lambda_{t_{i-1}}$.
\end{proof}

\begin{proof}[(a)] Similarly to (b).
\end{proof}

\begin{theorem} \label{extension}
Let $v$ be an invariant of smooth links.

(a) If $v$ is an invariant of F-isotopy, then it extends to an invariant $\bar v$ of thickenable topological links.
Moreover, $\bar v$ is invariant under F-isotopy.

(b) If $v$ is $0$-solenoidal,%
\footnote{By Proposition \ref{SFextension}(b) $v$ is $0$-solenoidal if and only if it is invariant under F$_\Q$-isotopy.}
then it extends to an invariant $\bar v$ of $\Q$-thickenable links of solenoids.
Moreover, $\bar v$ is invariant under F$_\Q$-isotopy.

(c) If $v$ is $0$-braidable, then it extends to an invariant of tame links of solenoids.

(d) If $v$ is a $0$-cableable invariant of link homotopy,%
\footnote{It is easy to see that every $0$-cableable invariant of link homotopy is actually $0$-solenoidal.}
then it extends to an invariant $\bar v$ of links of solenoids.
Moreover, $\bar v$ is invariant under F$_\Q$-isotopy and further extends to a link homotopy invariant 
$\bar{\bar v}$ of link maps of solenoids.
\end{theorem}

\begin{proof}[Proof. (b)] Given a $\Q$-thickenable link of solenoids $\Lambda$, let $\eps$ be given by Lemma \ref{approximation0} for $\Lambda$
and some $\Q$-thickening of $\Lambda$.
By Corollary \ref{taming} there exists a smooth link $L$ such that for some $r$ the composition $\Sigma\xr{\pi^\infty_r}mS^1\xr{L}S^3$ is 
$\eps$-close to $\Lambda$.
Let us define $\bar v$ by $\bar v(\Lambda)=v(L)$.

It remains to show that $\bar v$ is well-defined and is invariant under F$_\Q$-isotopy.
Given a $\Q$-thickenable link of solenoids $\Lambda'\:\Sigma'\to S^3$ that is F$_\Q$-isotopic to $\Lambda$, we have $\bar v(\Lambda')=v(L')$, where
$L'$ is a smooth link (given by Corollary \ref{taming}) such that for some $r'$ the composition $\Sigma'\xr{\pi^\infty_{r'}}mS^1\xr{L'}S^3$ is 
$\eps'$-close to $\Lambda'$, where $\eps'$ is given by Lemma \ref{approximation0} for $\Lambda'$ and some $\Q$-thickening of $\Lambda'$.
Then by Lemma \ref{approximation0} $L'$ is F$_\Q$-isotopic to $\Lambda'$, and similarly $L$ is F$_\Q$-isotopic to $\Lambda$.
Hence $L'$ is F$_\Q$-isotopic to $L$.
Since $v$ is $0$-solenoidal, by Proposition \ref{SFextension} $v(L')=v(L)$.
Hence $\bar v(\Lambda')=\bar v(\Lambda)$.
Thus $\bar v$ is well-defined (by considering the case $\Lambda'=\Lambda$) and moreover is invariant under F$_\Q$-isotopy.
\end{proof}

\begin{proof}[(a)] Similarly to (b), using Lemma \ref{approximation}(a) in place of Lemma \ref{approximation0}.
\end{proof}

\begin{proof}[(d)]
Let $\Lambda\:\Sigma\to S^3$ be a link map of solenoids.
Then there exists an $\eps>0$ such that every map $\Lambda'$, $7\eps$-close to $\Lambda$, is a link map.
By Lemma \ref{Mardesic}(a) there exists a smooth link $L$ such that for some $r$ the composition 
$F\:\Sigma\xr{\pi^\infty_r}mS^1\xr{L}S^3$ is $\eps$-close to $\Lambda$.
Let us define $\bar{\bar v}$ by $\bar{\bar v}(\Lambda)=v(L)$.

Let us show that $\bar{\bar v}$ is well-defined.
Let $L'\:mS^1\to S^3$ be another smooth link such that for some $r'$ the composition $F'\:\Sigma\xr{\pi^\infty_{r'}}mS^1\xr{L'}S^3$
is $\eps$-close to $\Lambda$.
Then $F'$ is $2\eps$-close to $F$.
By Lemma \ref{Mardesic}(b) there exists a $q\ge r,r'$ such that the compositions $f\:mS^1\xr{\pi^q_r}mS^1\xr{L}S^3$ and 
$f'\:mS^1\xr{\pi^q_{r'}}mS^1\xr{L'}S^3$ are $3\eps$-close.
Then there exists a $5\eps$-homotopy $L_t$ between embedded $\eps$-approximations $L_0$ of $f$ and $L_1$ of $f'$.
We may assume that $L_0$ is a cabling of $L$ and that $L_1$ is a cabling of $L'$.
Since $v$ is $0$-cableable, we have $v(L_0)=v(L)$ and $v(L_1)=v(L')$.
On the other hand, for each $t$ the composition $F_t\:\Sigma\xr{\pi^\infty_q}mS^1\xr{L_t}S^3$ is $5\eps$-close to $F_0$.
Since $L_0$ is $\eps$-close to $f$, $F_0$ is $\eps$-close to $F$, which is in turn $\eps$-close to $\Lambda$.
Hence $F_t$ is $7\eps$-close to $\Lambda$, and therefore is a link map.
Therefore each $L_t$ is a link map.
Since $v$ is a link homotopy invariant, $v(L_0)=v(L_1)$, whence $v(L)=v(L')$.

Let us show that $\bar{\bar v}$ is a link homotopy invariant.
Let $\Lambda_t\:\Sigma\to S^3$ be a link homotopy of solenoids.
Then there exists an $\eps>0$ such that every homotopy $\Lambda_t'$, $\eps$-close to $\Lambda_t$, is a link homotopy.
By Lemma \ref{Mardesic}(a) there exists a smooth homotopy $L_t$ such that for some $r$ the composition 
$F_t\:\Sigma\xr{\pi^\infty_r}mS^1\xr{L_t}S^3$ is $\eps$-close to $\Lambda_t$.
As long as $\eps$ is sufficiently small, $\bar{\bar v}(\Lambda_i)=v(L_i)$ for $i=0,1$.
On the other hand, since $F_t$ is $\eps$-close to $\Lambda_t$, it is a link homotopy.
Hence $L_t$ is also a link homotopy.
Therefore $v(L_0)=v(L_1)$.
Hence $\bar{\bar v}(\Lambda_0)=\bar{\bar v}(\Lambda_1)$.

It remains to show that the restriction $\bar v$ of $\bar{\bar v}$ to embedded links of solenoids 
is invariant under F$_\Q$-isotopy.
Let $\Lambda$ and $\Lambda'$ be F$_\Q$-isotopic links of solenoids.
It is immediate from the definition of F$_\Q$-isotopy that either $\Lambda$ and $\Lambda'$ are ambient isotopic
or both of them are $\Q$-thickenable.
In both cases $\bar v(\Lambda)=\bar v(\Lambda')$, using that $\bar{\bar v}$ is well-defined (or that it is 
a link homotopy invariant) in the first case, and using part (b) in the second case.
\end{proof}

\begin{proof}[(c)] To simplify matters, we confine our attention to the case of just one knotted solenoid which is
not a topological knot. 
The general case is treated similarly.

Let $\Sigma\not\cong S^1$ be the image of a tame knotted solenoid $\Lambda$.
Then it can be represented as $\bigcap_{i=1} T_i$, where $\dots\subset T_1\subset T_0$ are tame solid tori
such that the core knot of each $T_{i+1}$ is a $p_i$-braiding of the core knot of $T_i$ for some $p_i>1$.
Such a {\it defining sequence} $\dots\subset T_1\subset T_0$ for $\Sigma$ is called {\it maximal} if it cannot be refined
to a new defining sequence by inserting an additional solid torus between some consecutive pair $T_{i+1}\subset T_i$.
By \cite{JWZZ}*{Proposition 3.2} a maximal defining sequence always exists.
Two defining sequences $\dots\subset T_1\subset T_0$ and $\dots\subset T'_1\subset T'_0$ for $\Sigma$ are called {\it strongly equivalent} 
if there exist orientation preserving homeomorphisms $h_0\:(S^3,T_0)\to (S^3,T_0')$ and $h_i\:(T_{i-1},T_i)\to (T'_{i-1},T'_i)$
such that $h_i|_{\partial T_{i-1}}=h_{i-1}|_{\partial T_{i-1}}$ for $i\ge 1$; and {\it equivalent} if they become strongly equivalent 
upon omitting some initial segments $T_i\subset\dots\subset T_0$ and $T'_j\subset\dots\subset T'_0$.
By \cite{JWZZ}*{Theorem 3.4} every two maximal defining sequences for $\Sigma$ are equivalent.

Let $\dots\subset T_1\subset T_0$ be a maximal defining sequence for $\Sigma$ and let $K_i$ be the core knot of $T_i$,
regarded as an oriented knot via the embedding $\Lambda$ (by using that $T_i$ is a $\Q$-thickening of $\Lambda$).
We set $\bar v(\Lambda)=v(K_0)$.
Since $v$ is $0$-braidable, $v(K_i)=v(K_0)$ for each $i$.
Thus $\bar v$ is not affected by omitting any initial segment $T_{i-1}\subset\dots\subset T_0$.
On the other hand, it is easy to see that $\bar v$ is also preserved by strong equivalence of maximal defining systems.
Thus $\bar v$ is well-defined.
By construction it is an invariant of ambient isotopy.
\end{proof}

\begin{remark}
The author's expectation that more subtle invariants of linked solenoids than those given by Theorem \ref{extension}
can be obtained from (colored) finite type invariants of links by generalizing the method of \cite{M24-1} turned out to be wrong 
as the method of \cite{M24-1} (which does have a version for embeddings in higher dimensions, see \cite{MR1}*{Remarks (i)--(iii) in \S1.5}) 
now appears to be intrinsically limited to dealing with ANRs.
\end{remark}

\appendix

\section{Magnetic fields and link invariants} \label{magnetic}

The purpose of this Appendix is to explain the origin of Akhmetiev's Problem \ref{aa-problem}.
\smallskip

The magnetic field of a star or a planet can often be assumed to be ``frozen'' into the medium
(such as plasma or magma) in the sense that as the medium moves, the field moves along with it.
Mathematically, a {\it magnetic field} is a divergence-free vector field $\xi$ on a 
compact Riemannian 3-manifold $M$ which is tangent to $\partial M$.
Every volume-preserving diffeomorphism $h\:M\to M$ associates to $\xi$ another 
magnetic field $h_*\xi$ such that the flux of $\xi$ across any surface $S$ equals 
the flux of $h_*\xi$ across $h(S)$, and the flow $\phi^\xi$ is transformed by conjugation:
$\phi_t^{h_*\xi}=h\phi_t^\xi h^{-1}$.
Namely, $h_*$ is given by the adjoint action of the group of volume-preserving diffeomorphisms on its
Lie algebra, which consists precisely of divergence-free vector fields tangent to $\partial M$
(cf.\ \cite{ArKh}*{Chapter~I, \S\S2-3, 5, 7-8}, \cite{FrHe}*{Appendix A}).

The energy $E(\xi)=\frac12\int_M\big(\xi(x),\xi(x)\big)\ d\Vol_x$ of the magnetic field $\xi$ is not an invariant
of this action by volume-preserving diffeomorphisms.
(Speaking physically, the energy of a magnetic field can dissipate into heat.)
In fact, as speculated by  A.~Sakharov and Ya.~Zeldovich (the physicists) and checked by M.~Freedman 
(the mathematician), the energy of a magnetic field in the unit ball $B\subset\R^3$ whose 
integral curves are horizontal circles centered at points of the vertical axis can be made arbitrarily 
small by this action (see \cite{Ar1}*{\S1.3}, \cite{ArKh}*{\S III.3}).
However, it was shown by Freedman and He \cite{FrHe} that when $M\subset\R^3$, 
there is a positive lower bound for $E(h_*\xi)$ over all volume-preserving 
diffeomorphisms $h$ as long as there exists a non-trivial (up to ambient isotopy) smooth link $L$ 
such that $\xi$ leaves invariant some tubular neighborhood $N$ of $L$ and has a non-zero flux across
a meridional disk of each solid torus of $N$.
A slightly weaker result was obtained earlier by M. Freedman for any closed $3$-manifold $M$ \cite{Fr}.

On the other hand, by elaborating on C. F. Gauss' formula for the linking number:
\[\lk(K,Q)=\frac1{4\pi}\int_{S^1}\int_{S^1}\frac{\left<K'(x),Q'(y),K(x)-Q(y)\right>}{||K(x)-Q(y)||^3}\,dx\,dy,\]
H. Moffatt (see \cite{De}) and V. I. Arnold \cite{Ar1} found an invariant (up to volume preserving 
diffeomorphisms) of arbitrary magnetic fields (without assuming invariant solid tori) in a simply-connected $M\subset\R^3$, 
called the {\it helicity}:
\[H(\xi)=\frac1{4\pi}\int_{C_2(M)}\frac{\left<\xi(x),\xi(y),x-y\right>}{||x-y||^3}\,d\Vol_x\,d\Vol_y,\]
where $C_2(M)=\{(x,y)\in M\x M\mid x\ne y\}$, and proved that its absolute value provides a lower bound for $E(\xi)$.
The latter integral is actually proper when $M$ is compact, since the integrand extends to a continuous function on 
the standard completion $C_2[M]$ of $C_2(M)$ (also known as the Axelrod--Singer completion or the spherical 
Fulton--MacPherson completion), which is a compactification of $C_2(M)$ when $M$ itself is compact.
Upon this extension the helicity integral can be rewritten as
\[H(\xi)=\int_{C_2[M]}\alpha\wedge\alpha\wedge g^*\Vol_{S^2},\]
where $\alpha$ is the closed $2$-form dual to $\xi$ and $g\:C_2[M]\to S^2$ is the Gauss map, defined by
$g(x,y)=\frac{x-y}{||x-y||}$ for $(x,y)\in C_2(M)$ \cite{CP}.
This corresponds to rewriting the Gauss integral (cf.\ \cite{Ar1}*{\S4.2}, \cite{ArKh}*{\S III.4.B}) as
\[\lk(K,Q)=\int_{S^1\x S^1} (K\x Q)^*g^*\Vol_{S^2},\]
which computes the degree of the composition $S^1\x S^1\xr{K\x Q} C_2(M)\xr{g} S^2$.

\begin{problem}[Arnold--Moffatt%
\footnote{The first question of this type was apparently raised in print by Moffatt \cite{Mof}*{p.~367}.
In the words of Freedman \cite{Fr}, ``Arnol'd's invariant is a generalization of the homological linking 
number of two closed curves applied to the trajectories of [the vector field]. 
This has led Moffatt (1985) to conjecture that other `higher-order' linking (not detectable homologically) 
also leads to positive lower bounds on [the energy].''
It can be argued, however, that Freedman and He \cite{FrHe} proved Moffatt's conjecture without solving
Problem \ref{arnold-moffatt2} (see Remark \ref{FrHe} below). 
Other questions of this type, which come closer to the statement of Problem \ref{arnold-moffatt2},
were later raised by Arnold \cite{Ar-pr}*{Problem 1990-16} (compare \cite{ArKh}*{\S III.8.A}) and 
Arnold--Khesin \cite{ArKh}*{Remark III.7.18 and remark following the proof of Theorem I.9.9};
see also ``Problem B (Arnol'd's question)'' in \cite{De}.
For a further discussion see \cite{ArKh}*{Chapter III} and \cite{De}.}%
] \label{arnold-moffatt2}
Do there exist higher-order analogues of the helicity invariant which
\begin{enumerate}
\item are modeled on higher-order analogues of the linking number;
\item are invariants of magnetic fields up to volume-preserving diffeomorphisms; and 
\item provide lower bounds for the energy of the magnetic field?
\end{enumerate}
\end{problem}

One obvious approach to Problem \ref{arnold-moffatt2} is to seek a higher-order helicity in the form
\[I(\xi)=\int_{C_n(M)}F\big(y_1,\xi(y_1),\dots,y_n,\xi(y_n)\big)\,d\Vol_{y_1}\cdots d\Vol_{y_n},\tag{$*$}\]
where $C_n(M)=\{(y_1,\dots,y_n)\in M^n\mid y_i\ne y_j\text{ if }i\ne j\}$ and $F$ is some continuous function 
on the preimage $TC_n(M)$ of $C_n(M)$ in $(TM)^n$, in which case it would be natural to seek the associated 
higher-order linking number in the form
\[v(L)=\int_{C_n(N)}F\big(L(x_1),L'(x_1),\dots,L(x_n),L'(x_n)\big)\,dx_1\cdots dx_n,\tag{$**$}\]
where $N$ is a closed $1$-manifold and $L\:N\to M$ is a smooth link.
(Indeed the Gauss integral and the helicity are of this type.)
Naturally, the integrals ($*$) and ($**$) are expected to be proper when $M$ is compact as their integrands 
are expected to extend to $C_n[M]$ and $C_n[N]$ respectively, where $C_n[X]$ is an appropriate 
manifold-with-corners completion of $C_n(X)$ which is its compactification when the manifold $X$ is compact.
(The compactification $C_n[N]$ should certainly be equivalent to the standard one as $N$ is one-dimensional, 
but $C_n[M]$ might well involve additional blowups.)

\begin{proposition} {\rm (compare \cite{A13}*{proof of Theorem 6})} \label{finite-type} 
Every link invariant $v$ as in {\rm ($**$)} is a type $n$ invariant.
\end{proposition}

\begin{proof} For a singular link $L_\x$ with $r$ double points $z_1,\dots,z_r$ and 
its $2^r$ resolutions $L_\epsilon$, $\epsilon=(\epsilon_1,\dots,\epsilon_r)\in\{1,-1\}^r$,
we have \[v(L_\x)=\sum_\epsilon\epsilon_1\cdots\epsilon_rv(L_\epsilon)=
\int_{C_n(N)}\sum_\epsilon\epsilon_1\cdots\epsilon_rFL_\epsilon^*(x)\,dx,\]
where $x=(x_1,\dots,x_n)$ and $L_\epsilon^*(x)=\big(L_\epsilon(x_1),L_\epsilon'(x_1),\dots,L_\epsilon(x_n),L_\epsilon'(x_n)\big)$.
If $n<r$, then for each $x$ there is at least one $z_i$ that has none of the $x_j$'s near its two point-inverses in $N$.
Then changing the crossing at $z_i$ does not affect $L_\epsilon^*(x)$ at all; that is, $L_\epsilon^*(x)=L_{\epsilon^i}^*(x)$, 
where $\epsilon^i$ is obtained from $\epsilon$ by reversing $\epsilon_i$.
Thus for each $x$ we can partition the $2^r$ summands into canceling pairs, whence $v(L_\x)=0$.
\end{proof}

\begin{remark} Actually $v$ must be a type $\lfloor n/2\rfloor$ invariant as long as $F$ extends to a continuous map 
on a reasonable completion $TC_n[M]$ of $TC_n(M)$.
Indeed, in the notation of the proof of Proposition \ref{finite-type}, if $n<2r$, then for each $x$ there is 
at least one $z_i$ that has at most one of the $x_j$'s near its two point-inverses in $N$.
Then changing the crossing at $z_i$ will barely affect $FL_\epsilon^*(x)$; more precisely, we claim that 
$\sup_{x\in C_n(N)}|FL_\epsilon^*(x)-FL_{\epsilon^i}^*(x)|\to 0$ as $L_\epsilon\xr{C^1} L_\x$, where $i=i(x)$.
Given this, it follows that $v(L_\x)\to 0$ as $L_\epsilon\xr{C^1} L_\x$; but $v$ is a link invariant, 
so $v(L_\x)=0$.

To justify the claim, let $r=2\sup_{x\in N}||L_\x'(x)||$ and let $Q$ be a compact neighborhood of $L_\x(N)$.
Then $L_\epsilon^*\big(C_n(N)\big)$ lies in the subspace $D_rC_n(Q)$ of tangent $3$-disks of radius $r$ 
as long as $L_\epsilon$ is sufficiently $C^1$-close to $L_\x$.
Now the closure of $D_rC_n(Q)$ in $TC_n[M]$ must be compact, so the extension of $F$ over $TC_n[M]$ is uniformly 
continuous on it, and in particular $F$ itself is uniformly continuous on $D_rC_n(Q)$ with respect to 
the new metric induced from $TC_n[N]$ (and not the old metric of $(TM)^n$).
Since there is at most one of the $x_j$'s near the two point-inverses of $z_i$, $L_\epsilon^*(x)$ and 
$L_{\epsilon^i}^*(x)$ differ in at most two coordinates (namely, $L_\epsilon(x_j)$, $L_\epsilon'(x_j)$ vs. 
$L_{\epsilon^i}(x_j)$, $L_{\epsilon^i}'(x_j)$) and the distance between $L_\epsilon^*(x)$ and 
$L_{\epsilon^i}^*(x)$ tends to zero in both metrics as $L_\epsilon\xr{C^1} L_\x$; in fact, both 
$L_\epsilon^*(x)$ and $L_{\epsilon^i}^*(x)$ must converge to the same point in $TC_n[M]$.
(Let us note that when there are two of the $L_\epsilon(x_j)$'s near $z_i$, the distance can tend to zero
in the old metric but remain close to a constant in the new metric, as $L_\epsilon^*(x)$ and 
$L_{\epsilon^i}^*(x)$ may converge to two distinct points in $TC_n[M]$ in this case.)
\end{remark}

\begin{remark} \label{finite-type2}
It should be noted that seeking a higher-order helicity in the form ($*$) is not the only option. 
For instance, the integral might involve not only $\xi$, but also a vector field $\eta$ such that $\curl\eta=\xi$
or some other auxiliary object.
The proof of Proposition \ref{finite-type} might well extend to such generalizations of ($*$) and ($**$)
because all that it needs from $v$ is that it is an additive invariant which obtains information only by looking 
at finite configurations of points and is unable to process more than $n$ points at once.
Also the invariants of magnetic fields discussed in Remark \ref{asymptotic} below are not necessarily of the form ($*$), 
even though they too are related to finite type invariants of links.
\end{remark}

If a divergence-free vector field $\xi$ leaves invariant a solid torus $T$, then the flux of $\xi$ across any meridional disk 
of $T$ does not depend on the choice of the disk, and so can be denoted $\Flux(\xi|_T)$ as long as we agree to fix a generator of $H_1(T)$.
If moreover $\partial T$ is disjoint from the support $\Supp(\xi)$, then $\Flux(\xi|_T)$ is a discrete characteristic of $\xi$, 
in the sense that it depends only on the isotopy class $[T]$ of $(T,\partial T)$ in $\big(M,\,M\but\Supp(\xi)\big)$.
The set $E(\xi)$ of all isotopy classes of embeddings $S^1\x (D^2,\,\partial D^2)\to\big(M,\,M\but\Supp(\xi)\big)$
is an invariant of $\xi$ in the sense that diffeomorphisms of $M$ act on $E(\xi)$ by bijections;
and the map $\Phi(\xi)\:E(\xi)\to\R$ given by $[T]\mapsto\Flux(\xi|_T)$ is an invariant of $\xi$ up to volume-preserving diffeomorphisms.
The desired higher-order helicity $I$ may well depend on this invariant $\Phi$.
Since $I$ is also supposed to be modeled on a higher-order linking number $v$, which is most likely an invariant of
$m$-component links for some fixed $m$, it seems not unreasonable to assume that if $\Supp(\xi)$ lies 
in a tubular neighborhood $N=(T_1,\dots,T_m)$ of an $m$-component link $L=(K_1,\dots,K_m)$, then $I$ satisfies the equation
\[I(\xi)=v(L)\prod_{i=1}^m\big(\Flux(\xi|_{T_i})\big)^k\tag{${**}*$}\]
for some constant $k>0$, provided that $I$ vanishes on each $[\xi|_{T_i}]$, where $[\cdot]$ means $\cdot$ extended by zeroes over $M$.
Such an equation does hold for a suitably normalized version of the helicity \cite{ArKh}*{Lemma III.2.1}:
if $H([\xi|_{T_i}])=0$ for $i=1,2$, then
\[\tfrac12H(\xi)=\lk(L)\Flux(\xi|_{T_1})\Flux(\xi|_{T_2}).\]

Let us say that a vector field $\xi$ in a homologically $1$-connected $3$-manifold $M$ is {\it modeled} on a smooth link $L=(K_1,\dots,K_m)$ 
in $M$ if $\xi$ has its support in a tubular neighborhood of $L$ and every its nontrivial orbit $X$ is a knot which is a parallel pushoff of 
some $K_i$ (``parallel'' means that $\lk(X,K_i)=0$).
It is not hard to see that every magnetic field that is modeled on a knot has zero helicity.

\begin{proposition} \label{satellitable} Let $v$ be an invariant of $m$-component links in a homologically $1$-connected $3$-manifold $M$ and 
$I$ an invariant of magnetic fields in $M$ which is related to $v$ by {\rm (${**}*$)} and vanishes on every magnetic field 
that is modeled on a knot.
Then $v$ is $k$-satellitable.
\end{proposition}

\begin{proof} Given an $m$-component link $L$, integers $p_1,\dots,p_m$ and a $(p_1,\dots,p_m)$-satellite $L'$ of $L$,
we need to show that $v(L')=(p_1\cdots p_m)^kv(L)$.
Let $T=(T_1,\dots,T_m)$ be a tubular neighborhood of $L$ containing $L'$, let $T'=(T_1',\dots,T_m')$ be a tubular neighborhood of $L'$ 
contained in $T$, and let $\xi$ be a magnetic field with support in $T'$ which is modeled on $L'$.
Each $[\xi|_{T'_i}]$ is a magnetic field modeled on the knot $K_i'$, so $I$ vanishes on it.
Also $I$ vanishes on $[\xi|_{T_i}]$ because it is the same vector field.
Hence \[v(L')\prod_{i=1}^m\big(\Flux(\xi|_{T'_i})\big)^k=I(\xi)=v(L)\prod_{i=1}^m\big(\Flux(\xi|_{T_i})\big)^k.\]
On the other hand, $\Flux(\xi|_{T_i})=p_i\Flux(\xi|_{T'_i})$.
Thus $v(L')=(p_1\cdots p_m)^kv(L)$.
\end{proof}

Returning to Akhmetiev's Problem \ref{aa-problem}, it is concerned with those link invariants that are expected to yield 
higher-order helicities.
The reason why Akhmetiev needs finite type invariants should now be clear from Proposition \ref{finite-type} and Remark \ref{finite-type2}
(compare his own considerations \cite{A21}*{\S3 prior to \S3.1}).
The reason why he actually needs cableable invariant seems to be related to Proposition \ref{satellitable} 
(compare \cite{A14}*{Proof of Lemma 4.3}, \cite{A21}*{Proof of Theorem 14.2}, \cite{AS}*{part I, \S6}), which however arrives 
at satellitable, rather than cableable invariants.
The original source of cableable (or ``asymptotic'' in his terminology) invariants in his work is related the theory of
asymptotic link invariants (see Remark \ref{asymptotic} below), see \cite{A11}*{Remark on the condition (70)},
\cite{A14}*{Definition 4.1}, \cite{A15}*{Definition 1.1}, \cite{A16}*{\S2.1 and \S3.1}.

\begin{remark}
There are some no-go theorems which come close to solving Problem \ref{arnold-moffatt2} in the negative.
A. Enciso, D. Peralta-Salas and F. Torres de Lizaur \cite{EPT} proved that when $M$ is a homology $3$-sphere,
every invariant $I$ of divergence-free $C^1$ vector fields up to volume-preserving diffeomorphisms
is a function of the helicity, as long as its Fr\'echet derivative is an integral operator with continuous kernel.
The authors remark that the latter condition holds for all $I$ of the form ($*$) under some ``mild technical 
assumptions'' on $F$; on the other hand, they note that their proof does not work for $C^r$ vector fields 
where $r>1$.
The proof elaborates on a previous work by E. Kudryavtseva \cite{Ku}, who obtained a similar result 
(stated more technically) for certain $M$ with $\partial M\ne\emptyset$.
\end{remark}

\begin{remark}\label{FrHe}
The results of Freedman and He \cite{FrHe} come close to a positive solution,
even an ultimate one, of Problem \ref{arnold-moffatt2}.
They discovered a variation $c$ of the helicity which roughly speaking counts the number of those crossings 
in a link of trajectories whose algebraic number is counted by the helicity.
For an arbitrary magnetic field $\xi$ in an $M\subset\R^3$, $c(\xi)$ provides a lower bound for $E(\xi)$, 
and although $c$ itself is not an invariant, an invariant $c_{\rm top}$ is defined by minimizing it
over all volume-preserving diffeomorphisms applied to $\xi$.
If $\xi$ leaves invariant a tubular neighborhood $N=(T_1,\dots,T_m)$ of a link $L=(K_1,\dots,K_m)$, 
then $c_{\rm top}(\xi)$ is bounded below by 
\[\Big(\sum_{i=1}^m ac_i(L)\big|\Flux(\xi|_{T_i})\big|\Big)\min_{1\le i\le m}\big|\Flux(\xi|_{T_i})\big|,\]
where $ac_i(L)\ge 1$ unless $K_i$ is an unknot lying in a $3$-ball disjoint from $L\but K_i$.
By definition, $ac_i(L)=\inf_{p,q\ge 1}\inf_{(P,Q)}\frac{1}{pq} c(P,Q)$, where $(P,Q)$ runs over all 
$(m+1)$-component links consisting of a $p$-satellite $P$ of $K_i$ and a $(q,\dots,q)$-satellite $Q$ of $L$ 
and $c(P,Q)$ is the number of times that $P$ passes over $Q$ minimized over all planar diagrams of $(P,Q)$.
In fact, $ac_i(L)\ge\sum_{j\ne i}|\lk(K_i,K_j)|$.
\end{remark}

\begin{remark}\label{asymptotic}
Another no-go theorem is due to S. Baader and J. March\'e \cite{BaMa}.
For a magnetic field $\xi$ in an $M\subset\R^3$ let $\phi$ be the flow of $\xi$ and for a point $x\in M$
let $K_T(x)=\{\phi_t(x)\mid t\in [0,T]\}\cup [x,\phi_T(x)]$, where $[x,\phi_T(x)]$ denotes the straight line segment in $\R^3$.
Under certain hypotheses on $\xi$ (see \cite{BaMa}), $K_T(x)$ is a knot for almost all $x\in M$ and $T>0$.
But regardless of whether it actually happens to be a knot often or at all, one may try using it to compute invariants of
magnetic fields.
In particular, T.~Vogel \cite{Vo}, elaborating on Arnold's arguments \cite{Ar1} (see also \cite{ArKh}) proved that when $M$ is 
simply-connected and $\xi$ is an arbitrary magnetic field, $\lk_\xi(x,y):=\lim_{S,T\to\infty}\frac{1}{ST}\lk\big(K_S(x),K_T(y)\big)$ 
exists if understood appropriately and $\int_M\int_M\lk_\xi(x,y)\,d\Vol_x\,d\Vol_y$ equals the helicity $H(\xi)$.

On the other hand, R.~Komendarczyk and I.~Voli\'c \cite{KoVo} study for a given knot invariant $v$
the expression $v^{(k)}(\xi):=\lim_{T\to\infty}\int_M\frac{1}{T^k}v\big(K_T(x)\big)\,d\Vol_x$, where
$\xi$ is assumed to be nonvanishing.
They show, using configuration space integrals, that if $v$ is a type $n$ invariant, then $v^{(2n)}(\xi)$ exists 
and is invariant under volume-preserving diffeomorphisms isotopic to the identity.
Moreover, if $v^{(i)}(\xi)$ vanish for $k<i\le 2n$, then under certain hypotheses $v^{(k)}(\xi)$ also exists.
They also obtain a version of the Baader--March\'e theorem: if $\xi$ is ergodic, then $v^{(2n)}(\xi)=c_v\big(H(\xi)\big)^n$
for some (possibly zero) constant $c_v$.

The original Baader--March\'e theorem deals with $v^{[k]}(\xi,x):=\lim_{T\to\infty}\frac{1}{T^k}v\big(K_T(x)\big)$
and says that under slightly different assumptions on $\xi$, if $v$ is a type $n$ invariant and $\xi$ is ergodic, then 
$v^{[2n]}$ exists for almost all $x\in M$ and equals $c_v\big(H(\xi)\big)^n$ \cite{BaMa}.
Their proof is shorter and more elementary; it is based on the Polyak--Viro theory of arrow diagrams.
It is clear from their proof that it extends to finite type invariants of links.
Also their constant $c_v$ can be described explicitly,%
\footnote{The constant $\alpha_\Gamma$ in \cite{BaMa}*{Lemma 1} is easily seen to be the probability of obtaining 
the diagram $\Gamma$ upon drawing $n$ random oriented chords in the unit disk.}
and it follows from this description that $c_v\ne 0$ if $v$ can be represented, up to invariants of type $n-1$, by 
a {\it positive} linear combination of arrow diagrams.
A certain secondary invariant $v^{[k]}$, $k<2n$, is studied in \cite{A20}.
\end{remark}

\section*{Acknowledgements}

I would like to thank P. M. Akhmetiev for introducing me to cableable (or asymptotic, in his terminology) invariants and 
for stimulating discussions over the years.
I'm also grateful to I.~Alexeev, F. Ancel, M. Il'insky, S. Podkorytov, E. V. Shchepin, A. Yasuhara and my wife E. Melikhova for stimulating
discussions and useful remarks.
Many thanks to the two referees (one human and one AI) whose careful reading helped to eliminate a number of inaccuracies and to 
the editors of the volume --- I. Alexeev, V. Ionin and A. V. Malutin, who recruited these referees.
(Apart from the referees' remarks that I got from the journal, I have not used AI to write this paper.)
I am also grateful to A.~V.~Malutin for putting a lot of effort into improving the wording of some remarks in the introduction
(the ones that explain relations with the work of P. M. Akhmetiev).

\end{document}